\documentclass[11pt]{amsart}
\usepackage[colorlinks]{hyperref}
\usepackage{amsmath,amsthm,amssymb,color}
\usepackage{pdfsync}
\usepackage[latin1]{inputenc}
\usepackage{graphicx}
\usepackage{pstricks,epic,eepic}
\usepackage{esint} 
\usepackage{xcolor} 
\usepackage{enumitem} 
\usepackage{geometry,}
\usepackage{marginnote}
\hypersetup{ 
	colorlinks,
	linkcolor={red!50!black},
	citecolor={blue!50!black},
	urlcolor={blue!80!black}
}

\usepackage{pgfplots}
\usetikzlibrary{arrows}

\newtheorem{theorem}{Theorem}[section]
\newtheorem*{theorem*}{Theorem}
\newtheorem*{lemma*}{Lemma}
\newtheorem*{theorema*}{Theorem A}
\newtheorem*{theoremb*}{Theorem B}
\newtheorem*{theoremc*}{Theorem C}
\newtheorem*{mainlemma*}{Main Lemma}

\newtheorem{lemma}[theorem]{Lemma}

\newtheorem{corollary}[theorem]{Corollary}
\newtheorem{proposition}[theorem]{Proposition}
\newtheorem{claim}{Claim}
\newtheorem*{claim*}{Claim}
\newtheorem*{conjecture*}{Conjecture}
\theoremstyle{definition}
\newtheorem{definition}[theorem]{\bf Definition}

\theoremstyle{remark}
\newtheorem{remark}[theorem]{\bf Remark}

\numberwithin{equation}{section}

\newcommand{\vv}{\vspace{2mm}}
\newcommand{\vvv}{\vspace{4mm}}

\newcommand{\supp}{\operatorname{supp}}
\newcommand{\dist}{{d}}
\newcommand{\pom}{{\partial\Omega}}
\newcommand{\diam}{\operatorname{diam}}
\newcommand{\R}{\mathbb R}
\newcommand{\HH}{\mathcal H}
\newcommand{\wt}{\widetilde}
\newcommand{\haj}{Haj{\l}asz}
\begin{document}

\title
[Extrapolation of solvability]
{Extrapolation of solvability of the regularity and the Poisson regularity problems in rough domains}

\author[Josep M. Gallegos]{Josep M. Gallegos}
\address{Departament de Matem\`atiques\\ Universitat Aut\`onoma de Barcelona
 \\ Edifici C Facultat de Ci\`encies\\08193 Bellaterra 
}
\email{jgallegosmath@gmail.com}
\thanks{M.M. was supported by IKERBASQUE and partially supported by the grant PID2020-118986GB-I00 of the Ministerio de Econom\'ia y Competitividad (Spain) and by the grant IT-1615-22 (Basque Government). X.T. and J.G. were supported by the European Research Council (ERC) under the European Union's Horizon 2020 research and innovation programme (grant agreement 101018680) and partially supported
by MICINN (Spain) under grant PID2020-114167GB-I00, and 2021-SGR-00071 (Catalonia). X.T. was also partly supported by the Mar\'ia de Maeztu Program for units of excellence (Spain) (CEX2020-001084-M)}

\author[Mihalis Mourgoglou]{Mihalis Mourgoglou}
\address{Departamento de Matem\'aticas, Universidad del Pa\' is Vasco, UPV/EHU, Barrio Sarriena s/n 48940 Leioa, Spain and\\
IKERBASQUE, Basque Foundation for Science, Bilbao, Spain.}
\email{michail.mourgoglou@ehu.eus}
\author[Xavier Tolsa]{Xavier Tolsa}
\address{ICREA, Barcelona\\
Dept. de Matem\`atiques, Universitat Aut\`onoma de Barcelona \\
and Centre de Recerca Matem\`atica, Barcelona, Catalonia.}
\email{xtolsa@mat.uab.cat}

\begin{abstract}
    Let $\Omega\subset \R^{n+1}$, $n\geq2$, 
    be an open set satisfying the corkscrew condition with {$n$-Ahlfors regular boundary} $\pom$, but without any connectivity assumption. We study the connection between solvability of the regularity problem for {divergence form elliptic operators} with boundary data in the \haj-Sobolev space $M^{1,1}(\partial\Omega)$ and the {weak-$\mathcal A_\infty$} property of the {associated elliptic} measure.
	In particular, we show that solvability of the regularity problem in {$M^{1,1}(\pom)$ is equivalent to the solvability of the regularity problem in $M^{1,p}(\pom)$ for some $p>1$}. 
	{We also prove analogous extrapolation results for the Poisson regularity problem defined on tent spaces.}
	Moreover, under the hypothesis that $\pom$ supports a weak $(1,1)$-Poincar\'e inequality, we show that the solvability of the regularity problem in the \haj-Sobolev space $M^{1,1}(\pom)$ is equivalent to a stronger solvability	in a Hardy-Sobolev space of tangential derivatives. 
\end{abstract}

\maketitle

\section{Introduction}

In this paper we study the solvability of the regularity problem for {divergence form elliptic operators} in corkscrew domains $\Omega$ with {$n$-Ahlfors regular} boundary $\pom$ in the \haj-Sobolev space $M^{1,1}(\pom)$. 
In the recent work \cite{MT}, the second and third named authors proved the equivalence between the solvability of the Dirichlet problem in $L^{p'}(\pom)$ and the solvability of the regularity problem in $M^{1,p}(\pom)$ for the Laplacian in this type of domains (where $p,p'>1$ are H\"older conjugate exponents). This result opened up the study of the regularity problem for such non-smooth domains, which had only been investigated in Lipschitz domains for classical Sobolev spaces before. 
Here we focus on the endpoint case $p=1$, not considered in \cite{MT}.
The analogous result for Lipschitz domains and classical Sobolev spaces obtained in the pioneering work \cite{DK1} of Dahlberg and Kenig was that solvability in a Hardy-Sobolev space $HS^{1,1}(\pom)$ is equivalent to solvability in a classical Sobolev space for some $p>1$. {In the present setting,} it turns out that the \haj-Sobolev space $M^{1,1}(\pom)$ is the adequate space to consider the endpoint case $p=1$. {This is not completely unexpected since $\dot M^{1,1}(\R^{n+1}) = HS^{1,1}(\R^{n+1})$ or, more generally, $\dot M^{1,1}(\wt M) = HS^{1,1}(\wt M)$ when $\wt M$ is a manifold satisfying a doubling condition and supporting a strong Poincar\'e inequality (see the work of Koskela and Saksman \cite{KS} and Badr and Dafni \cite{BD1} for these results).}\color{black}

\vv

\subsection{Definitions}\label{section:definitions}
We introduce some definitions and notations. 
A set $E\subset \R^{n+1}$ is called $n$-{\textit {rectifiable}} if there are Lipschitz maps
$f_i:\R^n\to\R^{n+1}$, $i=1,2,\ldots$, such that 
$$
\HH^n\biggl(E\setminus\bigcup_i f_i(\R^n)\biggr) = 0,
$$
where $\HH^n$ stands for the $n$-dimensional Hausdorff measure. We will assume $\HH^n$ to be normalized so that it coincides with $n$-dimensional Lebesgue measure in
$\R^n$.

All measures in this paper are assumed to be Radon measures.
A measure $\mu$ in $\R^{n+1}$ is called 
 $n$-{\textit {Ahlfors}} {\textit {regular}}  if there exists some
constant $C_{0}>0$ such that
$$C_0^{-1}r^n\leq \mu(B(x,r))\leq C_0\,r^n\quad \mbox{ for all $x\in
\supp\mu$ and $0<r\leq \diam(\supp\mu)$.}$$

The measure $\mu$ is \textit{uniformly  $n$-rectifiable} if it is 
$n$-Ahlfors regular and
there exist constants $\theta, M >0$ such that for all $x \in \supp\mu$ and all $0<r\leq \diam(\supp\mu)$ 
there is a Lipschitz mapping $g$ from the ball $B_n(0,r)$ in $\R^{n}$ to $\R^{n+1}$ with $|g|_{\operatorname{Lip}} \leq M$ such that
\[
\mu (B(x,r)\cap g(B_{n}(0,r)))\geq \theta r^{n}
\]
where $|g|_{\operatorname{Lip}}$ is the Lipschitz seminorm of $g$.

A set $E\subset \R^{n+1}$ is $n$-Ahlfors regular if $\HH^n|_E$ is $n$-Ahlfors regular.
 Also, 
$E$ is uniformly $n$-rectifiable if $\HH^n|_E$ is uniformly $n$-rectifiable.

The notion of uniform rectifiability should be considered a quantitative version of rectifiability. It was introduced in the pioneering works \cite{DS1} and \cite{DS2} of David and Semmes, who were seeking a good geometric framework under which all singular integrals with odd and sufficiently smooth kernels are bounded in $L^2$.

\vv

Following \cite{JK1}, we say that $\Omega \subset \R^{n+1}$ satisfies the \textit{corkscrew condition}, or that it is a corkscrew open set or domain if
there exists some $c>0$ such that for all
$x\in\pom$ and all $r\in(0, 2\diam(\Omega))$ there exists a ball $B\subset B(x,r)\cap\Omega$ such that
$r(B)\geq c\,r$. We say that $\Omega$ is \textit{two-sided corkscrew} if both $\Omega$ and $\R^{n+1}\setminus
\overline\Omega$ satisfy the corkscrew condition.

If $X$ is a metric space, given an interval $I\subset\R$, any continuous $\gamma:I \to X$ is called {\it path}. A path of finite length is called {\it rectifiable path}.

Following \cite{HMT}, we say that the open set $\Omega  \subset \R^{n+1}$ satisfies the {\it  local John condition} if there is $\theta\in(0,1)$
such that the following holds: For all $x\in\pom$ and $r\in(0,2\diam(\Omega))$ 
there is $y\in B(x,r)\cap \Omega$ such that $B(y,\theta r)\subset \Omega$ with the property that for all
$z\in B(x,r)\cap\pom$ one can find a rectifiable {path} $\gamma_z:[0,1]\to\overline \Omega$ with length
at most $\theta^{-1}|x-y|$ such that 
$$\gamma_z(0)=z,\qquad \gamma_z(1)=y,\qquad \dist(\gamma_z(t),\pom) \geq \theta\,|\gamma_z(t)-z|
\quad\mbox{for all $t\in [0,1].$}$$
If both $\Omega$ and $\R^{n+1}\setminus\Omega$ satisfy the local John condition, we say that $\Omega$ satisfies the {\it two-sided local John condition}. Note that the local John condition implies the corkscrew condition.

An open set $\Omega \subset \mathbb{R}^{n+1}$ is said to satisfy the \textit{weak local John condition} if there are $\lambda, \theta \in(0,1)$ and $\Lambda \geq 2$ such that the following holds: For every $x \in \Omega$ there is a Borel subset $F \subset B(x, \Lambda \dist(x, \partial \Omega)) \cap \partial \Omega)$ with $\sigma(F) \geq \lambda \sigma(B(x, \Lambda \dist(x, \partial \Omega)) \cap \partial \Omega)$ such that for every $z \in F$ one can find a rectifiable path $\gamma_z:[0,1] \rightarrow \bar{\Omega}$ with length at most $\theta^{-1}|x-z|$ such that
\[
\quad \gamma_z(0)=z, \quad \gamma_z(1)=x, \quad \dist\left(\gamma_z(t), \partial \Omega\right) \geq \theta\left|\gamma_z(t)-z\right| \quad\mbox{ for all }t \in[0,1].
\]

\vv 
Given an $n$-Ahlfors regular set $E \subset \R^{n+1}$, and a surface ball $\Delta_0 = B_0\cap E$, we say that a Borel measure $\mu$ defined on $E$ belongs to \textit{weak-$\mathcal A_\infty(\Delta_0)$} if there are positive constants $C$ and $\theta$ such that for each surface ball $\Delta = B(x,r)\cap E$ with $B(x,2r)\subseteq B_0$, we have
\begin{equation}\label{eq:def_weak_A_infty}
    \mu(F) \leq C \left ( \frac{\mathcal H^n(F)}{\mathcal H^n(\Delta)}\right)^\theta \mu(2\Delta), \quad \mbox{for every Borel set $F\subset \Delta$.}
\end{equation}

\vv

Let us now turn our attention to Sobolev spaces. Let $X$ be a metric space equipped with a doubling measure $\sigma$, that is, there exists a uniform constant $C_\sigma \geq 1$ such that $\sigma(B(x,2r)) \leq C_\sigma \sigma(B(x,r))$, for all $x\in X$ and $r>0$.
For a function $f:X\to \R$, we say that a non-negative function $\nabla_H f:X\to \R$ is a \textit{\haj~gradient} of $f$ if
\begin{equation}
	\label{eq:def_haj_grad}
|f(x)-f(y)|\leq |x-y| \left(\nabla_H f(x) + \nabla_H f(y)\right) \quad\mbox{for $\sigma$-a.e. $x,y\in X$,}
\end{equation}
and we denote the collection of all \haj~gradients of $f$ by $D(f)$. {Note that the \haj~gradient of a function is not uniquely defined. }\color{black}
For $p\in(0,+\infty]$, we define the \textit{homogeneous \haj-Sobolev} space $\dot M^{1,p}(X)$ as the space of functions $f$ that have a \haj~ gradient $\nabla_H f \in D(f)\cap L^p(X)$, and the \textit{inhomogeneous \haj-Sobolev} space $M^{1,p}(X)$ as $M^{1,p}(X) := \dot M^{1,p}(X) \cap L^p(X)$. These spaces were introduced by \haj~ in \cite{H}, and for more information on these and other Sobolev spaces in metric measure spaces, the reader may consult the book \cite{HKST}.
For $\tau \in (0, +\infty]$, we define the following weighted norm (or quasinorm if $p<1$) 
\begin{equation}
\Vert f \Vert_{M^{1,p}_\tau(X)} := \begin{cases} 
\frac{1}{\tau}\Vert f \Vert_{L^p(X)} + \inf_{g \in D(f)} \Vert g \Vert_{L^p(X)}, &\tau \in (0, +\infty), \\
\inf_{g \in D(f)} \Vert g \Vert_{L^p(X)}, &\tau=+\infty,
\end{cases}
\end{equation}
(seminorm in the case $\tau = +\infty$).
The most interesting cases are: $\tau = +\infty$, or $\tau = \diam X$ in the case $X$ has finite diameter. For this reason, we will use the notation 
\[
\Vert f \Vert_{M^{1,p}(X)} := \begin{cases}
\Vert f \Vert_{M^{1,p}_{\diam X}(X)}, & \diam X<+\infty,\\
\Vert f \Vert_{M^{1,p}_1(X)}, &\diam X = +\infty,
\end{cases}
\]
and 
\[
\Vert f \Vert_{\dot M^{1,p}(X)}:=\Vert f \Vert_{M_{\infty}^{1,p}(X)}.
\]
Note that the election $\tau = \diam X$ in the case $X$ has finite diameter makes the norm scale-invariant.
Regarding the case $\tau = +\infty$, we will use the convention that $\frac{1}{+\infty} = 0$, and $ +\infty \cdot a = +\infty$ for all $a>0$.

Let $f: X \rightarrow \mathbb{R}$ be an arbitrary real-valued function. We say that a Borel measurable function $g: X \rightarrow \mathbb{R}$ is an \textit{upper gradient} of $f$ if for all compact rectifiable paths $\gamma$ the following inequality holds:
\[
|f(x)-f(y)| \leq \int_\gamma g \,d \mathcal{H}^1,
\]
where $x, y \in X$ are the endpoints of the path. We say that $(X,\sigma)$ supports a {\it weak} {$(1,p)$-{\it Poincar\'e inequality} for $p \geq 1$},  if there exist constants $C\geq 1$ and $\Lambda \geq 1$ so that for every ball $B$ centered at $\Sigma$ with radius $r(B) \in (0, \diam \Sigma)$ and  every  pair $(f,g)$, where $f \in L^1_{loc}(\sigma)$ and $g$ is an upper gradient of $f$, it holds 
\begin{equation}
\label{eq:Poincare}
 \fint_{B} \left| f(x) - f_B\right| \,d\sigma(x) \leq C r(B) \left( \fint_{\Lambda B} |g(x)|^p\,d\sigma(x)\right)^{1/p},
\end{equation}
where we denoted 
\[
f_B= \fint_B f(y) \,d\sigma(y).
\]
{We say that $(X,\sigma)$ supports a \textit{strong Poincar\'e inequality} if it supports a weak $(1,1)$-Poincar\'e inequality with $\Lambda = 1$.}

\vv

We introduce Hardy-Sobolev spaces defined on a {$d$-rectifiable set} $E \subset \mathbb R^{n+1}$.
We say that a \color{black}{Lipschitz} \color{black} function $a(x)$ is a \textit{homogeneous Hardy-Sobolev} \textit{atom} if
\begin{enumerate}
    \item[(i)] $\supp a \subset B\cap E$ for some ball $B$ centered on $E$,
    \item[(ii)] $\Vert \nabla_t a \Vert_{L^\infty(E)} \leq \mathcal H^d(B\cap E)^{- 1}$,
\end{enumerate}
where $\nabla_t a$ is the tangential gradient of $a$ on $\pom$ which exists $\mathcal H^d$-a.e.$\,x\in E$ if $E$ is $n$-rectifiable (see \cite[Theorem 11.4]{M} for example). 
We define the \textit{(atomic) Hardy-Sobolev space} $HS^{1,1}(E)$ as the subspace of functions $f\in L^1(E)$ which admit an \textit{atomic decomposition}, that is, there exists a family of Hardy-Sobolev atoms $(a_j)_j$ and a sequence of real numbers $(\lambda_j)_j$ with the property that
\[
f = \sum_j \lambda_j a_j \quad \mbox{in $L^1(E)$}
\]
and $\sum_j |\lambda_j| < +\infty$. We equip $HS^{1,1}(E)$ with the {seminorm} $\Vert f \Vert_{HS^{1,1}} = \inf \sum_j |\lambda_j|$ where the infimum is taken over all possible atomic decompositions. 

\vv

Let $\Omega\subset\R^{n+1}$ be an open set and  set $\sigma:=\HH^n|_{\partial\Omega}$ to be its surface measure. For $\alpha>0$ and $x \in \pom$, we define the {\it cone with vertex $x$ and aperture $\alpha>0$} by
\begin{equation}\label{eqconealpha}
\gamma_\alpha(x)=\{ y \in \Omega: |x-y|<(1+\alpha)\dist(y, \pom)\},
\end{equation}
the \textit{$R$-truncated cone with vertex $x$ and aperture $\alpha$} by $\gamma_{\alpha,R}(x) = \gamma_\alpha(x) \cap B(x,R)$, and the {\it non-tangential maximal operator} of a measurable function $u:\Omega \to \R$ by
\begin{equation}\label{eqNalpha}
N_\alpha(u)(x):=\sup_{y \in \gamma_\alpha (x)} |u(y)|, \,\,x\in \pom.
\end{equation}
We also may define the \textit{truncated non-tangential maximal operator} $N_{\alpha,R}$ analogously.
 If $\pom$ is Ahlfors regular, then $\| N_\alpha(u)\|_{L^p(\sigma)} \approx_{\alpha,\beta} \| N_\beta(u)\|_{L^p(\sigma)}$ {for all $\alpha,\beta>0$} and so, from now on, we will only write $N$ dropping the dependence on the aperture (see \cite[Proposition 2.2]{HMT} for a proof of this fact). 
 Following \cite{KP}, we also introduce the \textit{modified non-tangential maximal function} for any function $u$ in $L^2_{\operatorname{loc}}(\pom)$ as 
 \begin{equation}
 \label{eq:def_modified_nontangential_op}
 \wt N_{\alpha}  u(x):= \sup_{y \in \gamma_\alpha(x)} \left (  \fint_{B\left(y, \frac 1 4 \dist(x,\pom)\right)} |u(z)|^2\, dz \right)^{1/2}, \quad x \in \pom,
 \end{equation}
and, similarly, the \textit{truncated modified non-tangential maximal function} $\wt N_{\alpha,R}$. Again, we will omit the dependence on the aperture.
 \color{black}
 We always consider divergence form elliptic operators $\mathcal L:= \operatorname{div} A\nabla$ where $A$ is always a real valued, not necessarily symmetric $(n+1)\times (n+1)$ matrix verifying the strong ellipticity conditions
 \begin{equation}
 	\label{eq:ellipticity_conditions_operator}
 	\lambda |\xi|^2 \leq \sum_{i,j=1}^{n+1} A_{ij}(x) \xi_i \xi_j, \quad \Vert A \Vert_{L^\infty(\Omega)}\leq \frac 1 \lambda, \quad x\in \Omega, \quad \xi \in \R^{n+1} 
 \end{equation}
 for some $\lambda \in (0,1)$, the \textit{ellipticity constant} of $\mathcal L$. We will denote by $\mathcal L^*$ the operator associated to the transpose matrix $A^*$. 
 \color{black}

\vv
We are now ready to state the definitions of solvability of the $L^p$-Dirichlet and the $M^{1,p}$-regularity problem in~$\Omega$ for $p \in (0, +\infty)$:
\begin{itemize}
\item In a domain $\Omega  \subset \R^{n+1}$, we say that {\it the Dirichlet problem is solvable in $L^p$} for the {operator $\mathcal L$}  (we write $(D^{\mathcal L}_{p})$ is solvable) if there exists some constant $C_{D_p}>0$  such that, for any $f\in C_c(\pom)$, the solution $u:\Omega\to\R$ of the continuous
Dirichlet problem for $\mathcal L$ in $\Omega$ with boundary data $f$ satisfies
\begin{equation}\label{eq:main-est-Dirichlet}
	\| N(u)\|_{L^p(\sigma)} \leq C_{D_p}\|f\|_{L^p(\sigma)}.
\end{equation}

\item In a domain $\Omega \subset \R^{n+1}$, we say that {\it the regularity problem is solvable in $M^{1,p}$} for the {operator $\mathcal L$}  (we write $(R^{\mathcal L}_{p})$ is solvable) if  there exists some constant $C_{R_p}>0$  such that, for any  compactly supported Lipschitz function
$f:\pom\to\R$, the solution $u:\Omega\to\R$ of the continuous
Dirichlet problem for $\mathcal L$ in $\Omega$ with boundary data $f$ satisfies
\begin{equation}\label{eq:main-est-reg}
\| \wt N(\nabla u)\|_{L^p(\sigma)} \leq \begin{cases} C_{R_p}\|f\|_{\dot M^{1,p}(\sigma)}, \quad \mbox{if $\Omega$ is bounded or $\pom$ is unbounded} \\
	C_{R_p}\|f\|_{M^{1,p}(\sigma)}, \quad \mbox{if $\Omega$ is unbounded and $\pom$ is bounded}.
 \end{cases}
\end{equation}
Here we denoted $M^{1,p}(\sigma)$ and $\dot M^{1,p}(\sigma)$
to be the Haj\l asz  Sobolev spaces defined on the metric measure space $(\pom, \sigma)$. 
\end{itemize}
One may also define the \textit{tangential regularity problem} in domains with rectifiable boundaries using the tangential derivative with respect to the boundary at a boundary point, which for Lipschitz functions exists for $\sigma$-a.e.\ point of the boundary. 
In that case the only difference in the definition of the regularity problem is that we ask for the estimate  
\begin{equation}
\label{eq:main-est-reg2}
\|\wt N(\nabla u)\|_{L^p(\sigma)} \leq
	\begin{cases} \widetilde{C}_{R_p}\| \nabla_t f\|_{ L^{p}(\sigma)}, \quad &\mbox{if $\Omega$ is bounded or $\pom$ is unbounded}\\
	\widetilde{C}_{R_p}\left ( \| \nabla_t f\|_{ L^{p}(\sigma)} + \frac{\Vert f\Vert_{L^p(\sigma)}}{\diam(\pom)} \right), \quad &\mbox{if $\Omega$ is unbounded and $\pom$ is bounded}
	\end{cases}
\end{equation}
and we write that $(\widetilde R^{\mathcal L}_p)$ is solvable ($p>1$). 
These definitions are customary in nicer domains such as Lipschitz domains {and are stronger than solvability of $(R^{\mathcal L}_p)$ in general domains. {In \cite[Section 10]{MT} there is an example of a chord-arc domain where $(\wt R_p^\Delta)$ is not solvable for any $p\geq1$ but $(R^\Delta_p)$ is solvable for some $p>1$}}. 
\begin{remark}
For $p=1$, the tangential regularity problem $(\wt R^{\mathcal L}_1)$ is considered in Hardy-Sobolev spaces, that is, we require the estimate
\[
\Vert \wt N(\nabla u) \Vert_{L^1(\sigma)} \leq 
	\begin{cases} 
		\widetilde C_{R_1} \Vert f \Vert_{HS^{1,1}(\sigma)}, \quad &\mbox{if $\Omega$ is bounded or $\pom$ is unbounded}\\
	\widetilde{C}_{R_1}\left ( \Vert f \Vert_{HS^{1,1}(\sigma)} + \frac{\Vert f\Vert_{L^1(\sigma)}}{\diam(\pom)} \right), \quad &\mbox{if $\Omega$ is unbounded and $\pom$ is bounded}.
\end{cases}
\]
It is known that \eqref{eq:main-est-reg2} with $p=1$ cannot hold even in the {unit ball} in the case of the Laplacian. {In fact, for smooth domains, we also have $\Vert \wt N(\nabla u) \Vert_{L^1(\pom)} \gtrsim \Vert f \Vert_{\dot M^{1,1}(\pom} \approx \Vert f \Vert_{HS^{1,1}(\pom)}$ (see Proposition \ref{prop:reverseregularity} for the $\gtrsim$ and Proposition \ref{prop:equiv_cald_maximal_hardy_space} for the $\approx$) which is not comparable to $\Vert \nabla_t f \Vert_{L^1(\pom)}$ in general}.
\end{remark}
\begin{remark}
	\label{rmk:modified_nonmodified}
	 {If the operator $\mathcal L$ satisfies for instance the scale invariant locally Lipschitz condition $\sup_{x \in \Omega}\delta_\Omega(x) \| \nabla A \|_{L^\infty(B(x,  c\delta_\Omega(x))}<\infty$,  for some $c \in (0, \tfrac{1}{2})$,  and $u$ is a solution of $\mathcal Lu=0$  in $\Omega$}, then  $\Vert \wt N(\nabla u)\Vert_{L^p(\pom)}$ and $\Vert N(\nabla u) \Vert_{L^p(\pom)}$ are comparable. 
	 For  general elliptic operators without additional regularity,   this is not true as the gradient $\nabla u$ of a solution is  defined in {the $L_\textup{loc}^2(\Omega)$ sense  but not pointwisely}. 
\end{remark}

\vv

In the sequel, we will assume that  $n \geq 2$. 

\vv

\subsection{Main results}

Let us now state our main results. {We remark that the following results are valid (unless specified otherwise) for bounded domains, unbounded domains with compact boundary, and unbounded domains with unbounded boundary.}
\begin{theorem}[Extrapolation of solvability of the regularity problem]
\label{thm:extrapolation}
Let $\Omega \subset \R^{n+1}$ be a corkscrew domain with {$n$-Ahlfors regular} boundary $\pom$ such that $(R^{\mathcal L}_p)$ is solvable for some $p>1$. Then, there exists $\epsilon>0$ depending on $n$, {the ellipticity constant $\lambda$ of $\mathcal L$}, and the constant of $n$-Ahlfors regularity such that
\[
(R^{\mathcal L}_t) \mbox{ is solvable for all } t\in(1-\epsilon, p).
\]
\end{theorem}
This result is novel even for the Laplacian, although in \cite[Corollary~1.6]{MT} it was shown that
\[
(R^\Delta_p) \implies (R^\Delta_q) \mbox{ for $q \in (1,p)$}.
\]

\vv

Note also that solvability of $(D^{\mathcal L^*}_{p'})$ is already a consequence of $(R^{\mathcal L}_p)$ where $\frac 1 p + \frac 1 {p'} = 1$ thanks to the results in \cite[Theorem A.2]{MT}.\color{black}
\color{black}

\vv

\begin{remark}
Solvability $(R^{\mathcal L}_q)$ for $q\leq1$ stands for solvability in the \haj-Sobolev space $M^{1,q}(\pom)$, in contrast with the classical tangential case, where the problem must be considered in Hardy-Sobolev spaces. In Section \ref{section:hajlasz}, we will see that \haj-Sobolev spaces share many properties with Hardy spaces such as atomic decompositions. In this direction, Badr and Dafni  {showed} in \cite[Theorem 1.5]{BD1} that $\dot M^{1,1}(\wt M) = HS^{1,1}(\wt M)$ when $\wt M$ is a complete Riemannian manifold equipped with a doubling measure that supports a {strong} Poincar\'e inequality.
\end{remark}

\vv

In the converse direction, we also prove extrapolation from the endpoint. To do so, we prove that {$(R^{\mathcal L}_1)$} implies the weak-$\mathcal A_\infty(\sigma)$ property for the {elliptic} measure $\omega_{\mathcal L^*}$, which is then used to show the solvability of $(R^{\mathcal L}_p)$ for some $p>1$.
\begin{theorem}[$(R_1^{\mathcal L})$ implies the weak-$\mathcal A_\infty$ property]
\label{thm:equivalence_R1_weakAinf}
Let $\Omega \subset \R^{n+1}$ be a corkscrew domain with {$n$-Ahlfors regular} boundary $\pom$ and $\omega_{\mathcal L^*}$ be its {$\mathcal L^*$-elliptic} measure.
Then,
\[
(R^{\mathcal L}_1) \mbox{ is solvable $\implies \omega_{\mathcal L^*} \in$ weak-$\mathcal A_\infty(\sigma)$.}\color{black}
\]
{In particular, there exists some $p>1$ such that $(D^{\mathcal L^*}_{p'})$ is solvable, where $1/p+1/p'=1$.}

\end{theorem}

\vv

\begin{theorem}[Extrapolation from the endpoint]
\label{thm:extrapolation_1}
Let $\Omega \subset \R^{n+1}$ be a corkscrew domain with {$n$-Ahlfors regular} boundary $\pom$ such that {$(R^{\mathcal L}_1)$} is solvable.
Then, there exists $p>1$ {depending only on $n$, the Ahlfors regularity constants, the corkscrew constant, the ellipticity constant of the operator $\mathcal L$, and the $(R_1^{\mathcal L})$ constant} such that
\[
(R^{\mathcal L}_q) \mbox{ is solvable for all } q\in(1, p).
\]
\end{theorem}

\vv

{These results are also new even in the case of the Laplacian. Note though, that they were already known in the context of Lipschitz domains (see Section \ref{section:history} below).}

\vv

\begin{remark}
	In the case the operator $\mathcal L$ is the Laplacian or is {such that $A \in \textup{DKP}(\Omega)$} {(see \cite[Definition 1.31]{MPT})}, we also have
	\[
	 \omega_{\mathcal L^*} \in\mbox{ weak-$\mathcal A_\infty(\sigma)$}\implies(R^{\mathcal L}_1) \mbox{ is solvable}.\color{black}
	\]
This is a consequence of Theorem \ref{thm:extrapolation}, that $(D^{\mathcal L^*}_{p'})$ implies $(R_p^{\mathcal L})$ for these operators (see \cite[Theorem 1.33]{MPT}), and Theorem \ref{thm:characterization_weak_A_infty}. \color{black} 
	Also, from Theorem \ref{thm:equivalence_R1_weakAinf}, we deduce the equivalence between the solvability of the Dirichlet problem in BMO and the solvability of the regularity problem $(R^\Delta_1)$ for the Laplacian in corkscrew domains with {$n$-Ahlfors regular} boundary (see \cite{HL} by Hofmann and Le for the results concerning the solvability of the Dirichlet problem in BMO). 
\end{remark}
 
Moreover, under the hypothesis that $\partial\Omega$ supports a weak $(1,1)$-Poincar\'e inequality \eqref{eq:Poincare}, we also obtain results for the solvability of the tangential regularity problem $(\wt R^{\mathcal L}_1)$. In fact, in Section \ref{section:tangential_sobolev}, we will relate solvability $(\wt R^{\mathcal L}_1)$ with solvability in another Hardy-Sobolev space  (see Section \ref{section:HMT_def} for the definition of the Hardy HMT space $H^{1,1}$).
\color{black}
\begin{theorem}
\label{thm:solvability_classical_regularity_poincare}
Let $\Omega \subset \R^{n+1}$ be a domain with uniformly $n$-rectifiable boundary $\pom$ that supports a weak $(1,1)$-Poincar\'e inequality.
Then, 
\[
(R^{\mathcal L}_1)\mbox{ is solvable $\iff$ } (\widetilde R^{\mathcal L}_1) \mbox{ is solvable.}
\]
\end{theorem}

\begin{remark} 
\label{remark:Azzam_Tapiola_Tolsa}
Recently, there have been some advances connecting the existence of weak Poincar\'e inequalities supported in the boundary of a domain with some geometric properties of the domain. In particular, Azzam showed in \cite{A} that any $n$-Ahlfors-regular set $E\subset \R^{n+1}$ supporting a {weak $(1,n)$}-Poincar\'e inequality  \eqref{eq:Poincare} must be uniformly $n$-rectifiable. Also, Tapiola and the third named author proved in \cite[Theorems 1.2 and 1.3]{TT} that if $\Omega$ is a $2$-sided local John domain with $n$-Ahlfors regular boundary $\pom$ (\color{black}{or equivalently $\Omega$ is a two-sided chord-arc domain}\color{black}), then $\pom$ must support a {weak $(1,1)$}-Poincar\'e inequality \eqref{eq:Poincare}.
\end{remark}

In addition, under better connectivity assumptions, we prove the following reverse estimate to $(R_1^{\mathcal L})$.
\begin{proposition}
	\label{prop:reverseregularity}
	
	Let $\Omega \subset \R^{n+1}$ be a domain with {$n$-Ahlfors regular} boundary satisfying the local John condition. 
	Let $f$ be a Lipschitz function in $M^{1,1}(\partial\Omega)$ 
	and $u$ be the solution of the continuous Dirichlet problem for $\mathcal L$ with boundary data $f$. Then, we have
	\begin{equation*}
		\Vert f \Vert_{\dot M^{1,1}(\sigma)} \lesssim\|\wt N(\nabla u)\|_{L^{1}(\sigma)}.
	\end{equation*}
\end{proposition}
\color{black}

Finally, we also prove analogous extrapolation results for the Poisson regularity problem (see \eqref{eq:def_mod_Poisson_reg} for its definition and Section \ref{section:variants} for related results).
\begin{theorem}
	\label{thm:extrapolation_Poisson}
	Let $\Omega\subset \R^{n+1}$ be a corkscrew domain with $n$-Ahlfors regular boundary $\pom$ such that either $\Omega$ is bounded or $\pom$ is unbounded, and $1<p\leq2$. Then, there exists $\epsilon>0$ depeding only on $n$, the Ahlfors regularity constant of $\pom$, and the ellipticity constant of $\mathcal L$ such that solvability of the modified Poisson regularity problem $({\wt{PR}}^{\mathcal L}_p)$ implies solvability of $(\widetilde{{PR}}^{\mathcal L}_{q})$ for all $q\in(1-\epsilon,p)$.
\end{theorem}
\color{black}

We also want to highlight some intermediate results which may be of independent interest. These are: \color{black} {the atomic decomposition of tent spaces defined on domains that appears in Section \ref{section:tent_spaces},} the atomic decompositionss of \haj-Sobolev spaces $M^{1,p}$, $0<p\leq1$ proved in Section \ref{section:hajlasz}, the connections between \haj-Sobolev spaces and other tangential Sobolev spaces for uniformly $n$-rectifiable sets supporting weak $(1,1)$-Poincar\'e \eqref{eq:Poincare} inequalities proved in Section \ref{section:tangential_sobolev}, and especially the localization theorem (Theorem \ref{thm:localization}) used to show Theorem \ref{thm:extrapolation}.

\vv

\subsection{Historical remarks}
\label{section:history}

Let us provide some historical context regarding the solvability of the regularity problem. For an extensive account of the history of the solvability of the Dirichlet problem in rough domains for the Laplacian, and more general elliptic operators and systems, we refer the reader to the introductions of \cite{MT,MPT}. 

The regularity problem was first solved in $C^1$ domains for all $p\in(1,+\infty)$ by Fabes, Jodeit, and Rivi\`ere \cite{FJR} by proving the invertibility of the layer potential operators using Fredlholm theory. In Lipschitz domains, Fredholm theory is not available and new ideas were needed in order to show invertibility of the layer potentials. Verchota \cite{V} used the Rellich identity introduced in \cite{JK2} by Jerison and Kenig to show the solvability of $(\wt R^\Delta_p)$ for $p\in(1,2]$. This result was later extrapolated to the optimal range of exponents by Dahlberg and Kenig in \cite{DK1}, without the use of layer potentials. They showed that $(\wt R^\Delta_p)$ is solvable in Lipschitz domains for $p\in[1,2+\varepsilon)$, for some $\varepsilon>0$ depending only on the dimension and the Lipschitz constants of the domain.
In chord-arc domains with vanishing constant (also called regular Semmes-Kenig-Toro domains, see \cite[Definition 4.8]{HMT}), invertibility of layer potentials was proved by Hofmann, Mitrea, and Taylor in \cite{HMT}.

Concerning the duality between $(\wt R_p)$ and $(D_{p'})$, in Lipschitz domains, Verchota established $(\wt R^\Delta_p) \iff (D^\Delta_{p'})$ for $p\in(1,+\infty)$ for the Laplacian. For general real elliptic operators in divergence form, Kenig and Pipher \cite{KP} showed that $(\wt R^{\mathcal L}_p)$ implies $(D^{\mathcal L^*}_{p'})$ for $p\in(1,+\infty)$. Also in Lipschitz domains, Dindo\v{s} and Kirsch \cite{DK2} obtained the endpoint case where the regularity data is in the Hardy-Sobolev space $HS^{1,1}(\pom)$ and the Dirichlet data is in $\operatorname{BMO}(\pom)$, also for general elliptic operators in divergence form. 
For corkscrew domains with $n$-Ahlfors regular boundary, the second and third named authors \cite{MT} proved that {$(D^\Delta_{p'})\Rightarrow (R^\Delta_p) $ for the Laplacian and $(R^{\mathcal L}_p) \Rightarrow (D^{\mathcal L^*}_{p'})$ for general elliptic operators}. This was the first time the regularity problem $(R_p)$ in \haj-Sobolev spaces was considered.  {In corkscrew domains with uniformly $n$-rectifiable boundaries,  the second and third named authors in collaboration with Poggi} showed in \cite{MPT} that $(D^{\mathcal L^*}_{p'})$ implies $(R^{\mathcal L}_p)$ for elliptic operators $\mathcal L$ satisfying the Dahlberg-Kenig-Pipher condition (see \cite[Definition 1.18]{MPT}) by going through the $p'$-solvability of a certain Poisson-Dirichlet problem. This work was finished at the same time as \cite{DHP}, where {Dindo\v{s}}, Hofmann, and Pipher found similar results in the context of Lipschitz {graph} domains  using different techniques.  {Their proof is a lot shorter than the one in \cite{MPT} but it does not seem to extend to  rougher domains since it uses the existence of a uniform``preferred direction" at the boundary,  as occurs in a domain above a Lipschitz graph.} {More recently, Feneuil gave in \cite{F} an alternative and simpler proof of solvability of $(R^{\mathcal L}_p)$ in $\mathbb R^{n+1}_+$ for Dahlberg-Kenig-Pipher operators.}
Also,  Dai, Feneuil, and Mayboroda showed in \cite{DFM} that in uniform domains (that is corkscrew domain with {Ahlfors} regular boundaries and Harnack chains), solvability of the regularity problem {in the  \haj-Sobolev space} is stable under Carleson perturbations of the elliptic operator.

{We refer the reader to Section 1.2.4 in \cite{MPT} for an exposition of the history of the Poisson problem.}

\subsection{Strategy of the proofs}
In order to prove Theorem \ref{thm:extrapolation} we will {implement a strategy inspired by the works} \cite{DK1,KP}, although the lack of regularity of the domains we consider greatly complicates some of the steps. On one hand, the arguments of \cite{DK1} require the use of (classical tangential) Hardy-Sobolev spaces defined on the boundary of Lipschitz domains and its atomic decompositions. For this reason, we develop atomic decompositions for \haj-Sobolev  spaces $M^{1,p}$ on {uniformly perfect metric spaces equipped with a doubling measure} in Section  \ref{section:hajlasz}. These atomic decompositions are influenced by the ones in the works of Badr and Dafni \cite{BD1}, and Badr and Bernicot \cite{BB}, {but they present several differences, both in idea and in proof}. The main ones are that
we work in metric spaces instead of manifolds and thus we cannot consider ``classical" derivatives, we prove atomic decomposition for the $M^{1,p}$ spaces for the full range $0<p\leq1$, not only for $p=1$, and we consider {both compact and unbounded} spaces. Also, {our proof of the atomic decomposition (Theorem \ref{thm:atomic_decomposition}) avoids the use of Sobolev embeddings which allows us to directly obtain a decomposition in Lipschitz atoms.}
On the other hand, we develop a localization theorem in the style of \cite[Theorem 5.19]{KP}, but the lack of connectivity of the domain makes the proof interesting on its own right. {The proof is quite delicate, but the idea is to upper bound the term $\wt N_{R/2}(\nabla u)$ in terms of the elliptic measure of the domain $\omega_{\mathcal L^*}$. Then, we deal with the elliptic measure using that it belongs to weak reverse H\"older class.}

Our proof of Theorem \ref{thm:equivalence_R1_weakAinf} {follows closely} the one in the appendix of \cite{MT}, which in turn has some similarities with the proofs in \cite{KP,DK2} although the absence of connectivity in our domains brings {significant} difficulties. {Moreover, in the present case, working in the endpoint $p=1$ makes things more complicated}. We need to employ a weak reverse H\"older inequality for the maximal non-tangential operator (Lemma \ref{lemma:reverse_holder_nontangential}), to avoid problems with the fact that the Hardy-Littlewood maximal operator is unbounded in $L^1$. We also prove a weak reverse H\"older inequality involving the Hardy-Littlewood maximal operator (in relation with {the} $L \log L$ norm), which later improves itself to a weak reverse $p$-H\"older inequality giving the desired result. 
Then, Theorem \ref{thm:extrapolation_1} is a consequence of Theorem \ref{thm:equivalence_R1_weakAinf}, the localization theorem \ref{thm:localization}, and a good-$\lambda$ inequality, in a similar fashion to the proofs of extrapolation in \cite{DK1,KP}.

\color{black}

The proof of Theorem \ref{thm:solvability_classical_regularity_poincare} follows from the theory of Sobolev spaces on metric spaces supporting weak $(1,1)$-Poincar\'e inequalities \eqref{eq:Poincare}. For more information on this topic, see the book \cite{HKST} by Heinonen, Koskela, Shanmugalingam, and Tyson. The proof of Proposition \ref{prop:reverseregularity} follows from arguments in \cite{KP, HMT}, and the proof of Proposition \ref{prop:equiv_cald_maximal_hardy_HMT_space}, that is, the relation between the Hardy-HMT space $H^{1,1}$ (see Section \ref{section:HMT_def}) and the \haj-Sobolev space $M^{1,1}$ follows from the techniques in \cite{BD2}. \color{black}

{
Finally, the proof of Theorem \ref{thm:extrapolation_Poisson} is very similar to the one of Theorem \ref{thm:extrapolation}, but it requires an atomic decomposition for the tent spaces $T^p_2(\Omega)$ defined on domains. This atomic decomposition (Theorem \ref{thm:atomic_decomposition_tent}) is inspired by the one in \cite{CMS} but with some complications originated by the lack of structure of our domains.
}

\subsection{Organization of the paper} 
The paper is organized as follows. In Section \ref{section:prelim} we review some necessary results needed along the paper. In Section \ref{section:hajlasz} we develop atomic decompositions for the \haj-Sobolev spaces. In Section \ref{section:tangential_sobolev} we prove Theorem \ref{thm:solvability_classical_regularity_poincare}, Proposition \ref{prop:equiv_cald_maximal_hardy_HMT_space} about the relation between the spaces $H^{1,1}$ and $M^{1,1}$ under the hypothesis of a weak $(1,1)$-Poincar\'e inequality \eqref{eq:Poincare}, and some further results on the reversibility of the $(R^{\mathcal L}_1)$ estimate for more regular domains (Proposition \ref{prop:reverseregularity}).
In Section \ref{section:extrapolation} is where we show the main results of the paper. We start by proving the localization Theorem \ref{thm:localization}, and we use it to obtain Theorem \ref{thm:extrapolation} in Section \ref{section:extrapolation_down}. Lastly, we prove Theorem \ref{thm:equivalence_R1_weakAinf} in Section \ref{section:extrapolation_up}, and Theorem \ref{thm:extrapolation_Poisson} in Section \ref{section:extrapolation_Poisson}.

\section{Preliminaries}
\label{section:prelim}
We will write $a\lesssim b$ if there is a constant $C>0$ so that $a \leq Cb$ and $a \approx b$ to mean $a\lesssim b \lesssim a$. For a ball $B = B(x_0, r)$, we write $kB$ to denote $B(x_0, kr)$ for $k>0$. For functions $f \in L^1_{\operatorname{loc}}(\sigma)$, we write $f_B$ to denote $\fint_B f(x)\, d\sigma(x)$.

\vv

\subsection{\haj-Sobolev spaces}
Let $(X, \,d(\cdot,\cdot), \,\sigma)$ be a {uniformly perfect metric space} (see below for the definition) equipped with a doubling measure $\sigma$. \begin{remark}[Uniformly perfect space]
	\label{rmk:uniformly_perfect}
	We say that a metric space $X$ is \textit{uniformly perfect} if there exists a constant $\lambda\geq2$ such that
	\[
	\mbox{$\lambda B\backslash B \neq \varnothing$ for all balls $B$ such that $\lambda B\neq X$.}
	\]
	Moreover, since $\sigma$ is doubling, we can assume that $\sigma(\lambda B) \approx \sigma(\lambda B \backslash B)$. For example, by choosing larger $\lambda$, we may assume there exists $x \in  (\lambda-1/3) B \backslash \left(4/3\right) B$. Hence, 
	\begin{align*}
		\sigma(\lambda B \backslash B) \geq \sigma(B(x, r(B)/3)) \gtrsim \sigma(B(x,2\lambda r(B))) \geq \sigma(B) \gtrsim\sigma(\lambda B).
	\end{align*}
\end{remark}

We start by presenting a Leibniz' rule for \haj-Sobolev functions.
\begin{lemma}[Leibniz' rule for \haj-Sobolev functions, Lemma 5.20 \cite{HK}]
\label{lemma:leibniz}
Let $u \in M^{1, p}(X)$ and $\varphi$ be a bounded Lipschitz function. Then $u \varphi \in M^{1, p}(X)$. Moreover, if $L$ is a Lipschitz constant of $\varphi$ and $\supp \varphi=K$, then
\[
\left(g\|\varphi\|_{\infty}+L|u|\right) \chi_K \in D(u \varphi) \cap L^p(X),
\]
for every $g \in D(u) \cap L^p(X)$.
\end{lemma}

\vv

Next, we will prove two kinds of Poincar\'e inequalities for \haj-Sobolev functions. One is for $p\geq1$ and functions with zero mean and the other is for all $p>0$ and functions with compact support.

\vv

\begin{theorem}[Poincar\'e inequalities]
\label{thm:poincare_compact_supp}
Let $f \in M^{1,p}(X)$ and $B \subset X$ be a ball. Then, if $p\geq1$, we have
\[
\int_B |f-f_B|^p \, d\sigma \lesssim r(B)^p \int_B |\nabla_H f|^p \, d\sigma, \quad \mbox{for all $\nabla_H f \in D(f)$.}
\]
Moreover, for all $p>0$, if $\supp f \subset B$ and $\sigma(\lambda B \backslash B) \gtrsim \sigma(B)$, we get 
\[
\int_B |f|^p \, d\sigma \lesssim r(B)^p \int_{\lambda B} |\nabla_H f|^p \, d\sigma, \quad \mbox{for all $\nabla_H f \in D(f)$.}
\]
\end{theorem}
\begin{proof}
    For a proof of the first inequality see \cite[Proposition 2.1]{MT} for example.

	For the second inequality, for any $\varepsilon>0$ we may choose $z\in \lambda B \backslash B$ (thus $f(z)=0$) such that $\nabla_H f(z) \leq \varepsilon + \inf_{\lambda B\backslash B} \nabla_H f$. Thus, $\nabla_H f(z)^p \sigma(B) \lesssim \nabla_H f(z)^p \sigma(\lambda B\backslash B) \leq \varepsilon \sigma(\lambda B \backslash B) + \int_{\lambda B\backslash B} |\nabla_H f|^p d\sigma \leq \varepsilon \sigma(\lambda B)+ \int_{\lambda B} |\nabla_H f|^p d\sigma$. 
	Then, we have
	\begin{align*}
	\frac{1}{r(B)^p}\int_B |f|^p d\sigma &=
	\frac{1}{r(B)^p}\int_B |f(x) - f(z)|^p d\sigma(x) \leq \frac{\diam(\lambda B)^p}{r(B)^p} \int_{B}(\nabla_H f(x)+\nabla_H f(z))^p d\sigma(x) \\
	&\lesssim_p \int_B  |\nabla_H f|^p d\sigma + |\nabla_H f(z)|^p \sigma(B) \lesssim  \int_B  |\nabla_H f|^p d\sigma + \int_{\lambda B}|\nabla_H f|^p d\sigma + \varepsilon \sigma(\lambda B)\\
	&\lesssim \int_{\lambda B } |\nabla_H f|^p \, d\sigma + \varepsilon \sigma(\lambda B).
	\end{align*}
Since, this is true for all $\varepsilon>0$, we obtain
\[
\int_B |f|^p \, d\sigma \lesssim r(B)^p \int_{\lambda B} |\nabla_H f|^p \, d\sigma
\]
as desired.
\end{proof}

\subsection{Whitney decomposition}
We introduce the \textit{Hardy-Littlewood maximal operator} as
\[
M_{\sigma}f(x) := \sup_{B(z,r) \ni x} \fint_{B(z,r)}  |f(y)|\,d\sigma(y), \quad M_{\sigma}\mu(x) := \sup_{B(z,r) \ni x} \frac{\mu(B(z,r))}{\sigma(B(z,r))},
\] 
the \textit{centered Hardy-Littlewood maximal operator} as
\[
M_{c,\sigma}f(x) := \sup_{r>0} \fint_{B(x,r)} |f(y)|\,d\sigma(y), \quad M_{c,\sigma}\mu(x) := \sup_{r>0} \frac{\mu(B(x,r))}{\sigma(B(x,r))},
\]
and the \textit{truncated centered Hardy-Littlewood maximal operator} as
\begin{equation} 
	\label{eq:def_maximal_truncat}
M_{c,\sigma,t} f(x) := \sup_{0<r<t} \fint_{B(x,r)} |f(y)|\, d\sigma(y), \quad M_{c,\sigma,t}\mu(x):= \sup_{0<r<t} \frac{\mu(B(x,r))}{\sigma(B(x,r))},
\end{equation}
for $f \in L^1_{\text{loc}}(\sigma)$ or positive Radon measures $\mu$.  We will omit the subscript $\sigma$ when there is no possible confusion in the measure used. We will also use extensively the following Whitney ball decomposition, proved for spaces of homogeneous type by Coifman and Weiss in \cite[Theorem 3.2]{CW}.
\begin{theorem}[Whitney ball decomposition, Theorem 3.2 \cite{CW}]
\label{thm:whitney_balls}
Let $(X,d,\sigma)$ be a uniformly perfect metric space with constant $\lambda$ equipped with a doubling measure and suppose $U \subsetneq X$ is an open bounded set. Then there exists a sequence of balls $(B_i)_i$ where $B_i = B(x_i, r_i)$ satisfying
\begin{itemize}
\item[(i)] $U=\bigcup_i \underline{B_i} = \bigcup_i B_i$ where $\underline{B_i} := \frac{1}{\lambda} B_i$, \color{black}
\item[(ii)] there exists a constant $N$ depending only on $\lambda$ and the doubling constant of $\sigma$ such that no point of $X$ belongs to more than $N$ of the balls $B_i$,
\item[(iii)] $\wt{{B}}_i \cap(X\backslash U) \neq \varnothing$ for each $i$, where $\wt{{B}}_i:=\lambda B_i$,
\item[(iv)] if $B_i \cap B_j \neq \emptyset$, then $r_i \approx r_j$.
\end{itemize}
\end{theorem}
\color{black}

\subsection{Calder\'on maximal operator}
{In a metric measure space $(X, d(\cdot,\cdot), \sigma)$,  where $\sigma$ is a Borel regular measure,  for $f \in L^{1}_{\operatorname{loc}}(X)$,  we define the \textit{Calder\'on-type maximal operator} by
\[
\Lambda_{\sharp} f(x):=\sup _{r>0} \fint_{B(x, r)} \frac{\left|f(y)- f_B \right|}{r} \,d \sigma(y),
\]
 It was shown in \cite[Theorem 3]{KT} that if $\sigma$ is doubling, this operator characterizes functions in $M^{1,1}(X)$.}
\begin{theorem}
	\label{thm:calderon_maximal}
 {	If $\sigma$ is doubling,  then, for any $f\in \dot{M}^{1,1}(X)$, we have
	\[
	\Vert f \Vert_{\dot{M}^{1,1}(X)} \approx \Vert \Lambda_{\sharp} f \Vert_{L^1(X)}.
	\]
	Moreover, there exists $C>0$ such that if $f \in L^1_{\operatorname{loc}}(X)$ and $\Lambda_\sharp f \in L^1(X)$, then $C\Lambda_\sharp f \in D(f)\cap L^1(X)$.}
\end{theorem}

\subsection{Hardy space $H^1$}
We first introduce the atomic characterization of the real Hardy space $H^1$ on a {$d$-Ahlfors regular} set $E \subset \R^{n+1}$ with  measure $\sigma := \HH^d|_{\pom}$.
We say that a function $a(x)$ is an \textit{atom} if
\begin{enumerate}
	\item[(i)] $\supp a \subset B\cap E$ for some ball $B$ centered on $E$,
	\item[(ii)] $\Vert a \Vert_{L^\infty(E)} \leq \sigma(B)^{- 1}$,
	\item[(iii)] $\int a \, d\sigma = 0$.
\end{enumerate}
We define the \textit{(atomic) Hardy space} $H^{1}(E)$ as the subspace of functions $f\in L^1(E)$ that admit an \textit{atomic decomposition}, that is, there exists a family of atoms $(a_j)_j$ and a sequence of real numbers $(\lambda_j)_j$ with the property that
\[
f = \sum_j \lambda_j a_j \quad \mbox{in $L^1(E)$}
\]
and $\sum_j |\lambda_j| < +\infty$. We equip $H^1(E)$ with the norm $\Vert f \Vert_{H^1} = \inf \sum_j |\lambda_j|$ where the infimum is taken over all possible atomic decompositions. {This definition can be compared with the one of Hardy-Sobolev spaces in the introduction.}

It is well known that real Hardy spaces admit an equivalent definition in terms of the {grand maximal function}.
For $f \in L_{\text {loc}}^1(E)$, we define its \textit{grand maximal function} as
\[
\wt {\mathcal M}_{\text{gr}} f(x):=\sup _{\varphi \in \mathcal{T}_1(x)}\left|\int f \varphi \,d \sigma\right|,
\]
where $\mathcal{T}_1(x)$ is the set of Lipschitz functions $\varphi$ for which there exists a ball $B$ centered at $x$ with radius $r(B)$ such that
\begin{itemize}
\item $\operatorname{supp} \varphi \subset B$,
\item $\|\varphi\|_{L^{\infty}(E)} \leq \sigma(B)^{-1}$ and $|\varphi|_{\operatorname{Lip}} \leq(r(B) \sigma(B))^{-1}$.
\end{itemize}
Then, we have the equivalence
\[
H^1(\pom) = \{ f \in L^1_{\text{loc}}(E) \, | \, \wt{ \mathcal M}_{\text{gr}}f\in L^1(E)\}
\]
and, moreover, both norms are equivalent
\[
\Vert \wt {\mathcal M}_{\text{gr}}f \Vert_{L^1(E)} \approx \Vert f \Vert_{H^1(E)}.
\]
\color{black}

\subsection{Measure theoretic boundary, tangential gradients and Hofmann-Mitrea-Taylor Sobolev spaces}
\label{section:HMT_def}
An open set $\Omega \subset \R^{n+1}$ has \textit{finite perimeter} if the distributional gradient $\nabla \chi_\Omega$ of $\chi_\Omega$ is a locally finite $\R^{n+1}$-valued measure. From results of De Giorgi and Moser it follows that $\nabla \chi_{\Omega} = -\nu \HH^{n}_{\partial^*\Omega}$, where $\partial^*\Omega\subset \pom$ is the reduced boundary of $\Omega$ and $|\nu(x)|=1$ $\HH^n$-a.e. in $\partial^*\Omega$ and $\nu$ is the \textit{measure theoretic unit normal}.
The \textit{measure theoretic boundary} $\partial_* \Omega$ consists of the points $x \in \pom$ such that
\[
\liminf_{r \to 0} \frac{m(B(x,r)\cap\Omega)}{r^{n+1}} > 0 ~~\mbox{ and }~~ \liminf_{r \to 0} \frac{m(B(x,r)\backslash\overline\Omega)}{r^{n+1}} > 0.
\]
When $\Omega$ has finite perimeter, it holds that $\partial_*\Omega \subset \partial^*\Omega \subset \partial\Omega$ and $\HH^n(\partial^*\Omega \backslash \partial_*\Omega)=0$. A good reference for those results is the book of Maggi \cite{M}.

\vv

In \cite[Section 3.6]{HMT}, Hofmann, Mitrea, and Taylor introduced a notion of tangential derivatives and gradient very well suited for integration by parts.

Let $\Omega \subset \R^{n+1}$ be an open set {of locally finite perimeter}. For a $C^1_c$ function $\varphi:\mathbb R^{n+1} \to \R$ and $1\leq j,k \leq n+1$, we define the tangential derivatives on $\partial^* \Omega$ by
\[
\partial_{t,j,k} \varphi := \nu_j (\partial_k \varphi)|_{\pom} - \nu_k (\partial_j \varphi)|_{\pom}.
\]
where $\nu_i$ are the components of the outer unit normal $\nu$ and
\[
\nabla_{\operatorname{HMT}} f:= \left(\sum_k \nu_k \partial_{t,j,k}f\right)_{1 \leq j \leq n+1}.
\]
By the arguments in \cite[p 2676]{HMT}, if $\varphi$, $\psi$ are $C^1$ in a neighborhood of $\pom$ {and compactly supported}, we have that
\begin{equation}
\label{eq:HMT_int_by_parts}
\int_{\partial^*\Omega} \partial_{t,j,k} \psi \varphi \,d \HH^n = \int_{\partial^*\Omega} \psi \partial_{t,k,j} \varphi\, d\HH^n.
\end{equation}

{For $p\geq1$, we denote by   $L^{1,p}(\mathcal{H}^n|_{\partial^* \Omega})$  the  \textup{HMT} \textit{Sobolev space} (see \cite[ display (3.6.3)]{HMT}) defined as the subspace} of functions in $L^p(\mathcal{H}^n|_{\partial^* \Omega})$ for which there exists some constant $C(f)$ such that
\[
\sum_{1 \leq j, k \leq n+1}\left|\int_{\partial^* \Omega} f \partial_{t, k, j} \varphi\, d \mathcal{H}^n\right| \leq C(f)\|\varphi\|_{L^{p^{\prime}}\left(\mathcal{H}^n \mid \partial^* \Omega\right)}
\]
for all $\varphi \in C_c^{\infty}\left(\mathbb{R}^{n+1}\right)$ where $p'$ is the H\"older conjugate exponent. By the Riesz representation theorem, for each $f \in L^{1,p}\left(\left.\mathcal{H}^n\right|_{\partial^* \Omega}\right)$ and each $j, k=$ $1, \ldots, n+1$, there exists some function $h_{j, k} \in L^p\left(\left.\mathcal{H}^n\right|_{\partial^* \Omega}\right)$ such that 
\[
\int_{\partial^* \Omega} h_{j, k} \varphi \,d \mathcal{H}^n=\int_{\partial^* \Omega} f \partial_{t, k, j} \varphi \,d \mathcal{H}^n
\]
and we set $\partial_{t, j, k} f:=h_{j, k}$, so that this is coherent with \eqref{eq:HMT_int_by_parts}. It is easy to check that Lipschitz functions with compact support are contained in $L^{1,p}\left(\HH^n|_{\partial^* \Omega}\right)$.

We introduce a lemma (proved in \cite[Lemma 6.4]{MT}) that clarifies the relationship between the tangential derivative $\nabla_t f$ and $\partial_{t,j,k} f$ for Lipschitz functions on the boundary of a set of locally finite perimeter.
\begin{lemma}[Lemma 6.4, \cite{MT}]
	\label{lemma:6.4-MT}
Let $\Omega \subset \mathbb{R}^{n+1}$ be an open set of locally finite perimeter. Then, for any compactly supported Lipschitz function $f: \partial \Omega \rightarrow \mathbb{R}$ and for every $j, k \in\{1,2, \ldots, n+1\}$, we have
\[
\partial_{t, j, k} f(x)=-\nu_j\left(\nabla_t f\right)_k(x)+\nu_k\left(\nabla_t f\right)_j(x) \quad \mbox{ for } \mathcal{H}^n|_{\partial^* \Omega} \mbox{-a.e. } x \in \partial^* \Omega.
\]
\end{lemma}
Note that in \cite{MT} it is stated for domains with uniformly $n$-rectifiable boundary but only the $n$-rectifiability of $\partial^*\Omega$ is necessary.
\color{black}

Let's restrict ourselves to the case  {that $\partial^*\Omega$ supports a weak {$(1,1)$}-Poincar\'e inequality \eqref{eq:Poincare}.} We {will define} the {\textit{Hardy-Sobolev version of the \textup{HMT} space on $\partial^*\Omega$}, which we denote by $H^{1,1}(\HH^n|_{\partial^*\Omega})$,} using a grand maximal characterization in a similar manner to \cite{BD2}.  First,  {for any $x \in \partial^*\Omega$, we set}
\[
\mathcal M_{\text{gr}}(\nabla_{\operatorname{HMT}} f)(x) := \sup_{g \in \wt{\mathcal T}_1(x)} \left | \int f g\,d\HH^n|_{\partial^*\Omega}\right|
\]
for any $f\in L^1_{\operatorname{loc}}(\partial^*\Omega)$ with prescribed $\nabla_{\operatorname{HMT}} f$ and where {
\[
\wt{\mathcal T}_1(x) = \left\{g \in L^{\infty}(\partial^*\Omega) \, \Bigg| \, 
\begin{array}{c} \exists\,\,\textup{ball}\,B \,\,\textup{s.t.}\, x\in B, \,\,\operatorname{supp} g \subset B,\\
	\Vert g \Vert_\infty \leq (r(B)\HH^n(B \cap \partial^*\Omega))^{-1}, \\
	\int_{\partial^* \Omega} g \, d\HH^n = 0
\end{array}\right\}.
\] 
}

Let's show that the previous grand maximal operator is well defined in the sense that if $f,h$ are Lipschitz on $\partial^*\Omega$ and satisfy $\nabla_{\operatorname{HMT}} f = \nabla_{\operatorname{HMT}} h$, then
\begin{equation}
	\label{eq:grand_maximal_HMT_well_defined}
	\sup_{g \in \wt{\mathcal T}_1(x)} \left | \int f g\,d\HH^n|_{\partial^*\Omega}\right| = \sup_{g \in \wt{\mathcal T}_1(x)} \left | \int h g\,d\HH^n|_{\partial^*\Omega}\right|.
\end{equation}
Without loss of generality, assume {that} the left hand side of \eqref{eq:grand_maximal_HMT_well_defined} is strictly larger than its right hand side.
Then, we may fix $g \in \wt{\mathcal T}_1(x)$ with $\supp g \subset B$ such that 
\[
\left | \int (f-h) g\,d\HH^n|_{\partial^*\Omega}\right|\geq\left | \int f g\,d\HH^n|_{\partial^*\Omega}\right| - \left | \int h g\,d\HH^n|_{\partial^*\Omega}\right| > 0.
\]
But, using that $g$ has zero mean and that $\partial^*\Omega$ supports a weak {$(1,1)$-}Poincar\'e inequality \eqref{eq:Poincare}, 
we obtain
\begin{align*}
	\left | \int (f-h) g\,d\HH^n|_{\partial^*\Omega}\right| &\leq \Vert g \Vert_{L^\infty} \int_B \left|f-h - \fint_{B}(f-h)\,d\HH^n|_{\partial^*\Omega}\right | \, d\HH^n|_{\partial^*\Omega} \\
	&\lesssim  \Vert g \Vert_{L^\infty} \, {r(B)}\,\int_{\Lambda B} |\nabla_t(f-h)| \, d\HH^n|_{\partial^*\Omega} = 0
\end{align*}
{since, by Lemma \ref{lemma:6.4-MT},  $\nabla_{\operatorname{HMT}} (f-h)(x) = -\nabla_t (f-h)(x)$ for $\mathcal{H}^n|_{\partial^* \Omega}$-$\textup{a.e.}\,x \in \partial^* \Omega$}.  Hence, \eqref{eq:grand_maximal_HMT_well_defined} holds.

Then, {for an open set $\Omega \subset \mathbb{R}^{n+1}$ of locally finite perimeter such that $\partial^*\Omega$ supports a weak {$(1,1)$}-Poincar\'e inequality \eqref{eq:Poincare}, we  define}
\[
H^{1,1}(\HH^n|_{\partial^*\Omega}):= \{f \in L^{1,1}(\HH^n|_{\partial^*\Omega}) \,|\, \mathcal M_{\text{gr}}(\nabla_{\operatorname{HMT}} f) \in L^1(\HH^n|_{\partial^*\Omega}) \}
\]
and we equip it with the seminorm
\[
\Vert f \Vert_{H^{1,1}(\HH^{n}|_{\partial^*\Omega})} := 
\Vert\mathcal M_{\text{gr}}(\nabla_{\operatorname{HMT}} f)\Vert_{L^1(\HH^{n}|_{\partial^*\Omega})}.
\]
It is easy to see that Lipschitz functions with compact support {on $\partial^* \Omega$} are in $H^{1,1}(\HH^n|_{\partial^*\Omega})$.

\begin{remark}
	{If $\pom$ and $f$ are smooth,  $\pom$ can be regarded as a Riemannian manifold and it holds that $\pom=\partial^*\Omega$. Then one can show that $g$ is the divergence of a compactly supported vector field $\Phi\in L^\infty(\partial\Omega, T(\pom))$ (where $T(\pom)$ is the tangent space to $\pom$) supported on $B$ and}
	\[
	\int f g \, d\HH^n|_{\pom} = \int f \operatorname{div} \Phi \, d\HH^n|_{\pom} = -\int (\nabla_t f, \Phi)_x\, d\HH^n|_{\pom}(x)
	\]
	holds. In \cite[Lemma 2.3]{MT}, the second and third named authors proved that in this case $\nabla_t f = -\nabla_{\operatorname{HMT}} f$ $\HH^n|_{\pom}$-a.e., which shows {the connection between the definition of the grand maximal operator $\mathcal M_{gr}$ with the other grand maximal operator $\wt{\mathcal M}_{gr}$ defined in the previous section.}  
	\color{black}
\end{remark}

\begin{remark}
	\label{rmk:diff_hardy}
	Note that since $\partial^*\Omega$ is $n$-rectifiable,  $HS^{1,1}(\partial^*\Omega)$  is well-defined and is not necessarily equal to the space $H^{1,1}(\partial^*\Omega)$ defined here.   Indeed, the space $HS^{1,1}(\partial^*\Omega)$ contains functions that are (controlled) sums of Lipschitz atoms and $H^{1,1}(\partial^*\Omega)$ contains functions whose HMT gradients have a grand maximal function in $L^1(\partial^*\Omega)$.  These two approaches to Hardy-Sobolev spaces on $\partial^*\Omega$ were already studied in \cite{BD1, BD2} {non-compact Riemannian manifolds equipped with a doubling measure and admitting a strong Poincar\'e inequality.}

\end{remark}
\color{black}

\begin{remark}
	Following the ideas in \cite{BD2}, we can also show that the functions in $\wt{\mathcal T}_1(x)$ can also be regarded as a generalized divergence of a vector field satisfying an integration by parts formula \eqref{eq:integration_by_parts_duality}.
	
	Let 
	\[
	\mathcal S := \{ V \in L^1( \HH^n|_{\partial^*\Omega}, \R^{n+1}) \, | \, \exists f \in L^{1,1}( \HH^n|_{\partial^*\Omega}), ~V = \nabla_{\operatorname{HMT}} f \text{ on }\partial^*\Omega \}
	\]
	be a subspace of $L^1( \HH^n|_{\partial^*\Omega}, \R^{n+1})$ with norm $\Vert V \Vert_{L^1( \HH^n|_{\partial^*\Omega}, \,\R^{n+1})} = \int_{\partial^*\Omega} |V| \,d\HH^n|_{\pom}$. {Recall that, thanks to \cite[Lemma 6.4]{MT}, we know that $\nabla_t f = -\nabla_{\operatorname{HMT}} f$ $\HH^n|_{\partial^*\Omega}$-a.e.}
	Let $\tilde g \in L^\infty(B\cap\partial^*\Omega)$ for some $B$ centered on $\partial^*\Omega$ with $\Vert \tilde g \Vert_\infty \leq 1$ and $\int \tilde g \,  d\HH^n|_{\partial^*\Omega} = 0$ and define a linear functional on $\mathcal S$ by
	\[
	F(V) := - \int_B \tilde gf \, d\HH^n|_{\partial^*\Omega} \quad \mbox{if $V= \nabla_{\operatorname{HMT}} f$},
	\]
	which is well defined since $\int \tilde g \, d\HH^n|_{\partial^*\Omega} =0$, and the weak {$(1,1)$}-Poincar\'e inequality \eqref{eq:Poincare} implies that if $V = \nabla_{\operatorname{HMT}}f = 0$ then $f$ must be constant.
	We check that $F$ is bounded thanks to the weak $(1,1)$-Poincar\'e inequality \eqref{eq:Poincare}:
	\begin{align*}
		|F(V)| =  \left |\int_B \tilde g\left (f- \int_B f \,  d\HH^n|_{\partial^*\Omega} \right )  \,d\HH^n|_{\partial^*\Omega} \right| 
		&\leq C r(B) \Vert \tilde g \Vert_\infty \int_{\Lambda B} |\nabla_t f|  d\HH^n|_{\partial^*\Omega} \\
		&= C r(B) \Vert \tilde g \Vert_\infty \Vert V \Vert_{L^1({ \HH^n|_{\partial^*\Omega}}, \,\R^{n+1})}.
	\end{align*}
	The Hahn-Banach theorem shows that $F$ can be extended to a bounded linear functional on $L^1({ \HH^n|_{\partial^*\Omega}}, \R^{n+1})$, without increasing its norm.
	Thus, by duality, there exists a vector field $\Phi \in L^{\infty}({ \HH^n|_{\partial^*\Omega}}, \R^{n+1})$ such that
	\begin{equation}
		\label{eq:integration_by_parts_duality}
		\int (\Phi, V)\,d \HH^n|_{\partial^*\Omega}  = -\int \tilde gf\, d\HH^n|_{\partial^*\Omega}, \quad \mbox{for all $V \in \mathcal S$ and $f$ satisfying $\nabla_{\operatorname{HMT}} f = V$}
	\end{equation}
	with $\Vert \Phi\Vert_\infty \leq C r(B) \Vert \tilde g \Vert_\infty$. 
	
	This allows us to consider the following equivalent definition of $\mathcal M_{gr}(\nabla_{\operatorname{HMT}}f)$:
	\[
	{\mathcal M}_{gr}(\nabla_{\operatorname{HMT}}f) := \sup_{\Phi \in \widehat{\mathcal T}_1(x)} \left |\int (\nabla_{\operatorname{HMT}}f \cdot \Phi) \, d\HH^n|_{\partial^*\Omega} \right|
	\]
	where the test functions $\Phi \in \widehat{\mathcal T}_1(x)$ satisfy
	\[
	\begin{array}{c} 
		\Phi \in L^{\infty}(\HH^n|_{\partial^*\Omega}, \,\R^{n+1}),\\
		\exists g \in \wt{\mathcal T}_1(x) \mbox{ such that $g$ and $\Phi$ satisfy \eqref{eq:integration_by_parts_duality} for all $V \in \mathcal S$.}\\
	\end{array}
	\] 
\end{remark}

\subsection{Strong Poincar\'e inequalities on domains with $n$-Ahlfors regular boundary}
We define the $n$-dimensional Hausdorff content of a set $E\subset \R^{n+1}$ as
\[
\mathcal H^n_\infty(E) := \inf_{(B_i)_i}\left\{\sum_{i=1}^\infty \diam(B_i)^n \,\, |\,\, E \subset \bigcup_i B_i, \,0<\diam(B_i)<+\infty \right\}
\]
and the Sobolev $1$-capacity of a set $E \subset \mathbb R^{n+1}$ as
\[
\operatorname{Cap}_1(E) := \inf_{\substack{g \in W^{1,1}(\mathbb R^{n+1}),\\ g \equiv 1 \mbox{ on }E}} \Vert \nabla g \Vert_{L^1(\R^{n+1})}. 
\]
\begin{remark}
	Note that if $E\subset \R^{n+1}$ is an $n$-Ahlfors regular set, we have that
	\[
	\HH^n(E) \approx \HH^n_\infty(E).
	\]
\end{remark}
For $W^{1,1}$ functions vanishing in a set of positive Cap$_1$ capacity, there is the following Poincar\'e inequality due to Maz'ya:
\begin{theorem}
	Let $u \in W^{1,1}(B)$ where $B \subset \mathbb R^{n+1}$ is the unit ball. Assume that $u$ vanishes on a compact set $F \subset B$ with $\operatorname{Cap}_1(F)>0$. Then, there exists a constant $C_n>0$ depending only on $n$ such that
	\[
	\Vert u \Vert_{L^1(B)} \leq \frac{C_n}{\operatorname{Cap}_1(F)} \Vert \nabla u \Vert_{L^1(B)}.
	\]
\end{theorem}
For a reference of this result see \cite[Theorem 9.1]{Maz}.
On the other hand, \cite[Theorem 3.5]{KKST} states that $n$-Hausdorff content and Sobolev $1$-capacity are comparable.
\begin{theorem}
	Let $F \subset \mathbb R^{n+1}$ be a compact set. Then,
	\[
	\operatorname{Cap}_1(F) \approx \mathcal H_\infty^{n}(F)
	\]
	with comparability constants depending only on $n$.
\end{theorem}
Combining both theorems yields the following strong Poincar\'e inequality for functions vanishing on an $n$-Ahlfors regular set.
\begin{corollary}
	\label{coro:Poincare_functions_vanishing_pom}
	Let $\Omega \subset \mathbb R^{n+1}$ be a corkscrew domain with $n$-Ahlfors regular boundary $\pom$, $x\in\pom$, and $0<r<\diam \pom$. Let $u \in C(\Omega\cap B(x,r)) \cap W^{1,1}(B(x,r)\cap\Omega)$ vanishing on $\pom \cap B$ and extended by $0$ in $\Omega^c\cap B(x,r)$. Then,
	\[
	\Vert u \Vert_{L^1(B(x,r))} \leq {C r} \Vert \nabla u \Vert_{L^1(B(x,r))}
	\]
	with constant $C>0$ depending only on $n$ and the Ahlfors regularity constant of $\pom$.
\end{corollary}

\vv

\color{black}
\subsection{Harmonic and elliptic measures}

Assume that $\pom$ is $n$-Ahlfors regular, thus the continuous Dirichlet problem for an elliptic operator $\mathcal L$ satisfying \eqref{eq:ellipticity_conditions_operator} is solvable. By the maximum principle and the Riesz Representation Theorem, there exists a family of probability measures $\{\omega_{\mathcal L}^x\}_{x\in\Omega}$ on $\pom$ so that for each $f \in C_c(\pom)$  and each $x \in \Omega$, the solution $u$ to the continuous Dirichlet problem with data $f$ satisfies $u(x) = \int_{\pom} f(\xi) \, d\omega_{\mathcal L}^x(\xi)$. We call $\omega_{\mathcal L}^x$ the \textit{$\mathcal L$-elliptic measure with pole at $x$}.

We present two lemmas very important in the study of the {elliptic} measure $\omega_{\mathcal L}$ of a domain (see \cite{MPT} and the references therein for example). 
\begin{lemma}[Lemma 2.16 in \cite{MPT}]
\label{lemma:Bourgain}
Let $\Omega \subsetneq \mathbb{R}^{n+1}$ be an open set with $n$-Ahlfors regular boundary $\pom$. Then there exists $c>0$ depending only on $n$, the $n$-Ahlfors regularity constant of $\pom$, and the ellipticity constant of $\mathcal L$, such that for any $\xi \in \partial \Omega$ and $r \in(0, \operatorname{diam}(\partial \Omega) / 2]$, we have that $\omega_{\mathcal L}^x(B(\xi, 2 r) \cap \partial \Omega) \geq c$, for all $x \in \Omega \cap B(\xi, r)$.
\end{lemma}

\vv

\begin{lemma}[Lemma 2.17 in \cite{MPT}]
\label{lemma:bound_green_harmonic_measure}
Let $\Omega \subsetneq \mathbb{R}^{n+1}$ be an open set with $n$-Ahlfors regular boundary $\pom$. Let $B=B\left(x_0, r\right)$ be a closed ball with $x_0 \in \partial \Omega$ and $0<r<\operatorname{diam}(\partial \Omega)$. Then
\[
G(x, y) \lesssim \frac{\omega_{\mathcal L}^y(4 B)}{r^{n-1}}, \quad \text { for all } y \in \Omega \backslash 2 B \text { and } x \in B \cap \Omega \text {, }
\]
where $G(x,y)$ is the Green's function of the domain $\Omega$ with pole at $y$ and the implicit constant depends only on $n$, the ellipticity constant of $\mathcal L$, and the $n$-Ahlfors regularity constant of $\pom$.
\end{lemma}
Now we present a Lemma concerning the rate of decay at infinity of a bounded $\mathcal L$-harmonic function vanishing at the boundary (except on a ball). See \cite[Lemma 2.11]{AGMT} for its proof.
\begin{lemma} 
	\label{lemma:AGMT}
Let $\Omega \subset \mathbb{R}^{n+1}$ be a domain with $n$-Ahlfors regular boundary. Let $u$ be a bounded, $\mathcal L$-harmonic function in $\Omega$, and let $B$ be a ball centered at $\partial \Omega$. Suppose that $u$ vanishes continuously in $\partial \Omega \backslash B$ and, if $\Omega$ is not bounded, $u$ also vanishes at $\infty$. Then, there is a constant $\alpha>0$ such that
$$
|u(x)| \lesssim \frac{r(B)^{n-1+\alpha}}{(r(B)+\operatorname{dist}(x, B))^{n-1+\alpha}}\|u\|_{L^{\infty}(\Omega \cap(3 B \backslash 2 B))} .
$$
Both $\alpha$ and the constant implicit in the above estimate depend only on $n$, the Ahlfors-regularity constant of $\partial \Omega$ and the ellipticity constant of $\mathcal L$.
\end{lemma}
\color{black}

We present some of the {main results in connection with} the study of the Dirichlet and regularity problems in this general setting.
\begin{theorem}[Characterization of the weak-$\mathcal A_\infty$ property for harmonic measure]
\label{thm:characterization_weak_A_infty}
Let $\Omega \subset \mathbb \R^{n+1}$, {$n\geq2$} be an open bounded corkscrew domain with $n$-Ahlfors regular boundary $\pom$. Then, the following are equivalent:
\begin{enumerate}
    \item the harmonic measure $\omega$ belongs to weak-$\mathcal A_\infty(\mathcal H^n|_{\pom})$,
    \item the harmonic measure $\omega$ satisfies a weak $p$-reverse H\"older inequality for some $p>1$, that is
    \[
    	\left(\fint_{B\cap\pom} \left (\frac{d \omega^x}{d\sigma} \right)^p d \sigma\right)^{1 / p} \lesssim \fint_{(2 B)\cap\pom} \left (\frac{d \omega^x}{d\sigma} \right) d \sigma 
    \]
    for all balls $B$ centered on $\pom$ with $\diam(B) \leq 2\diam(\pom)$ and all $x \in \Omega \backslash 3B$.
    \item $(D^{\Delta}_{p'})$ is solvable for some $p'<+\infty$,
    \item $(D^\Delta_{\operatorname{BMO}})$ is solvable (see \cite{HL} for its definition),
    \item $(R^\Delta_p)$ is solvable for some $p>1$,
    \item $\Omega$ satisfies the weak local John condition and $\pom$ is uniformly $n$-rectifiable.
\end{enumerate}
{Moreover, for general elliptic operators, properties $(1),~ (2),~\mbox{and}~ (3)$ are equivalent, and $(5) \implies (3)$ for $\frac 1 p + \frac 1 {p'} =1$.\color{black}}
\end{theorem}
We refer the reader to \cite[Theorems 1.4 and 1.6]{HL} for  the implications $(1)\implies(4)\implies(3)\implies(1)$, to \cite[Theorem 1.1]{AHMMT} for $(1)\iff (6)$, to \cite[Theorems 1.2, 1.6 and A.2]{MT} for the relationship between Dirichlet and Regularity problems, and to \cite[Theorem 9.2]{MT} for $(2)\iff (3)$. \color{black}

\vvv

\subsection{Tent spaces on domains}
\label{section:tent_spaces} 
Let $\Omega \subset \R^{n+1}$ be a corkscrew domain with $n$-Ahlfors regular boundary $\pom$ such that either $\Omega$ is bounded or $\pom$ is unbounded, and $\sigma := \HH^n|_{\pom}$.

If $g: \Omega \rightarrow \mathbb{R}$ is a measurable function in $\Omega$, we define the \textit{area functional} of $g$ for a fixed aperture $\alpha>0$ by
$$
\mathcal{A}^{(\alpha)}(g)(\xi):=\left (\int_{\gamma_\alpha(\xi)}|g(x)|^2  \frac{d x}{\delta_{\Omega}(x)^{n+1}} \right)^{1/2}, \quad \xi \in \partial \Omega.
$$ 
In {\cite[Proposition 4.5]{MPT2}} it is shown that that $\Vert \mathcal A^{(\alpha)}(g)\Vert_{L^p} \approx \Vert \mathcal A^{(\beta)}(g) \Vert_{L^p}$ for every $\alpha, \beta>0$, and all $0<p<\infty$. Hence, we will write $\mathcal A$ omitting the dependence with respect to the aperture.
We also define the \textit{$q$-Carleson functional} for $q\geq1$ by
\[
\mathcal{C}_{q}(g)(\xi):=\sup _{r>0} \frac{1}{r^n} \int_{B(\xi, r) \cap \Omega}\left(\fint_{B(x,\delta_\Omega(x)/8)}|g|^q\, dm\right)^{1 / q} d m(x), \quad \xi \in \partial \Omega .
\]

Finally, we introduce the \textit{tent spaces}
$$ 
T_2^p(\Omega):=\left\{g \in L_{{loc}}^2(\Omega):\mathcal{A}(g) \in L^p(\sigma)\right\},
$$ 
for $p \in(0, \infty)$, and we equip them with
$$
\quad\|g\|_{T_2^{p}(\Omega)}:=\Vert\mathcal{A}(g) \Vert_{L^p(\sigma)}.
$$
{For $1\leq p < \infty$, these are Banach spaces (see \cite[page 16]{MPT2}).}

\begin{remark}
	{The relevant definitions and results in \cite{MPT2} for this section are stated for chord-arc domains, but only the corkscrew property and the Ahlfors regularity are used in the proofs, hence they are also valid for the class of domains we consider.}
\end{remark}

Now, we focus on showing an atomic decomposition for functions in these spaces.
Let $0<p\leq1$. We say a function $a(x)$ is a \textit{$T^p_2$ atom} if it is supported in a ball $B(\xi,R)\cap\Omega$ with $\xi \in \partial \Omega$, $0<R<\diam\pom$ and
\[
\int_{B\cap\Omega} a^2(x) \frac{d x}{\delta_\Omega(x)} \leq {\sigma(B\cap\pom)^{-1/p}} .
\]
If $a$ is an atom with support in $B(\xi,R)$, it is easy to check (see \cite[page 23]{MPT2}), that its area functional $\mathcal A(a)$ is supported in $B(\xi, 3R)\cap\pom$, and it satisfies
\[\Vert a \Vert_{T^q_2(\Omega)} \lesssim \sigma(B\cap\pom)^{1/q-1/p},\quad \mbox{ and }\quad\Vert a \Vert_{T_2^p(\Omega)} \lesssim 1
\]
for all $q\in(p,2]$.

\begin{theorem} 
	\label{thm:atomic_decomposition_tent}
	Let $0<p\leq1$ and $f \in L^\infty_c(\Omega)\cap T^p_2(\Omega)$. There exist sequences of real numbers $(\lambda_j)_j$ and $T^p_2$ atoms $(a_j)_j$ such that
	\[
	f = \sum_j \lambda_j a_j \quad \mbox{in $T_2^p(\Omega)$} \quad \mbox{and }\quad \sum_j |\lambda_j|^p \leq C \Vert f \Vert_{T^p_2(\Omega)}^p
	\]
	where $C>0$ only depends on the Ahlfors regularity constant of $\pom$, and the corkscrew constant of $\Omega$.
\end{theorem}

\vv

Before proving {the latter theorem,  let us} introduce some {necessary additional} definitions.

\vv

We say a point $x\in\pom$ has \textit{global $\tau$-density} with respect to a set $F \subset \pom$ if for all $0<r<\diam\pom$,
\[
\frac{\sigma(B(x,r)\cap F)}{\sigma(B(x,r))}\geq\tau.
\]
We denote by $F^*_\tau$ the set of points of global $\tau$-density of $F$. 
Note that for $F$ closed, we clearly have $F^*_\tau\subset F$, and by \cite[Proposition 3.3]{MPT2}, it holds that 
\[
\sigma((F^*_\tau)^c) \leq C \sigma(F^c)
\]
where $C$ depends on $\tau$ and the Ahlfors regularity constant of $\pom$.
Also, for $F\subset \pom$, we let 
\[
R_\alpha(F) := \bigcup_{\xi\in F} \gamma_\alpha(\xi)
\]
where $\gamma_\alpha(\xi)$ is the nontangential cone of aperture $\alpha$ associated to $\xi$ and the \textit{tent} over an open set $O$ as
\[
T(O) = \Omega \backslash R_\alpha(F).
\]
\vv

{Let us now} state some lemmas that will be necessary in the proof.
\begin{lemma}
	\label{lemma:lemma2_in_CMS}
	Given $\alpha>0$, there exists $\tau \in(0,1)$ close to 1 such that if $F \subset \partial \Omega$ and $\Phi$ is a non-negative measurable function in $\Omega$, then
	\[
	\int_{R_\alpha({F}_\tau^*)} \Phi(x) \delta_\Omega(x)^{n} \,d x \leq C_\alpha \int_{{F}}\left(\int_{\gamma_\beta(\xi)} \Phi(x) \,d x\right) d \sigma(\xi),
	\]
	where $\beta=\min \{1, \alpha\}$.
\end{lemma}
{\begin{proof}
For the proof, see \cite[Lemma 3.14]{MPT2} (or \cite[Lemma 2]{CMS} in the classical case).
\end{proof}
}

\vv

\begin{lemma}
	\label{lemma:tent_sublevelset_closed}
	{If $f \in L^\infty_c(\Omega)
	$ and $\tau>0$,} the set $F = \{\xi\in \pom\, |\, \mathcal A(f)(\xi)\leq \tau\}$ is closed.
\end{lemma}
	{\begin{proof}
	See \cite[Lemma 4.6]{MPT2} for the proof.
	\end{proof}
	}
	
	\vv

Finally, we present the proof of Theorem \ref{thm:atomic_decomposition_tent}, which is inspired by the  original proof of \cite[Theorem 1]{CMS} but with some modifications originated by the lack of structure of our domains. {In particular,  any open subset $O$ of $\R^{n}$ admits a ``very good" Whitney decomposition in disjoint cubes $(Q_j)_j$, and then its tent $T(O) \subset \R^{n+1}_+$ can be written as the disjoint union of $\bigsqcup_{j} (Q_j \times \mathbb R)\cap T(O)$. In our case, we have to do extra work and obtain $T(O)$ as a union of sets with bounded overlapping.}
\begin{proof}[Proof of Theorem \ref{thm:atomic_decomposition_tent}]
	Let $\tau$ sufficiently close to $1$, and for $k\in\mathbb Z$ let $O_k := \{x \in \pom\,|\, \mathcal A(f)>2^k\}$. By {Lemma \ref{lemma:tent_sublevelset_closed}}, we have that $O_k$ is open and let $O_k^* = \{x \in \pom \, |\, M_{c,\sigma}(\chi_{O_k})>1-\tau\}$.
	Then, $O_k \subset O^*_k$, but, by the boundedness of the Hardy-Littlewood maximal operator, it holds that $\sigma(O^*_k)\leq c_\tau \sigma(O_k)$.
	Let $(\lambda B_i^k)_i$ be a Whitney ball decomposition ({Theorem \ref{thm:whitney_balls}}) of $O_k^*$ as a subset of $\pom$.
	In particular, $O_k^* = \pom \cap\bigcup_i B_i^k$ and there exists $\lambda>1$ depending only on the Ahlfors regularity constant of $\pom$, such that $\lambda^2 B_i^k \cap (O_k^*)^c \neq \varnothing$ (inside of $\pom$).
	
	\begin{claim}
		There exists a constant $C>1$ depending on the constant $\lambda$ of the Whitney ball decomposition {Theorem \ref{thm:whitney_balls}} and the aperture of the cones $\alpha$ such that the tent
		\[
		T(O_k^*) \subset \Omega\cap \bigcup_i \left(C B_i^k \cap R_\alpha(B_i^k\cap\pom) \right),
		\] 
		if we consider the balls $B_i^k$ as balls in $\mathbb R^{n+1}$.
	\end{claim} 
	\begin{proof}[Proof of Claim 1]
		Let $x \in T(O_k^*)$, that is, $x \in \Omega \backslash R_\alpha((O_k^*)^c)$.  {By the definitions of} $R_\alpha$ and $\gamma_\alpha$, we have that 
		\[
		d(x, (O_k^*)^c) \geq (1+\alpha) d(x,\pom) = (1+\alpha) d(x,O_k^*).
		\]
		If we take $\xi\in O^*_k$ such that $d(x,\xi) \leq (1+\alpha/2) d(x,\pom)$, then, by triangular inequality we also have
		\[
		\frac{1+\alpha}{1+\alpha/2} d(x,\xi) \leq d(x, (O_k^*)^c) \leq d(x, \xi) + d(\xi,(O_k^*)^c).
		\]
		Hence, if $B_i^k$ is a Whitney ball containing $\xi$, it holds that
		\[
		d(x,\xi) \lesssim_\alpha d(\xi,(O_k^*)^c) \lesssim_\lambda r(B_i^k). 
		\]
		Thus, $x \in C_{\alpha,\lambda} B_i^k$. Moreover, $x \in \gamma_\alpha(\xi) \subset R_\alpha(B_i^k\cap\pom)$, proving the claim.
	\end{proof}
	The previous claim allows us to write $ T(O_k^*) \backslash T(O_{k+1}^*)$ as the following non-disjoint union of open sets
	\[
	T(O_k^*) \backslash T(O_{k+1}^*) = \bigcup_i  \left (C B_k^i \cap R_\alpha(B_k^i) \cap \left( T(O_k^*) \backslash T(O_{k+1}^*) \right) \right) =: \bigcup_i \Delta_i^k. 
	\]
	Note that $\supp f \subset \bigcup_{k\in\mathbb Z} T(O_k^*)$.	
	\begin{claim}
		The sets $\Delta_i^k$ have bounded overlapping, that is, there exists $N>0$ {depending only on $\lambda$ and $\alpha$} such that 
		\[
		\sum_i \chi_{\Delta_k^i}(x) \leq N ,\quad\mbox{for all }x \in T(O_k^*) \backslash T(O_{k+1}^*).
		\] 
	\end{claim}
	\begin{proof}[Proof of Claim 2]
		Assume $x \in \Delta_i^k \cap \Delta_j^k$. Then, there exist $\xi^k_i \in B^k_i \cap \pom$ and $\xi^k_j \in B^k_j \cap \pom$ such that $x\in\gamma_\alpha(\xi_i^k)\cap\gamma_\alpha(\xi_j^k)$, which implies
		\[
		(1+\alpha)^{-1}d(x,\xi^k_i) \leq  d(x, \xi^k_j) \leq (1+\alpha) d(x,\xi^k_i),
		\]
		and 
		\[
		d(x,\xi^k_i) \lesssim_{\alpha,\lambda} r(B^k_i), \quad \mbox{ and }\quad d(x,\xi^k_j) \lesssim_{\alpha,\lambda} r(B^k_j).
		\]
		Hence, $d(B_i^k, B_j^k)\leq d(\xi^k_j,\xi_i^k)\lesssim_{\alpha,\lambda} \min(r(B^k_i), r(B^k_j))$. 
		
		As a consequence, it cannot happen that $r(B_i^k) \ll r(B_j^k)$ or vice versa since that would imply that $\lambda B_i^k \cap \lambda B_j^k \neq \varnothing$ and we get a contradiction with {property (iv) of Theorem \ref{thm:whitney_balls}}. Hence, if $x$ belongs to $\Delta_i^k$ for $i\in I$, we have that $r(B_i^k) \approx r(B_j^k)$ for all $i,j \in I$, and $d(B_i^k,B_j^k) \approx r(B_i^k)$. Therefore, it can only belong to a uniformly bounded number of balls.
		
	\end{proof}
	Let $(\phi_i^k)_i$ be a partition of unity subordinated to the sets $\Delta_i^k$. We write
	\[
	a^k_i := f \phi_i^k \sigma(B_i^k)^{-1/p} (\mu^k_i)^{-1/2}
	\]
	where $\mu_i^k := \int_{\Delta^k_i} |f(y) \phi_i^k(y)|^2 \, \frac{dy}{\delta_\Omega(y)}$. Clearly, we have that these functions $a_i^k$ are multiples of $T^p_2$ atoms. If we set $\lambda_i^k := \sigma(B^k_i)^{1/p} (\mu^k_i)^{1/2}$, we obtain
	\[
	f =\sum_{k,i} \lambda^k_i a^k_i.
	\]
	So it is only left to prove that
	\[
	\sum_{k,i} |\lambda_i^k|^p \lesssim \Vert \mathcal A(f) \Vert_p^p = \Vert f \Vert_{T^p_2}^p.
	\]
	However,
	\[
	\mu_i^k = \int_{\Delta^k_i} |f(y) \phi_i^k(y)|^2 \, \frac{dy}{\delta_\Omega(y)} \leq \int_{\Delta^k_i} |f(y)|^2 \, \frac{dy}{\delta_\Omega(y)} \leq \int_{CB^k_i \cap (O_{k+1}^*)^c} |f(y)|^2 \, \frac{dy}{\delta_\Omega(y)}.
	\]
	{If we apply} Lemma \ref{lemma:lemma2_in_CMS} with $F = (O_{k+1})^c,$ $F_\tau^* = (O_{k+1}^*)^c,$ $R_\alpha(F^*_\tau) = \Omega\backslash T(O_{k+1}^*)$, and $\Phi(y) = |f(y)|^2 \delta_\Omega(y)^{-(n+1)}\chi_{CB_i^k}(y)$,  {we get}
	\[
	\mu_i^k \lesssim \int_{CB^k_j \cap O_{k+1}^c} \mathcal A(f)^2\, \,d\sigma \lesssim \sigma(B_j^k) (2^{k+1})^2
	\]
	by the definition of $O_{k+1}$.
	Finally, {using the bounded overlapping} {property} of the balls $B^k_j$, we {infer}
	\begin{align*}
		\sum_{j,k} |\lambda_j^k|^p &\lesssim \sum_{k,j} \sigma(B^k_j) 2^{kp} 
		\lesssim \sum_k \sigma(O^*_k) 2^{(k+1)p} \\
		&\lesssim \sum_k \sigma(O_k) 2^{(k+1)p} \lesssim \Vert \mathcal A(f) \Vert_{L^p(\sigma)}^p.  
	\end{align*}
\end{proof}
\color{black}
\vvv

\subsection{Variants of the Dirichlet and Regularity problems}
\label{section:variants}
Let $\Omega \subset \R^{n+1}$ be a corkscrew domain with $n$-Ahlfors regular boundary $\pom$ such that either $\Omega$ is bounded or $\pom$ is unbounded, and $\sigma := \HH^n|_{\pom}$.

For any $p\in(1,\infty)$, we say that the \textit{Dirichlet problem in the tent space $T^p_2$} is solvable for the operator $\mathcal L$ (we write $(\wt D^{\mathcal L}_{p})$ is solvable) if there exists $C>0$ such that for all $f \in C_c(\pom)$, the solution $u$ of the continuous {Dirichlet problem} satisfies 
	\begin{equation}
	\Vert \delta_\Omega |\nabla u| \Vert_{T^{p}_2(\Omega)} \leq C \Vert f \Vert_{L^{p}(\sigma)}.
	\end{equation}

Consider the following Poisson problem
\begin{equation}
	\label{eq:definition_Poisson_pb}
	\begin{cases}
	-\mathcal L v=H-\operatorname{div} \boldsymbol{\Xi}, & \text { in } \Omega, \\  
	v \in Y^{1,2}_0(\Omega), 
	\end{cases}
\end{equation}
with Poisson data $H\in L^\infty_c(\Omega)$ and $\boldsymbol{\Xi} \in L^\infty_c(\Omega; \, \R^{n+1})$ where $Y_0^{1,2}(\Omega)$ is the completion of $C_c^\infty(\Omega)$ under the norm $\Vert u \Vert_{Y^{1,2}} := \Vert u \Vert_{L^{\frac{2(n+1)}{n-1}}(\Omega)}+\Vert \nabla u \Vert_{L^2(\Omega)}$. 

For any $p \in(1, \infty)$, we say that the \textit{Poisson regularity problem} is solvable for the operator $\mathcal L$ (we write $(PR^{\mathcal L}_{p})$ is solvable) if there exists $C>0$ such that for each $H \in L_c^{\infty}\left(\Omega\right)$ and $\boldsymbol{\Xi} \in$ $L_c^{\infty}\left(\Omega ; \R^{n+1}\right)$, the solution $v$ of the {Poisson problem \eqref{eq:definition_Poisson_pb}}  satisfies
	\begin{equation}
	\Vert\wt{{N}}_2(\nabla v)\Vert_{L^p(\sigma)} \leq C\left(\left\|\mathcal C_{2_*} (H) \right\|_{L^p(\sigma)}+\|\mathcal C_2(\boldsymbol{\Xi}/\delta_\Omega)\|_{L^p(\sigma)}\right) .
	\end{equation}
Analogously, for $p\in(0,\infty)$, we say that the \textit{{modified} Poisson regularity problem} is solvable for the operator $\mathcal L$ (we write $(\wt {PR}^{\mathcal L}_{p})$ is solvable) if there exists $C>0$ such that for each $H \in L_c^{\infty}\left(\Omega\right)$ and $\boldsymbol{\Xi} \in$ $L_c^{\infty}\left(\Omega ; \R^{n+1}\right)$, the solution $v$ of the {Poisson problem \eqref{eq:definition_Poisson_pb}} satisfies
	\begin{equation}\label{eq:def_mod_Poisson_reg}
		\Vert\wt{{N}}_2(\nabla v)\Vert_{L^p(\sigma)} \leq C\left(\left\|\delta_\Omega H \right\|_{T^p_2(\Omega)}+\|\boldsymbol{\Xi}\|_{T_2^p(\Omega)}\right) .
	\end{equation}
	For {$0< p \leq 1$}, we say that the \textit{Poisson regularity problem} $(PR^{\mathcal L}_p)$ is solvable if the {modified} Poisson regularity problem $(\widetilde{PR}^{\mathcal L}_p)$ is.

{The Poisson regularity problem was introduced  in \cite{MPT} and the modified  Poisson regularity problem in \cite{MZ}. In $\mathbb{R}^{n+1}_+$ and for small Carleson perturbations of $t$-independent operators, the non-tangential estimates of the latter problem for $\nabla L_{\mathbb{R}^{n+1}}^{-1}\operatorname{div}$ (i.e.  the operator with kernel the fundamental solution instead of the Green function) first appeared (without this name) in \cite[Propositions 4.6 and 5.1]{HMM}.  It is easy to see that once this estimate  is obtained, then by the standard connection between the fundamental solution and the Green function, one can prove that the modified Poisson regularity problem is equivalent to the solvability of the Regularity problem. }

{By \cite[Theorem 1.8]{MZ}, there is the following relation between the modified Poisson regularity problem and the Dirichlet problem in tent spaces.}
\begin{theorem}
	\label{thm:mod_reg_implies_mod_Dirichlet}
	Let $\Omega \subset \mathbb R^{n+1}$ be a corkscrew domain with $n$-Ahlfors regular boundary $\pom$. Then, for any  $p \in [1, 2]$ such that $\frac 1 p + \frac 1 {p'}=1$, solvability of $(\wt{PR}_p^{\mathcal L})$ implies solvability of $(\wt D_{p'}^{\mathcal L^*})$.
\end{theorem}
{Furthermore, we also have that solvability of the Dirichlet problem in tent spaces implies that the associated elliptic measure is in weak-$\mathcal A_\infty(\sigma)$.}
\begin{proposition}
	\label{prop:dirichlet_tent_implies_weak_A_infty}
	Let $\Omega \subset \R^{n+1}$ be a corkscrew domain with $n$-Ahlfors regular boundary $\pom$. If $(\wt D_{p'}^{\mathcal L})$ is solvable for some $p'\geq2$, then $\omega_{\mathcal L} \in \mbox{weak-}\mathcal A_\infty(\sigma)$.
\end{proposition}
{The proof is very similar to that of \cite[Theorem 1.4]{HL} so we will only highlight the main differences.}
\begin{proof}[Sketch of proof]
	We intend to show that $\omega_{\mathcal L}$ satisfies the hypotheses of Lemma \ref{lemma:weak_A_infty} (which corresponds to \cite[Lemma 3.2]{HL}).
	
	Let the sets $F, F_1$, and $A_1$ be as in their proof, and define
	\[
	f := \chi_{A_1}
	\]
	which satisfies $\Vert f \Vert_{L^p(\sigma)} = \sigma(A_1)^{1/p}$.	In order to work with continuous data, we mollify $f$ as in \cite[Lemma 3.9]{HL} and obtain a family of functions $(f_s)_s$. We have that $\sup_s \Vert f_s \Vert_{L^p(\sigma)} \lesssim \Vert f\Vert_{L^p(\sigma)}$ thanks to the boundedness of the Hardy-Littlewood maximal operator in $L^p(\sigma)$. 
	
	Let $u_s$ be the solution of the Dirichlet problem with boundary data $f_s$. Then, for small $\varepsilon>0$, we have by Fatou's lemma and \cite[Claim 2]{HL} that
	\begin{align*}
		\omega^x_{\mathcal L}(A_1) &\leq \int_\pom \liminf_{s\to0}f_s \, d\omega^x_{\mathcal L} \leq \liminf_{s\to0} u_s(x) \\
		&\leq \liminf_{s\to0} C_{\varepsilon}\left(\frac{1}{\sigma\left(\Delta_X\right)} \int_{B_x \cap \Omega}|\nabla u_s|^2 \delta_\Omega \,dm\right)^{1 / 2} + C \varepsilon^\alpha.
	\end{align*}
	Using Fubini, we obtain that
	\[
	\int_{B_x \cap \Omega} |\nabla u_s|^2 \delta_\Omega \,dm \lesssim \int_{CB_x\cap\pom} \mathcal A(\delta_\Omega|\nabla u_s|)^2 \, d\sigma.
	\]
	Hence, by solvability of $(\wt D_{p}^{\mathcal L})$ and H\"older's inequality, we get
	\begin{multline}
		\omega^x_{\mathcal L}(A_1) \leq \liminf_{s\to0}C_{\varepsilon}\left(\fint_{CB_x\cap\pom} \mathcal A(\delta_\Omega|\nabla u_s|)^p \, d\sigma\right)^{1 / p} + C \varepsilon^\alpha
		\lesssim C_{\varepsilon} \sigma(CB_x\cap\pom)^{-1/p} \Vert f \Vert_{p} + C\varepsilon^\alpha \\
		=C_{\varepsilon} \left(\frac{\sigma(A_1)}{\sigma(CB_x\cap\pom)}\right)^{1/p} + C\varepsilon^\alpha  \leq C_{\varepsilon}\eta^{1/p}  + C \varepsilon^\alpha.
	\end{multline}
	The rest of the proof is {the same as}  the one in \cite{HL} with $\eta^{1/p}$ instead of $\gamma$.
\end{proof}

\color{black}

\vvv

\section{Atomic decomposition of \haj-Sobolev spaces on Ahlfors regular metric spaces}
\label{section:hajlasz}

In this section we will prove atomic decompositions for the \haj-Sobolev spaces $M_\tau^{1,p}(X)$, $p\leq1$, defined on {a uniformly perfect metric space $X$ equipped with a doubling measure $\sigma$}. The proofs are inspired by the ones appearing in the works of \cite{BB, BD1}. The main differences are that {we consider} atoms in the \haj-Sobolev space as we do not assume a differentiable structure on $X$, we consider the full range $p \in (0,1]$, and we do not assume $X$ is non-compact.

Let $(X, \,d(\cdot,\cdot),\, \sigma)$ be a a uniformly perfect metric space $X$ equipped with a doubling measure $\sigma$ with uniform perfectness constant $\lambda\geq2$ (see Remark \ref{rmk:uniformly_perfect}). We start by defining adequate atoms to carry out the atomic decompositions.
 \begin{definition}[$(\tau,t,p)$-atom]
\label{def:atom}
A $\sigma$-measurable function $a$ is a $(\tau,t,p)$-atom for $p \in (0,+\infty)$, $t\in(p,+\infty]$, and $\tau \in (0,+\infty]$ if there exists a ball $B$ and $\nabla_H a\in D(u) \cap L^t(X)$ such that
\begin{itemize}
	\item[(a)] $\supp a \subset B$,
	\item[(b)] $\Vert \nabla_H a \Vert_{L^t(X)} \leq \sigma(B)^{\frac 1 t - \frac1 p}$,
	\item[(c)] $\Vert a \Vert_{L^t(X)} \leq \tau
	\sigma(B)^{\frac 1 t - \frac1 p}$.
\end{itemize}
 
\end{definition}
Note that in other references (such as \cite{BB,BD1, CW, DK1}) some cancellation property is required for the atoms, but this is not necessary for our purposes. This cancellation property is much more important in the case of non-Sobolev Hardy spaces.  
Also, observe that there is always some control on the $t$-norm of $a$ thanks to the second Poincar\'e inequality in Theorem \ref{thm:poincare_compact_supp}, even when $\tau = +\infty$ ({as long as $\lambda B \subsetneq X$}), and that a $(\tau,r,p)$-atom is a $(\tau,t,p)$-atom too for $t<r$ by H\"older inequality.

We will now prove a Calder\'on-Zygmund decomposition for Lipschitz functions in the \haj-Sobolev space $M_\tau^{1,p}(X)$ (also with $p\leq1$). This decomposition has some similarities to the ones in \cite{BB,BD1} for $p=1$, although these works use classical Sobolev atoms instead of \haj-Sobolev atoms. 

\begin{lemma}[$q$\color{black}-Calder\'on-Zygmund decomposition]
\label{lemma:CZ}
Let $f \in M_\tau^{1,p}(X)\cap \operatorname{Lip}(X)$, $\tau \in (0,+\infty]$, $\nabla_H f \in D(f) \cap L^p(X)$, $p \in (0, +\infty]$, $q\in (0,p]$, and $\alpha >  \left (C_X^p {\fint_X (\tau^{-p}|f|^p + |\nabla_H f|^p)\, d\sigma}\right)^{1/p}$ for some constant $C_X$ depending on the Ahlfors regularity of $X$. 
Then, there exists a constant $C>0$, a collection of balls $(B_i)_i$, points $(x_i)_i$, {Lipschitz} functions $g$, $(b_i)_i$, and $(\varphi_i)_i$, and $L^q(X)$ functions $(\nabla_H b_i)_i$ such that 
\begin{gather}
    f - g = \sum_i b_i,\quad \mbox{$\sigma$-a.e. and in $M_\tau^{1,q}$};\\
    \bigcup_i B_i = \{x \in X\,|\, M(\tau^{-q}|f|^q+|\nabla_H f|^q)(x)>\alpha^q\} =: U_\alpha; \\
  {  g(x)=f(x) \mbox{ for $x\in U_\alpha^c$},~ |g|\leq \tau\alpha, ~|g|_{\operatorname{Lip}} \leq C\alpha;} \label{CZproperty2}\\ 
    (\varphi_i)_i \mbox{ forms a partition of unity of $U_\alpha$ subordinated to } (B_i)_i\mbox{ with $\|\varphi_i \|_{\operatorname{Lip}}\lesssim r(B_i)^{-1}$};\\
    x_i \in \lambda B_i \backslash U_\alpha,\, |f(x_i)|\leq\tau\alpha, \, |\nabla_H f(x_i)| \leq \alpha, \, b_i(x) := \left( f(x) - f(x_i) \right) \varphi_i(x),~ \forall i; \\
    \Vert b_i \Vert_{L^q(\sigma)} \leq C \tau\alpha \sigma(B_i)^{1/q},~ \Vert \nabla_H b_i \Vert_{L^q(\sigma)} \leq C \alpha \sigma(B_i)^{1/q},~ \nabla_H b_i \in D(b_i), ~\forall i;\label{CZproperty6}\\
    \sum_i \sigma(B_i) \leq  C{\alpha^{-p}}  \int_X \left( \tau^{-p}|f|^p + |\nabla_H f|^p \right) d\sigma ; \label{CZproperty7}\\
    \sum_{i} \chi_{B_i}(x) \leq N,~ \forall x\in X.
\end{gather}
The constant $C$ depends only on the Ahlfors regularity constants of $X$, and the constant $N$ is the one in Theorem \ref{thm:whitney_balls} (Whitney ball decomposition).
\end{lemma}
Before presenting the proof, we remark that no constant in the Calder\'on-Zygmund decomposition depends on the Lipschitz seminorm of $f$ and that the condition \newline $\alpha > \left (C_X^p {\fint_X (\tau^{-p}|f|^p + |\nabla_H f|^p )\,d\sigma}\right)^{1/p}$ is only meaningful in the case $\sigma(X)<+\infty$.\color{black}
\begin{proof}~
		Fix $f \in M_\tau^{1,p}(X)\cap\operatorname{Lip}(X)$, $\alpha>0$, and a \haj~derivative $\nabla_H f \in D(f)\cap L^p(X)$. We can redefine $\nabla_H f(x)$ as the $\limsup_{r \to 0} \fint_{B(x,r)} \nabla_H f \, d\sigma$ (as this only changes it in a set of measure zero) and by \cite[Lemma 10.2.2]{HKST} this new \haj{} {gradient} satisfies \eqref{eq:def_haj_grad} pointwise.\color{black}{}
		Let
		\[
		U_\alpha := \{x \in X\,|\, M(\tau^{-q}|f|^q+|\nabla_H f|^q)(x)>\alpha^q\}
		\]
		{where $M$ is the uncentered Hardy-Littlewood maximal function with respect to the measure $\sigma$.} The properties of the uncentered Hardy-Littlewood maximal operator $M$ (in particular the weak $(p,p)$ inequality for $M(|\cdot|^q)^{1/q}$ for $q\leq p$) give us
		\begin{align}
		\label{eq:measure_Ualpha}
			\sigma(U_\alpha) \leq C_X^p \alpha^{-p} (\tau^{-p}\Vert f\Vert_p^p + \Vert \nabla_H f \Vert_p^p) <+\infty 
		\end{align}
		since $f$ belongs to $M_\tau^{1,p}$ and $p\geq q$.
		
		\subsubsection*{Case $U_\alpha=\varnothing$.} We may choose $g=f$, $b_i=0$, and we already get a Calder\'on-Zygmund decomposition thanks to Lebesgue's differentiation theorem.
		
        \subsubsection*{Case $U_\alpha = X$.}
		This cannot happen thanks to the condition $\alpha >  \left (C_X^p {\fint_X (\tau^{-p}|f|^p + |\nabla_H f|^p )\,d\sigma}\right)^{1/p}$.

		\subsubsection*{Case {$\varnothing \neq U_\alpha \subsetneq X$}.} Consider a Whitney ball decomposition of $U_\alpha$ (as in Theorem \ref{thm:whitney_balls}) with balls $B_i$. We denote $\wt B_i = \lambda B_i$, a multiple of the ball $B_i$ given by the Whitney ball decomposition that satisfies $\wt B_i \cap U_\alpha^c \neq \varnothing$. Notice that we have already proved the two last properties of the statement: bounded overlapping and
		\[
		\sum_i \sigma(B_i) \leq N \sigma(U_\alpha) \leq  C_X N \alpha^{-p} (\tau^{-p}\Vert f\Vert_p^p + \Vert \nabla_H f \Vert_p^p) 
		\]
		thanks to \eqref{eq:measure_Ualpha}.
		
		Let's define the ``bad" functions $b_i$ of the decomposition. Consider a partition of unity $(\varphi_i)_i$ of $U_\alpha$ such that $\varphi_i$ is supported in $B_i$, and $\varphi_i$ is Lipschitz with $ |\varphi_i|_{\operatorname{Lip}} \lesssim \frac 1 {r(B_i)}$ for all $i$.
		Note that, by definition of the maximal operator $M$, we have
		\[
		\fint_{\wt B_i} \left(\tau^{-q}|f|^q + |\nabla_H f|^q \right)d\sigma \leq \alpha^q
		\]
		since the ball $\wt B_i$ intersects the complementary of $U_\alpha$. 
		This implies that there exists $x_i\in \wt B_i \backslash U_\alpha$ such that $\tau^{-q}|f(x_i)|^q + |\nabla_H f(x_i)|^q \leq \alpha^q$. 
		For such $x_i$, set
		\[
		b_i =  ( f - f(x_i)) \varphi_i.
		\]
		Clearly, $\supp b_i \subset B_i$. Using Leibniz' rule (Lemma \ref{lemma:leibniz}), we have a \haj~ gradient for $b_i$ given by 
		\begin{equation}
		\label{eq:haj_grad_bi}
		\nabla_H b_i(x) = \left (\frac{C}{r(B_i)} |f(x)- f(x_i)| + \nabla_H f(x) \right)\chi_{B_i}(x) \lesssim (\nabla_H f(x) + \nabla_H f(x_i))\chi_{B_i}(x)
		\end{equation}
		for some constant $C>0$.
		
		Now, we proceed to compute the $q$-norms of $b_i$ and $\nabla_H b_i$. We have
		\begin{align*}
			\int_X |b_i|^q d\sigma &\leq \int_{B_i} |f(x) - f(x_i)|^q d\sigma \\
			&\lesssim \int_{\wt B_i} \left(|f(x)|^q + |f(x_i)|^q \right)d\sigma 
			\lesssim \tau^q\alpha^q \sigma(\wt B_i) 
			\lesssim \tau^q\alpha^q \sigma(B_i)
		\end{align*}
		by the definition of $U_\alpha$ and properties of the uncentered maximal operator, doubling of the measure, and the definition of $x_i$.
		Also using the expression of $\nabla_H b_i$ considered in \eqref{eq:haj_grad_bi}, we can bound
		\begin{align*}
			\int_{X} |\nabla_H b_i|^q \,d\sigma &\lesssim \frac{1}{r(B_i)^q} \int_{B_i} |f(x)- f(x_i)|^q d\sigma(x)+ \int_{B_i}|\nabla_H f|^q d\sigma \\
			&\lesssim \int_{B_i} \left(|\nabla_H f(x)|^q + |\nabla_H f(x_i)|^q \right) d\sigma(x) \\
			&\leq \int_{\wt B_i} |\nabla_H f(x)|^q d\sigma +\alpha^q \sigma(B_i)
			\lesssim \alpha^q \sigma(B_i)
		\end{align*}
    using the \haj~gradient property and the same ideas as before.
	
	Now, we only have left to check the behavior of the ``good" part $g$. We have
	\[
	g = f \cdot \chi_{U_\alpha^c} + \sum_i f(x_i) \varphi_i,
	\]
	with $\supp \left(\sum f(x_i) \varphi_i\right) \subseteq U_\alpha$.
	The bound $|g|\leq\tau\alpha$ follows easily. In $U_\alpha^c$, it is a consequence of Lebesgue's differentiation theorem and the definition of $U_\alpha$. On $U_\alpha$, it is due to the fact that $|f(x_i)|\leq\tau\alpha$ and $\sum \varphi_i(x) = \chi_{U_\alpha}(x)$.

	The part that comes now is different from the proof in \cite{BD1} as that work does not use Lipschitz extensions.
	{We know that for $x\in U_\alpha^c$, we have $\nabla_H f(x)\leq \alpha$ and $\nabla_H f$ satisfies \eqref{eq:def_haj_grad} pointwise.} Hence, we may consider a Lipschitz extension $\tilde f$ of $f|_{U_\alpha^c}$ to $X$ with Lipschitz constant $2\alpha$.
	Then, for any $z\in X$ and $x \in U_\alpha$, we may rewrite 
	\begin{equation}
	\label{eq:expression_good_function}
	g(x) = \tilde f(z) +  \sum_i (\tilde f(x_i) - \tilde f(z) ) \varphi_i(x)
	\end{equation}
	since $f(x_i) = \tilde f(x_i)$ as $x_i \in U_\alpha^c$. Moreover, 
	\eqref{eq:expression_good_function} also holds for all $x\in X$ if we let $z=x$ in the expression. Our aim is to check that $g$ is Lipschitz with constant $C\alpha$ for some $C>0$.
	To this end, we will consider four different cases: $x,y \in U_\alpha$ and they are ``close", $x,y \in U_\alpha$ and they are ``far", $x\in U_\alpha$ and $y \in U_\alpha^c$, and $x,y\in U_\alpha^c$.
	
	\subsubsection*{Case $x,y \in U_\alpha$ and they are ``close".}
	Let $I_x := \{ i \,|\, x\in B_i\}$ for any $x \in U_\alpha$. Assume there exist $i_0 \in I_x$ and $j_0 \in I_y$ such that $ B_{i_0} \cap  B_{j_0} \neq \varnothing$ and fix $z\in B_{i_0} \cap B_{j_0}$. Then, the radii $r(B_k)$ for $k\in I_x \cup I_y$ are all comparable. 
	We have
	\begin{align}
	\label{eq:lip_constant_g}
	|g(x) - g(y)|  &= \left |\sum_{i\in I_x} (\tilde f(x_i) - \tilde f(z)) \varphi_i(x) - \sum_{i\in I_y} (\tilde f(x_i) - \tilde f(z)) \varphi_i(y) \right| \nonumber\\
	&=  \left |\sum_{i \in I_x \cup I_y} (\tilde f(x_i) - \tilde f(z)) (\varphi_i(x)-\varphi_i(y))  \right| \\
	&\leq C\alpha d(x,y)\sum_{i \in I_x \cup I_y}  \frac{d(x_i,z)}{r(B_i)}. \nonumber
	\end{align}
	Thanks to the properties of the Whitney ball decomposition, $\#I_x \leq N$ for all $x \in U_\alpha$ and, for all $i,k\in I_x$ the radius of $B_i$ and $B_k$ are comparable to $d(x,U_\alpha^c)$.
	Using that $z\in  B_{i_0} \cap  B_{j_0}$, we obtain $d(x_i,z) \lesssim r(B_i)$ for $i\in I_x \cup I_y$ . Using this in \eqref{eq:lip_constant_g} gives us the desired bound for $|g(x)-g(y)|$ in this case.
	
	\subsubsection*{Case $x,y \in U_\alpha$ and they are ``far".}
 	Assume that for all $i\in I_x$ and $j \in I_y$ we have that $ B_i \cap  B_j = \varnothing$. 
	In this case, we have $d(x,y)\gtrsim r(B_i)$ for all $i\in I_x \cup I_y$.
	Then, we write
	\begin{equation}
	\label{eq:proof_CZ_case_far}
	|g(x) - g(y)| \leq \sum_{i\in I_x} |\tilde f(x_i) - \tilde f(z)|  + \sum_{j\in I_y} |\tilde f(x_j) - \tilde f(z)|.
	\end{equation}
	Choose $z = x_{j_0}$ for some fixed $j_0\in I_y$. Then, we can bound
	\[
	\sum_{j\in I_y} |\tilde f(x_j) - \tilde f(z)| \leq \sum_{j\in I_y} 2 \alpha d(x_{j},z) \lesssim N\alpha d(x,y)
	\]
	since $d(x_j,z)$ is comparable to $r(B_j)$ for $j\in I_y$ (and the constant $N$ comes from the bounded overlapping). We can control the other term in the right hand side of \eqref{eq:proof_CZ_case_far} by using that $d(x_i,z)\leq {d(x_i,x)} + {d(x,y)} + {d(y,z)} \lesssim r(B_i)+d(x,y) + r(B_{j_0}) \lesssim 3d(x,y)$ so we also get the desired Lipschitz bound. 
	
	\subsubsection*{Case $x,y\in U_\alpha^c$.} 
	If $x,y\in U_\alpha^c$ then it's trivial to see that $|g(x)-g(y)| = |f(x) - f(y)| \leq 2\alpha$ because there $f$ coincides with $\tilde f$ which is globally $2\alpha$-Lispchitz.
	 
	\subsubsection*{Case $x\in U_\alpha$ and $y \in U_\alpha^c$.}
	Using that $\varphi_i(y)=0$ for all $i$, we take $z=x$ in \eqref{eq:expression_good_function} and get
	\begin{align*}
	|g(x) - g(y)| &\leq |\tilde f(x) - \tilde f(y)| + \left| \sum_{i \in I_x} (\tilde f(x_i) - \tilde f(x))\varphi_i(x) \right| \\
	&\lesssim |\tilde f(x) - \tilde f(y) | + 2\alpha \sum_{i \in I_x} r(B_i).
	\end{align*}
	Since we have $r(B_i) \lesssim d(x,y)$ for $i\in I_x$ and $\#I_x \leq N$, we obtain the Lipschitz bound for the final case.
\end{proof}
\begin{remark}
	The expression for the good part of decomposition $g$ appearing in \eqref{eq:expression_good_function} does not depend on the Lipschitz extension $\tilde f$ chosen of $f|_{U_\alpha^c}$. This will be used in the proof of the next theorem. 
\end{remark}

Using the previous Calder\'on-Zygmund decomposition, we are ready to obtain an atomic decomposition for Lipschitz functions in $M_\tau^{1,p}(X)$ for $p\leq1$. The proof presents similarities with \cite[Propositions 3.4 and 4.7]{BD1} for $p=1$ but with {some marked differences such as the fact that we avoid the use of Sobolev inequalities and thus obtain $(\tau, \infty, p)$-atoms in the decomposition}.
\begin{theorem}[Atomic decomposition]
\label{thm:atomic_decomposition}
Let $p\in (0,1]$, $\tau\in(0,+\infty]$, $f\in M_\tau^{1,p}(X)\cap\operatorname{Lip}(X)$, and $\nabla_H f \in D(f) \cap L^p(X)$. There exists a constant $C>0$, a sequence $(a_j)_j$ of {Lipschitz} $(\tau,\infty,p)$-atoms, and a sequence $(\lambda_j)_j$ of real numbers such that 
\[
f = \sum_j \lambda_j a_j \quad\mbox{in $M_\tau^{1,p}(X)$ }\color{black}\quad\text{ and }\quad  \sum_j |\lambda_j|^p  \leq C \left ( \tau^{-p} \Vert f \Vert_p^p + \Vert \nabla_H f \Vert_p^p\right).
\] 
\end{theorem}
Again, before presenting the proof, we remark that no constant in the atomic decomposition depends on the Lipschitz seminorm of $f$.\color{black}

\begin{proof}{}~
Choose $0<q<p$ and consider the sequence of Calder\'on-Zygmund decompositions with power $q$ (Lemma \ref{lemma:CZ}) of $f$ at heights $\alpha_j = 2^j$, for $j\in\mathbb Z$. We obtain sequences of functions $g^j$ and $b_i^j$, balls $B_i^j$, with $\supp b_i^j \subset B_i^j$, and open sets $U_{\alpha_j} = \cup_{i} B_i^j$. If the decomposition at height $\alpha_j$ does not exist (which can happen in the case $\sigma(X)<+\infty$), we define $g^j=0$ and $U_{\alpha_j} = X$. In any case, there exists $j_0$ big enough such that for all $j>j_0$ the $q$-Calder\'on-Zygmund decomposition exists at height $\alpha_j$.
	
Note that we have
	\begin{equation}
	\label{eq:geometric_sum_U_alpha}
	+\infty> \tau^{-p} \Vert f \Vert_p^p + \Vert \nabla_H f \Vert_p^p\gtrsim\int_X M(\tau^{-q}|f|^q + |\nabla_H f|^q)^{p/q} d\sigma \approx \sum_{j \in \mathbb Z} 2^{pj} \sigma(U_{\alpha_j}\backslash U_{\alpha_{j+1}})
	\end{equation}
	by definition of $U_{\alpha_j}$, and the strong $(p,p)$ bound for $M(|\cdot|^q)^{1/q}$ as $p>q$. Using that 
	\[
	2^{pj} = \frac{2^p -1}{2^p}\sum_{k=-\infty}^{j} 2^{pk}, 
	\]
	we can rewrite
	\begin{align}
	\label{eq:atoms_rearrangement}
		+\infty>\sum_{j \in \mathbb Z} 2^{pj} \sigma(U_{\alpha_j}\backslash U_{\alpha_{j+1}}) &= 
		\frac{2^p -1}{2^p}\sum_{j \in \mathbb Z} \left (\sigma(U_{\alpha_j}\backslash U_{\alpha_{j+1}}) \sum_{-\infty<k\leq j}  2^{pk} \right) \nonumber \\
		&= \frac{2^p -1}{2^p} \sum_{k\in\mathbb Z} \left(  2^{pk} \sum_{j \geq k} \sigma(U_{\alpha_j}\backslash U_{\alpha_{j+1}}) \right) \\
		&= \frac{2^p-1}{2^p} \sum_{k\in\mathbb Z}  2^{pk} \sigma(U_{\alpha_k})\nonumber
	\end{align}
	by rearranging the sums and using that $\sigma(U_{\alpha_j}) \to 0$ as $j\to\infty$.
	
	First we will check that $g^j \to f$ in $L^p(X)$ as $j \to \infty$ (if $\tau<+\infty$). Since $(g^j - f)$ vanishes in $U_{\alpha_j}^c$, we have
	\[
	\Vert g^j - f \Vert_p^p= \int_{U_{\alpha_j}} |g^j - f|^p d\sigma \lesssim_p
	\int_{U_{\alpha_j}} (|g^j|^p + |f|^p) d\sigma
	\lesssim \tau2^{jp} \sigma(U_{\alpha_j}) + \int_{U_{\alpha_j}} |f|^p d\sigma.
	\]
	As the sum in \eqref{eq:atoms_rearrangement} converges, the sequence  $2^{pj} \sigma(U_{\alpha_j})$ goes to $0$ when $j\to+\infty$.
	Also, since $f\in L^p(X)$ and $\sigma(U_{\alpha_j})\to0$, we have that $\int_{U_{\alpha_j}} |f|^p d\sigma \to 0$ as well.

	We will see that (if $\tau<+\infty$) $g^j \to 0$ in $L^p(X)$ as $j\to -\infty$. Taking into account that  $g^j = f$ in $U_{\alpha_j}^c$ and $|g^j| \leq \tau2^j$, we have
	\begin{equation}
		\label{eq:g^j_goes_to_0}
	\int_X |g^j|^p \,d\sigma \lesssim \int_{U_{\alpha_j}^c} |f|^p \,d\sigma + \sigma(U_{\alpha_j}) \tau2^{jp}.
	\end{equation}
	Again, $\sigma(U_{\alpha_j})2^{jp}$ goes to zero as $j\to-\infty$ because the sum in \eqref{eq:atoms_rearrangement} is convergent.
	On the other hand, by the dominated convergence theorem
	\[
	\int_{U_{\alpha_j}^c} |f|^p d\sigma \leq \int_X |f|^p \chi_{\{|f|\leq \alpha_j\}} d\sigma \to 0~ \mbox{ as $j\to-\infty$}
	\]
	as $|f|\leq \alpha_j$ in $U_{\alpha_j}^c$.

	Now, we will check that the \haj~gradient $\nabla_H (g^j-f)$ given by $\sum_i \nabla_H b_i^j$ goes to $0$ in $L^p(X)$ as $j \to \infty$. We compute
	\begin{align*}
	\int_X |\nabla_H (g^j - f)|^p d\sigma &\leq \sum_i
	\int_{X} |\nabla_H b_i^j|^p d\sigma.
	\end{align*}
	Using the bound for $\nabla_H b_i^j$ appearing in \eqref{eq:haj_grad_bi}, we obtain
	\[
	\sum_i
	\int_{X} |\nabla_H b_i^j|^p d\sigma \lesssim \sum_i\int_{B_i} \left( 2^{jp} + |\nabla_H f|^p \right) d\sigma \lesssim 2^{jp}\sigma(U_{\alpha_j}) + \int_{U_{\alpha_j}} |\nabla_H f|^p\, d\sigma,	
	\]
	where we have used the finite superposition of the balls $B_i$.
	Note that the term $\int_{U_{\alpha_j}} |\nabla_H f|^p\, d\sigma$ goes to $0$ because $|\nabla_H f|^p \in L^1(X)$, $\sigma(U_{\alpha_j}) \to 0$, and $2^{jp} \sigma(U_{\alpha_j})$ tends to $0$ when $j\to+\infty$ because the sum in \eqref{eq:atoms_rearrangement} is convergent.
	
	Finally, one can check that the $\dot M^{1,p}(X)$ quasinorm of $g^j$ goes to $0$ as $j \to -\infty$ by the same argument used in \eqref{eq:g^j_goes_to_0} to see that $g^j\to 0$ in $L^p(X)$ as $j\to-\infty$, taking into account that $|\nabla_H g^j| \lesssim 2^j$ instead of $|g^j| \leq \tau 2^j$. 
		
	\subsubsection*{Definition of the atoms.} We define $l^j := g^{j+1} - g^j$ (with support on $U_{\alpha_{j}}$) and $l^j_i = l^j \varphi_i^j$ where $\varphi_i^j$ is a Lipschitz partition of unity subordinated to $(B_i^j)_i$ with the property that each $\varphi_i^j$ has a \haj~ gradient satisfying $|\varphi_i^j|_{\operatorname{Lip}} \lesssim r(B^j_i)^{-1}$.
	In the case $\sigma(X)<+\infty$, $l^j \equiv 0$ for all $j<j_0$ for some $j_0 \in \mathbb Z$. Then, for $l^{j_0}$ we just define $l^{j_0}_1 = l^{j_0}$ (we consider $B^{j_0}_1 = X$ as a single ball).
	
	\begin{claim*}
	The functions $l_i^j$ are multiples of $(\tau,\infty,p)$-atoms. In particular, \[\ \Vert l_i^j \Vert_{M_\tau^{1,\infty}(X)}\lesssim 2^{j+1}.\]
	\end{claim*}
	\begin{proof}[Proof of claim]
	Clearly, we have $\supp l_i^j \subset B_i^j$, and
	\[
	\Vert l_{i}^j \Vert_\infty \leq C\tau 2^j, 
	\]
	Also, by Leibniz' rule (Lemma \ref{lemma:leibniz}), we can consider the following \haj~ gradient for $l_i^j$:
	\[
	\nabla_H l_i^j := C \chi_{B_i^j} \left ( \nabla_H g^{j+1} + \nabla_H g^j + r(B_i^j)^{-1} |l^j| \right).
	\]
	We trivially have the bound
	\[\nabla_H g^{j+1} + \nabla_H g^j \lesssim 2^{j+1}.\]
	Hence, we only need to show that $r(B_i^j)^{-1} |l_i^j| \lesssim 2^{(j+1)}$.
	First, notice that $U_{\alpha_j} \supset U_{\alpha_{j+1}}$.
	This allows us to reuse the expression in \eqref{eq:expression_good_function} for both $g^j$ and $g^{j+1}$. 
	In particular, we have the following identities:
	\[
		g^j(x) = \tilde f(x) +  \sum_k (\tilde f(x_k^j) - \tilde f(x) ) \varphi_k^j(x), \quad g^{j+1}(x) = \tilde f(x) +  \sum_{\wt k} (\tilde f(x_{\wt k}^{j+1}) - \tilde f(x) ) \varphi_{\wt k}^{j+1}(x),
	\] 
	where $x_k^j \in \lambda B_k^j \backslash U_{\alpha_j}$, $\,\,x_{\wt k}^{j+1} \in \lambda B_{\wt k}^{j+1} \backslash U_{\alpha_{j+1}}$, and $\tilde f$ is any Lipschitz extension of $f|_{U^c_{\alpha_{j+1}}}$ to $X$ with Lipschitz constant at most $2^{j+2}$. These representations follow from the fact that $(\varphi_i^j)_i$ form a partition of unity and the definition of the bad functions $b_i^j$ (analogously with $j+1$ instead of $j$).  
	Hence, for $x \in B_i^j$, we obtain
	\begin{align*}
		|l^j(x)|  &\leq |g^{j+1}(x) - g^j(x)| \\
		&= \left|\sum_{{\wt k}} (\tilde f(x_{\wt k}^{j+1}) - \tilde f(x)) \varphi_{\wt k}^{j+1}(x) - \sum_k(\tilde f(x_k^j) - \tilde f(x)) \varphi_k^j(x) \right|\\
		&\leq \sum_{B_{\wt k}^{j+1}\cap B_i^j\neq\varnothing} |\tilde f(x) - \tilde f(x_{\wt k}^{j+1})|\chi_{B_{\wt k}^{j+1}} + \sum_{B_i^j\cap B_k^j\neq\varnothing} |\tilde f(x) - \tilde f(x_k^j)|\chi_{B_k^j}.
	\end{align*}
	Since $\tilde f$ has Lipschitz constant bounded by $2^{j+2}$,
	we have for $x\in B_i^j$ that $|\tilde f(x) - \tilde f(x_k^j)| \leq  2^{j+2}d(x ,x_k^j) \lesssim  2^{j+2} r(B_i^j)$ where we have used that $B_k^j \cap B_i^j \neq \varnothing$ implies $r(B_k^j) \approx r(B_i^j)$. Similarly, we also have for $x\in B_i^j$ that $|\tilde f(x) - \tilde f(x_{\wt k}^{j+1})| \leq  2^{j+2}d(x , x_{\wt k}^{j+1}) \lesssim  2^{j+2} r(B_j^i)$ using that (for a fixed $i$) the balls $B_{l}^{j+1}$ with $B_{i}^j \cap B_l^{j+1}\neq\varnothing$ must have radii $r(B_l^{j+1})\lesssim r(B^j_i)$ (which is a consequence of $U_{\alpha_{j+1}} \subset U_{\alpha_j}$ and the Whitney ball decomposition in Theorem \ref{thm:whitney_balls}).
	Finally, using the bounded overlapping of the balls in the same generation, we get
	\[
	r(B_i^j)^{-1} |l_i^j| \lesssim 2^{(j+1)}
	\]
	which finishes the proof of the claim.
	\end{proof}
	In particular, $\left (\sigma(B_i^j)^{1/p}C 2^{j}\right)^{-1}l_i^j =: \frac{1}{\mu_{i}^j} l_i^j$ is a $(\tau,\infty,p)$-atom. 

\subsubsection*{Convergence of the atoms.} 
	Now, we want to check that the $\dot M^{1,p}(X)$ norm of the partial sums $l^j - \sum_{i}^K l_i^j$ tends to $0$. We can take as \haj~gradient of the partial sums the function $\sum_{i=K+1}^{+\infty} \nabla_H l_i^j$. Note that we have
	\[
	\left\Vert \sum_{i=K+1}^{+\infty} \nabla_H l_i^j \right\Vert_p^p \leq \sum_{i=K+1}^{+\infty} \Vert\nabla_H l_i^j\Vert_p^p
	\]
	but this corresponds to the tail of the following convergent series
	\[
	\sum_i \Vert\nabla_H l^j_i\Vert_p^p \lesssim \sum_i 2^{jp} \sigma(B^j_i) \lesssim 2^{jp} \sigma(U_{\alpha_j}) < +\infty.
	\]
	Note also that the term $2^{jp} \sigma(U_{\alpha_j})$ converges to $0$ since it corresponds to one of the tails in the right hand side of \eqref{eq:atoms_rearrangement}.
	The argument for the convergence of $\sum_i l_i^j \to l^j$ in $L^p(X)$ (in the case $\tau<+\infty$) is analogous.
	
Finally, we can set
\[
f = \sum_{i \in \mathbb N, ~j \in\mathbb Z} l_i^j
= \sum_{i \in \mathbb N, ~j \in\mathbb Z} \mu^j_i\left( \frac{l_i^j}{\mu^j_i} \right)	
\]
as a sum of $(\tau,\infty,p)$-atoms.
Then, we can bound the following sum
\begin{align*}
\sum_{i,j} |\mu_i^j|^p &=
 \sum_{i,j}|\sigma(B_i^j)^{1/p}C 2^{j}|^p\\
 &= 
 C'\sum_j {2^{pj}\left(\sum_i \sigma(B_i^j)\right)} \\
 &= C'\sum_j 2^{pj}\sigma(U_{\alpha_j})\\
 &\approx\int_X M(\tau^{-q}|f|^q + |\nabla_H f|^q)^{p/q} d\sigma\\
 &\lesssim \int_X \left(M(\tau^{-q}|f|^q)^{p/q} + M(|\nabla_H f|^q)^{p/q}\right) d\sigma \\
 &\lesssim C(\tau^{-p}\Vert f \Vert_p^p + \Vert \nabla_H f \Vert_p^p),
\end{align*}
where we have used \eqref{eq:atoms_rearrangement} and that the Hardy-Littlewood maximal operator is bounded in $L^{p/q}(X)$.
\end{proof}

\begin{remark}
	In the previous theorem we have obtained that every {Lipschitz} function $f\in M_\tau^{1,p}(X)$ can be written as a sum of $(\tau,t,p)$-atoms for $0<p\leq1$ and $t \in (p, +\infty]$.
	Note that the converse result is much simpler. Let
	\[
	f = \sum_{j} \lambda_j a_j,
	\]
	where $(a_j)_j$ are $(\tau,t,p)$-atoms and the coefficients $(\lambda_j)_j$ satisfy
	\[
	\sum_j |\lambda_j|^p  < +\infty.
	\]
	Then, taking into acccount that $\Vert a_j \Vert_p \leq 1$ by H\"older's inequality, we get
	\[
	\Vert f \Vert_p^p = \left\Vert \sum_{j} \lambda_j a_j \right\Vert_p^p \leq \sum_j |\lambda_j|^p \Vert a_j \Vert_p^p \leq \tau \sum_j |\lambda_j|^p < +\infty.
	\]
	Analogously, we also have
	\[
	\Vert \nabla_H f\Vert_p^p \leq \sum_j |\lambda_j|^p \Vert \nabla_H a_j \Vert_p^p \leq \sum_j |\lambda_j|^p < +\infty.
	\]
	Thus, $f$ belongs in $M^{1,p}_{\tau}(X)$ with norm bounded by $2 \sum_{j} |\lambda_j|^p$.
\end{remark}

\subsection{Interpolation in \haj-Sobolev spaces}

We will follow closely the proof presented in \cite[Theorem D]{CW}, which relies on the Calder\'on-Zygmund decomposition (Lemma \ref{lemma:CZ}).
\begin{theorem}[Interpolation of \haj-Sobolev spaces]
\label{thm:interpolation}
Let $0 < a < b \leq +\infty$, $\tau\in(0,+\infty]$, and $T$ be a continuous sublinear operator bounded from $M_\tau^{1,a}(X) \to L^{a,\infty}(X)$ and  bounded from $M_\tau^{1,b}(X)\to L^{b,\infty}(X)$.
Then $T$ is bounded on $M_\tau^{1,t}(X) \to L^t(X)$ for all $t \in (a,b)$.
\end{theorem}

\begin{proof}
{Before starting the proof, note that the space of Lipschitz functions $\operatorname{Lip}(X)$ is dense in $\dot M^{1,p}(X)$ for all $p>0$ (see \cite[Lemma 10.2.7]{HKST}). Hence, it is enough to show the result in the dense subspace $\operatorname{Lip}(X)\cap M_\tau^{1,t}(X)$.}
We will split the proof in {three cases depending on the diameter of the metric space $X$ and whether $b= +\infty$ or not.}
\subsubsection*{Case $\sigma(X)=+\infty$ and $b<+\infty$.}
	Fix a Lipschitz function $f \in M_\tau^{1,t}(X)$, and $\nabla_H f \in D(f)\cap L^t(X)$. Let $g_\alpha$ and $b_\alpha:=\sum_j b_{j,\alpha}$ be the functions in the Calder\'on-Zygmund decomposition (Lemma \ref{lemma:CZ}) of $f$ at height $\alpha>0$ with power $a$ in place of $p$.
	We write
	\begin{align*}
	t \int_0^{+\infty} \kappa^{t-1} \sigma(\{x \,|\, |Tf(x)|>\kappa\}) d\kappa \leq&~  t\int_0^{+\infty} \kappa^{t-1} \sigma(\{x \,|\, |Tg_\alpha(x)|>\kappa/2\})  d\kappa \\
	&+ t \int_0^{+\infty} \kappa^{t-1} \sigma(\{x \,|\, |Tb_\alpha(x)|>\kappa/2\}) d\kappa
	\end{align*} 
	where we have used the sublinearity of $T$.
	We will estimate each integral on the right hand side separately.
	
	First, for the bad part of the decomposition, we have
	\begin{equation}
	\label{eq:badpart_interpolation}
	\tau^{-a}\Vert b_\alpha \Vert_a^a + \Vert \nabla_H b_\alpha\Vert_a^a \lesssim 
	\sum_j \int_{B_{j,\alpha}} \left(\tau^{-a}|b_{j,\alpha}|^a + |\nabla_H b_{j,\alpha}|^a\right)\, d\sigma \lesssim \alpha^a \sum_j\sigma(B_{j,\alpha}) \lesssim \alpha^a \sigma(U_\alpha)
	\end{equation}
	by the properties of the Calder\'on-Zygmund decomposition. Then, by the boundedness of the operator $T$ from $M_\tau^{1,a}$ to $L^{a,\infty}$, recalling that $U_\alpha = \{x \, | \, M(\tau^{-a}|f|^a + |\nabla_H f|^a)(x) > \alpha^a  \}$, and choosing {$\alpha = \kappa$}, we get
	\begin{align*}
	\int_0^{+\infty} \kappa^{t-1} \sigma(\{x \,|\, |Tb_\kappa(x)|>\kappa/2\}) \,d\kappa &\lesssim \int_0^{+\infty} \kappa^{t-1} \frac{\tau^{-a}\Vert b_\kappa \Vert_a^a + \Vert \nabla_H b_\kappa\Vert_a^a}{\kappa^a}  \,d\kappa \\
	&\lesssim \int_0^{+\infty} \kappa^{t-1} \sigma(U_\kappa) \,d\kappa \\
	&= t^{-1}\Vert M(\tau^{-a}|f|^a + |\nabla_H f|^a)^{1/a} \Vert_t^t \\
	&\lesssim \tau^{-t}\Vert f \Vert_{t}^t + \Vert \nabla_H f \Vert_{t}^t
	\end{align*}
	using \eqref{eq:badpart_interpolation},  and that $M(|\cdot|^a)^{1/a}$ is strong $(t,t)$ bounded as $t>a$.
	
	As for the ``good" part, we are going to use that 
	\begin{align*}
	\Vert g_\kappa \Vert_{M_\tau^{1,b}}^b &\lesssim \int_{U_\kappa^c} (\tau^{-b}|f|^b + |\nabla_H f|^b )\,d\sigma + \kappa^b \sigma(U_\kappa)
	\end{align*}
	since $g_\kappa = f $ on $U_\kappa^c$, and the fact that, by Fubini, we have
	\begin{align*}
	\int_0^\infty \kappa^{t-1-b} \int_{\{|f| \leq \kappa\}} |f|^b \,d\sigma d\kappa &= 
	\int_X |f(x)|^b \int_{\{\kappa\geq |f(x)|\}} \kappa^{t-1-b} \,d\kappa d\sigma \\
	&= \int_X |f(x)|^b \frac{|f(x)|^{t-b}}{b-t} d\sigma \\
	&=  \frac{1}{b-t}\Vert f \Vert_t^t.
	\end{align*}
	Using that $U_\kappa^c \subset \{x \in X \,|\,|f|\leq \tau\kappa\} \cap \{x \in X \,|\,|\nabla_H f|\leq \kappa\}$ except for a set of zero measure, and the previous equation (for both $f$ and $\nabla_H f$), we get
	\begin{align*}
		\int_0^{+\infty} \kappa^{t-1} \sigma(\{x \,|\, |Tg_\kappa(x)|>\kappa/2\}) d\kappa &\lesssim  \int_0^{+\infty} \kappa^{t-1} \left (\frac{\Vert g_\kappa \Vert_{M_\tau^{1,b}}}{\kappa} \right)^b d\kappa \\
		&=  \int_0^{+\infty} \kappa^{t-1-b} \Vert g_\kappa \Vert_{M_\tau^{1,b}}^b \, d\kappa \\
		&\lesssim  \int_0^{+\infty} \kappa^{t-1-b} \left ( \int_{U_\kappa^c} \tau^{-b}|f|^b + |\nabla_H f|^b \,d\sigma + \kappa^b \sigma(U_\kappa) \right)\, d\kappa \\
		&\lesssim \tau^{-t}\Vert f \Vert_t^t + \Vert \nabla_H f \Vert_t^t + \int_0^{+\infty} \kappa^{t-1} \sigma(U_\kappa) \, d\kappa \\
		&\lesssim \tau^{-t}\Vert f \Vert_{t}^t + \Vert \nabla_H f \Vert_{t}^t
	\end{align*}
	where we have bounded $\int_0^{+\infty}\kappa^{t-1}\sigma(U_\kappa)$ as in the ``bad" part.

 \subsubsection*{Case $\sigma(X)<+\infty$.}
	The only difference is that the Calder\'on-Zygmund decomposition only exists for $\alpha > \alpha_0 = \left (C (\tau^{-t}\Vert f \Vert_{t}^t + \Vert \nabla_H f \Vert_{t}^t) / \sigma(X) \right )^{1/t}$ with $C$ depending on the weak $(1,1)$ norm of the uncentered maximal Hardy-Littlewood operator. Note though
	\[
	t\int_0^{\alpha_0} \kappa^{t-1} \sigma(X) d\kappa = \sigma(X) \alpha_0^t = C (\tau^{-t}\Vert f \Vert_{t}^t + \Vert \nabla_H f \Vert_{t}^t)
	\]
	and we may repeat the previous proof for $\alpha>\alpha_0$.

\subsubsection*{Case $b=+\infty$.}
Proceed as before, but instead of choosing $\alpha = \kappa$, choose $\alpha = \tilde C\kappa$ with $\tilde C$ small enough depending on the constant $C$ of the property \eqref{CZproperty2} {in the Calder\'on-Zygmund decomposition} ($|g_\alpha|_{\operatorname{Lip}} \leq C\alpha$), and the $L^\infty(\sigma)$ norm of $T$. With this choice, we have 
\[
\Vert Tg_\alpha \Vert_\infty <\kappa
\] 
and, as a consequence,
\[
t \int_0^{+\infty} \kappa^{t-1} \sigma(\{x \,|\, |Tg_\alpha(x)|>\kappa\}) \,d\kappa = 0.
\]
\end{proof}

\section{Regularity problem in the tangential Hardy-Sobolev space $HS^{1,1}$}
\label{section:tangential_sobolev}
In this section we will prove some relationships between the \haj-Sobolev space $\dot M^{1,1}(\pom)$ and the Hardy-Sobolev space $HS^{1,1}(\pom)$ when $\pom$ {supports the $(1,1)$-weak Poincar\'e inequality \eqref{eq:Poincare}}.  Moreover,  in {the} case  that $\HH^n(\pom \backslash \partial^*\Omega)=0$, we also compare the previous spaces to the Hardy HMT space $H^{1,1}(\pom)$. Similar results have already been shown in the case $p>1$ in \cite[Lemma 1.3]{MT},  {and} in the case $p=1$ in \cite{BD1, BD2} for Riemannian manifolds. We will use some of the results of Section \ref{section:hajlasz} with $X = \pom$ and $\sigma = \mathcal H^n|_{\pom}$.

\color{black}

The next proposition shows that, under the hypothesis that $\partial\Omega$ supports a weak $(1,1)$-Poincar\'e inequality \eqref{eq:Poincare}, the Hardy-Sobolev space $HS^{1,1}(\pom)$ coincides with the \haj-Sobolev space $\dot M^{1,1}(\pom)$.
\begin{proposition}
\label{prop:equiv_cald_maximal_hardy_space}
Let $E \subset \R^{n+1}$ be an {uniformly perfect $d$-rectifiable set} such that the measure $\sigma = \HH^d|_E$ is doubling.
Then,
\[
 \Vert f \Vert_{\dot M^{1,1}(\sigma)} \gtrsim \Vert  f \Vert_{{HS}^{1,1}(\sigma)}
\]
for all $f \in M^{1,1}(\sigma)\cap\operatorname{Lip}(E)$. 
Moreover, if $E$ supports a weak $(1,1)$-Poincar\'e inequality, then
\[
\Vert f \Vert_{\dot M^{1,1}(\sigma)} \lesssim \Vert  f \Vert_{{HS}^{1,1}(\sigma)}
\]
for all $f \in M^{1,1}(\sigma)\cap\operatorname{Lip}(E)$.
\end{proposition}
\begin{proof}
We will {first prove} the inequality ${\gtrsim}$. 
Let $f\in M^{1,1}(E)\cap\operatorname{Lip}(E)$. Thanks to the results in Section \ref{section:hajlasz}, we can find an atomic decomposition of $f$ in {Lipschitz} $(\infty,\infty,1)$-atoms $a_j$ belonging to $\dot M^{1,\infty}$ with the property that $f = \sum \lambda_j a_j$ {in $L^1(E)$} and $\Vert f \Vert_{\dot M^{1,1}}\leq\sum |\lambda_j| \leq 2 \Vert f \Vert_{\dot M^{1,1}}$. 
We will show that there exists a constant $c>0$ such that $c a_j$ is a Hardy-Sobolev atom (defined in Section \ref{section:definitions}). 
Since the $a_j$ are Lipschitz, we  have 
\[
|\nabla_t a_j(x)| \leq |a_j|_{\operatorname{Lip}} \lesssim \Vert \nabla_H a_j \Vert_{L^\infty(E)} \leq \sigma(B\cap E)^{-1} \quad \mbox{$\sigma$-a.e. $x\in E$}
\]
({see \cite[Section 11.2]{M}} for example).
Hence, $\Vert \nabla_t a_j \Vert_{L^\infty(E)}\lesssim \sigma(B\cap E)^{-1}$. Choosing the decomposition given by the $a_j$'s in the space $HS^{1,1}$ gives us the upper bound $\Vert f \Vert_{{HS}^{1,1}} \lesssim \Vert f \Vert_{\dot M^{1,1}}$.

For the reverse inequality ${\lesssim}$, we will show again that there exists a constant $c>0$ such that every Hardy-Sobolev atom $a$ satisfies that $c a$ is a \haj-Sobolev $(\infty, \infty, 1)$-atom. 
For the atom $a$, $|\nabla_t a|$ is an upper gradient of $a$ which belongs to $L^\infty(\pom)$, and since $\pom$ supports a weak $(1,1)$-Poincar\'e inequality \eqref{eq:Poincare}, we have that
\[
\int_{B\cap\pom} |a - a_{B}|\,d\sigma \leq C r(B) \int_{\Lambda B \cap \pom} |\nabla_t a|\,d\sigma, \quad \mbox{for any ball $B$ centered on $\pom$.}
\] 
Then, \cite[Theorem 8.1.7]{HKST} (i)$\implies$(iii) implies that there exists some $C>0$ such that $C M_c{\left(|\nabla_t a|\right)} \in D(a)\cap L^\infty(\pom)$.
Since the centered Hardy-Littlewood maximal operator is bounded in $L^\infty(\sigma)$, we have that there exists $c>0$ such that $c a$ is a  $(\infty,\infty,1)$-atom (see Definition  \ref{def:atom}).
\end{proof}

\begin{proof}[Proof of Theorem \ref{thm:solvability_classical_regularity_poincare}]
Theorem \ref{thm:solvability_classical_regularity_poincare} is a direct consequence of Proposition \ref{prop:equiv_cald_maximal_hardy_space}. 
\end{proof}

{ The next proposition shows that, in the case that $\partial^*\Omega$ supports a $(1,1)$-weak Poincar\'e inequality \eqref{eq:Poincare}, the Hardy HMT space $H^{1,1}(\partial^*\Omega)$ coincides with the \haj-Sobolev space $\dot M^{1,1}(\partial^*\Omega)$ via the characterization using $\Lambda_\sharp$ (see Theorem \ref{thm:calderon_maximal}).  
In particular, if we suppose that $\HH^n(\pom \backslash \partial^*\Omega)=0$, we could redefine solvability of $(\wt R_1^{\mathcal L})$ using Hardy HMT spaces instead of Hardy-Sobolev spaces. 

\begin{proposition}
	\label{prop:equiv_cald_maximal_hardy_HMT_space}
	Let $\Omega \subset \R^{n+1}$ be a domain of locally finite perimeter {such that $\partial^*\Omega$ supports a $(1,1)$-weak Poincar\'e inequality}. Then, if $\sigma_*:= \mathcal{H}^n|_{\partial^* \Omega}$, it holds that
	\[
 	\mathcal M_{\text{gr}}(\nabla_{\operatorname{HMT}}f)(x)  \approx  \Lambda_{\sharp} f(x), \quad \mbox{$\sigma_*$-a.e. $x \in \partial^*\Omega$}
	\]
	for all $f \in M^{1,1}(\sigma_*)\cap\operatorname{Lip}(\partial^*\Omega)$. 
	If we also assume that $\sigma_*$ is doubling, then it also holds that 
	$$
	\Vert \mathcal M_{\text{gr}}(\nabla_{\operatorname{HMT}}f)\|_{L^1(\sigma_*)} \approx \|f\|_{\dot M^{1,1}(\sigma_*)}.
	$$ 
\end{proposition}

\begin{proof}
	Let $f$ be a Lipschitz function in $M^{1,1}(\sigma_*)$. We first prove the inequality $\gtrsim$.
	Let $x \in \partial^* \Omega$ and $g \in \wt{\mathcal{T}}_1(x)$ ({as defined in Section \ref{section:HMT_def}}). Then, there exists a ball $B$ centered at $x$ of radius $r(B)$ such that $\operatorname{supp} g \subset B$ and $\Vert g\Vert_\infty \leq(r(B) \sigma_*(B))^{-1}$. Since $\int g \, d\sigma_*=0$, it holds that
	\[
	\begin{aligned}
		\int f g\, d \sigma_* &= \int\left(f-\fint_B f\,d\sigma_* \right)g\, d \sigma_*
	\end{aligned}
	\]
	which implies that
	\[
	\left|\int  f g \,d \sigma_*\right| \leq  \fint_B \frac{\left|f-\fint_B f\,d\sigma_*\right|}{r(B)} \,d \sigma_* \leq \Lambda_{\sharp} f(x).
	\]
	Taking the supremum over $g \in \wt{\mathcal T}_1(x)$ we obtain
	\[
	\mathcal M_{\text{gr}}(\nabla_{\operatorname{HMT}} f)(x) \leq  \Lambda_\sharp f(x).
	\]
	
	Let's prove the converse inequality $\Lambda_{\sharp}f(x) \lesssim\mathcal M_{\text{gr}}(\nabla_{\operatorname{HMT}} f)(x)$ for all $x\in\partial^*\Omega$.
	Let $B$ be a ball centered in $x\in\partial^*\Omega$, and $\Delta = B\cap\partial^*\Omega$. Then, we have
	\begin{align*}
		\int_{\Delta}& |f-f_\Delta| d\sigma_*  = 
		\sup_{\Vert g \Vert_{L^\infty(\sigma_*)} \leq 1} \left |\int_{\Delta} (f-f_\Delta) g \,d\sigma_* \right| 
		= \sup_{\Vert g \Vert_{L^\infty(\sigma_*)} \leq 1} \left |\int_{\Delta} f (g-g_\Delta) \,d\sigma_* \right| \\
		&\leq \sup_{\Vert \tilde g \Vert_{L^\infty(\sigma_*)} \leq 2,~ \int_B \tilde g \,d\sigma_* = 0} \left |\int_{\Delta} f \tilde g \,d\sigma_* \right| 
		\leq 2\sup_{\Vert \tilde g \Vert_{L^\infty(\sigma_*)} \leq 1,~ \int_B \tilde g \,d\sigma_* = 0} \left |\int_{\Delta} f \tilde g \,d\sigma_* \right|
	\end{align*}
	where $f_\Delta = \fint_\Delta f\, d\sigma_*$ and we have rewritten $g-g_\Delta =: \tilde g$.
	Dividing by $r(B) \sigma_*(B)$ on both sides, and taking into account that we can assume $\supp g \subset B$ above, we obtain
	\[
	\frac{1}{r(B) \sigma(\Delta)} \int_\Delta |f - f_\Delta|\, d\sigma_* \lesssim \mathcal M_{gr}(\nabla_{\operatorname{HMT}} f)(x),
	\]
	and taking the supremum over the balls centered at $x$, we deduce that $${\Lambda_\sharp f(x) \lesssim \mathcal M_{gr}(\nabla_{\operatorname{HMT}}f)(x)}.$$
	
	Finally, the comparability $	\Vert \mathcal M_{\text{gr}}(\nabla_{\operatorname{HMT}}f)\|_{L^1(\sigma_*)} \approx \|f\|_{\dot M^{1,1}(\sigma_*)}$ is a consequence of Theorem \ref{thm:calderon_maximal}. 
\end{proof}
Finally, we prove Proposition \ref{prop:reverseregularity} which states that, under the additional assumption that $\Omega$ satisfies the local John condition, we can reverse the $(R_1^{\mathcal L})$ solvability estimate. 
\begin{proof}[Proof of Proposition \ref{prop:reverseregularity}]

By arguments similar to those in \cite[pag 462]{KP} or \cite[Proposition 4.24]{HMT}, we will prove that 
$C \wt N(\nabla u)$ is a \haj\,\,gradient of $u|_{\pom}$ for some $C>0$ independent of $u$,
hence showing that $\Vert \wt N(\nabla u)\Vert_1 \gtrsim \Vert u \Vert_{\dot M^{1,1}}$.
Let $x,y\in\partial\Omega$, set $r = 2|x-y|$, $B=B(x,r)$, $x_B\in B\cap\Omega$ be a John center in $B\cap\Omega$, and $\tilde \gamma_z(s)$ to be the arclength parametrization of the non-tangential path connecting $z$ to $x_B$ for any $z \in \pom\cap B(x,r)$.
Recall that these paths $\tilde \gamma_z$ have length at most $Cr$ independently of the point $z$.

For $x,y\in \pom$, let $x_0:=x_B$, $x_k$ be a point in the path $\tilde\gamma_x$ satisfying $\dist(x_k, x) \approx r (1+c)^{-k}$ for some constant $c>0$ small enough, and that for each $k\geq0$, there is a ball $B_k$ centered at $x_k$ satisfying that $x_{k+1}\in B_k, ~2B_k \subset \Omega$ (define points $y_k$ analogously). Then, for $\sigma$-.a.e $x,y\in\pom$, we write
\begin{align*}
	|u(x) - u(y)| &\leq |u(x) - u(x_B)| + |u(y) - u(x_B)| \\
	&\leq \sum_{k=0}^\infty \left(|u(x_k) - u(x_{k+1})| + |u(y_k) - u(y_{k+1})|\right)
\end{align*}
where we have used that $\sigma$-a.e. $x\in\pom$, we have non-tangential convergence of $u$ to $u(x)$ ($u(y)$ respectively).

We can bound
\begin{align*}
	|u(x_k)-u(x_{k+1})| &\lesssim \Vert u - u(x_{k+1}) \Vert_{L^2(1.1 B_k)}\\&\lesssim r(B_k) \left(\fint_{1.1 B_k} |\nabla u|^2 \, dm\right)^{1/2} \\
	&\leq C r(1+c)^{-k} \wt N(\nabla u)(x)
\end{align*}
using a Poincar\'e inequality for solutions of divergence form elliptic PDEs in the interior such as \cite[Theorem 1]{Z}, and we consider the modified non-tangential maximal operator $\wt N$ associated with a cone of large enough aperture.
Summing up, we obtain
\[
	|u(x) - u(y)| \lesssim C r (\wt N(\nabla u)(x) + \wt N (\nabla u)(y)) \lesssim (\wt N(\nabla u)(x) + \wt N (\nabla u)(y)) |x-y|
\] 
for $\sigma\text{-a.e. } x,y\in\pom$.

\end{proof}
\color{black}

\color{black}
\section{Extrapolation of solvability of the regularity problem} \label{section:extrapolation}
In this section, we will prove the main results of this work. The theorems proved in Section \ref{section:hajlasz} regarding the atomic decompositions of \haj-Sobolev spaces will be used with $X = \partial\Omega$ (the $n$-Ahlfors regular boundary of a domain $\Omega$), $\sigma = \mathcal H^n|_{\partial\Omega}$, and Euclidean distance. The parameter $s$ in Section \ref{section:hajlasz} corresponding to the {Ahlfors} regularity will be now equal to $n$. {Note that the Ahlfors regularity of $\sigma$ implies that $\pom$ is uniformly perfect.}

First, we prove a localization theorem {inspired by the one in \cite[Theorem 5.19]{KP} for Lipschitz domains but adapted to much rougher domains without good connectivity properties}.

\begin{theorem}[Localization theorem for general elliptic operators]
	\label{thm:localization}
	Let $1<p<+\infty$,
	and $\Omega \subset \R^{n+1}$ be a corkscrew domain with {$n$-Ahlfors regular} boundary $\pom$ such that the problem $(D_{p'}^{\mathcal L^*})$ is solvable where $p'$ is the H\"older conjugate of $p$.  
	For $x_0 \in \pom$, let $B= B(x_0,R)$, and $f \in M^{1,p}(\partial\Omega)\cap \operatorname{Lip}(\pom)$ with $f$ {being constant} on $B(x_0,2R)$, and $\nabla_H f \in D(f) \cap L^p(\partial\Omega)$.
	Then,
	\[
	\fint_{B(x_0,R/2)\cap\partial\Omega} \wt N_{R/2}(\nabla u)^p d\sigma \lesssim 
	\left ( \fint_{A(x_0,R,2R)\cap\Omega} |\nabla u|\, dm  \right)^p
	\]
	{where  ${A(x_0, R, 2R)=B(x_0,2R)\backslash B(x_0,R)}$,  $\wt N_{R/2}$ is the modified non-tangential maximal operator truncated at height $R/2$, and $u$ is the solution of the continuous Dirichlet problem for the operator $\mathcal L$ with boundary data $f$.}
	
	Moreover, if the problem $(R_p^{\mathcal L})$ is solvable, we have
	\begin{multline}
	\fint_{B(x_0,R/2)\cap\partial\Omega} \wt N_{R/2}(\nabla u)^{p} d\sigma 
	\lesssim 
	\fint_{B(x_0,3R)\cap\partial\Omega} |\nabla_H f|^{p} d\sigma +
	\left ( \fint_{A(x_0,R,2R)\cap\Omega} |\nabla u|\, dm  \right)^{p}
	\end{multline}
	for all $f \in M^{1,p}(\pom)\cap \operatorname{Lip}(\pom)$, and $\nabla_H f \in D(f) \cap L^p(\pom)$. 
\end{theorem}
\begin{proof} 
	We start with the case when $f$ is {constant} on $B(x_0, 2R) \cap \pom$. {Without loss of generality, we may assume that $f$ vanishes on $B(x_0,2R) \cap \pom$ as we can subtract a constant to the solution without changing its gradient $\nabla u$.}
	\vspace{-2.5mm}
	\subsubsection*{Case $f\equiv0$ on $B(x_0,2R)\cap \pom$}
	Let $\phi$ be a smooth cutoff function for $B(x_0,3R/2)$ that is, $\phi \equiv 1$ in $B(x_0, 3R/2)$, $\supp \phi \subset B(x_0, 1.8R)$, and $|\nabla\phi|\lesssim R^{-1}$, $\xi \in B(x_0,R/2)\cap\partial\Omega$, and $x \in \gamma(\xi)\cap B(\xi,R/2)$.	
	
	We rewrite $u(x)$ as
	\begin{align}
		\label{eq:rewritting_u_localization}
		-u(x) = -(u\,\phi)(x) = \int_{\Omega} A^*(y)\nabla_y G(x,y) \cdot \nabla (u(y)\phi(y))\, dy.
	\end{align}
	where $G$ is the Green function for the operator $\mathcal L^*$ in $\Omega$ and $A^*$ is the associated matrix.
	
	Since { $\mathcal L(u\phi)=0$} in a neighborhood $B_r(x)$ of $x$ with $2r=\delta_\Omega(x):=\dist(x,\pom)$, we can use Caccioppoli's inequality to obtain
	\begin{align}
		\label{eq:bound_utimesphi_localization}
		\left( \fint_{B_{r/2}(x)} |\nabla u(z)|^2 \, dz \right)^{1/2} 
		&\lesssim \frac 1 r \sup_{z\in B_{2r/3}(x)} |(u\, \phi)(z)| \\
		&\leq \frac 1 r \sup_{z\in B_{2r/3}(x)} \left|\int_{A_B} (A^* \nabla_y G(z,y) \cdot \nabla \phi)u \, dy\right| \nonumber\\
		& \quad+ \frac 1 r \sup_{z\in B_{2r/3}(x)}\left|  \int_{B(x_0, 2R)} (A^* \nabla_y G(z,y) \cdot \nabla u)\phi\, dy \right| \nonumber\\
		& =: f_1(x) + f_2(x), \nonumber
	\end{align}
	where $A_B:=\overline{B(x_0,1.8R)} \backslash B(x_0, 3/2 R) \supset \supp \nabla\phi$. 
	
	First, we will bound $f_1(x)$. Note that, thanks to {Harnack inequality (note that $x$ is far from $A_B$),} $G(x,y) \approx G(t,y)$ for $y\in 1.05 A_B := \overline{B(x_0,1.05\cdot 1.8R)} \backslash B(x_0, 3/(2\cdot1.05) R)$ and $t \in B(x,\delta_\Omega(x)/2)$ . 
	Thus, we have
	\begin{align*}
		f_1(x) &= \frac 1 r \sup_{t \in B_{2r/3}(x)}\left|\int_{A_B} (\nabla_y G(t,y)\cdot A \nabla \phi) u\, dy \right| \\
		&\lesssim \frac 1 r \sup_{t \in B_{2r/3}(x)}\left(\int_{A_B} |\nabla_y G(t,y)|^2 \,dy\right)^{1/2} \left(\int_{A_B} |u \nabla \phi|^2 \,dy\right)^{1/2} \\
		&\lesssim \frac 1 {rR} \sup_{t \in B_{2r/3}(x)}\left(\int_{1.05 A_B} |G(t,y)|^2 \,dy\right)^{1/2} \left(\int_{A_B} |u \nabla \phi|^2 \,dy\right)^{1/2} \\
		&\lesssim \frac 1 {R^2r} \left(\int_{1.05 A_B} |G(x,y)|^2 \,dy\right)^{1/2} \left(\int_{A_B} |u|^2 \,dy\right)^{1/2} \\
		& \lesssim \frac 1 {R^2r} \sup_{t \in 1.05 A_B} G(x,t) \int_{1.05 A_B} |u| \,dy \lesssim \frac 1 {Rr} \sup_{t \in 1.05 A_B} G(x,t) \int_{1.05 A_B} |\nabla u| \,dy
	\end{align*}
	where we have used {Cauchy-Schwarz and Caccioppoli inequalities as well as Corollary \ref{coro:Poincare_functions_vanishing_pom}, which is a Poincar\'e inequality} for functions vanishing on a {$n$-Ahlfors regular set} (in this case $A_B \cap \pom$). 
	
	Now, we bound $f_2(x)$. Indeed, integrating by parts, using that $\mathcal Lu=0$, $G(x,y) \approx G(t,y)$ for $y\in A_B$ and $t \in B(x,\delta_\Omega(x)/2)$, and $\supp \nabla\phi \subset A_B$, we have
	\begin{align*}
		f_2(x) &= \frac 1 r \sup_{t\in B_{2r/3}(x)} \left|\int_{B(x_0, 2R)} (\nabla_y G(t,y) \cdot A \nabla u) \phi \, dy \right| \\
		&= \frac 1 r \sup_{t\in B_{2r/3}(x)} \left|\int_{A_B}  G(t,y) (A \nabla u \cdot \nabla \phi) \, dy  \right| 
		\lesssim \frac 1 r \int_{A_B}  \left|G(x,y) (A \nabla u \cdot  \nabla \phi)\right| \, dy \\
		&\lesssim \frac 1 {Rr} \sup_{t \in A_B} G(x,t) \int_{A_B} |\nabla u|\, dy.
	\end{align*}
	
	Since $G(x,\cdot)$ is {a positive} $\mathcal L^*$-subsolution in $1.05 A_B$, we have
	\[
	\sup_{t \in 1.05 A_B} G(x,t) \lesssim \fint_{1.1 A_B} G(x,y) \,dy \lesssim \fint_{1.1 A_B} \frac{\omega_{\mathcal L^*}^y(\Delta_\xi)}{\delta_\Omega(x)^{n-1}}\,dy 
	\]
	where we have also used $G(x,y) \lesssim \omega_{\mathcal L^*}^y(\Delta_\xi) \delta_\Omega(x)^{n-1}$ (Lemma \ref{lemma:bound_green_harmonic_measure}) where $\omega_{\mathcal L^*}^y$ is the elliptic measure for the operator $\mathcal L^*$ with pole at $y$, and $\Delta_\xi$ is the intersection of $B(\xi, \hat r)\cap \pom$ with $\delta_\Omega(x) \approx \hat r$. 
	Using this, the bounds for $f_1$ and $f_2$, and that $|A_B|\approx \sigma(B(x_0,R)) R$, we get
	\[
	\left( \fint_{B_{r/2}(x)} |\nabla u(z)|^2 \, dz \right)^{1/2} 
	\lesssim \sigma(B(x_0,R))\left(\fint_{1.1 A_B} \omega_{\mathcal L^*}^y(\Delta_\xi) \delta_\Omega(x)^{-n} dy\right) \left( \fint_{1.1A_B} |\nabla u| dy \right).
	\]
	On the other hand, we can further bound
	\[
	\fint_{1.1 A_B} \omega_{\mathcal L^*}^y(\Delta_\xi) \delta_\Omega(x)^{-n} dy \lesssim \fint_{1.1 A_B} M_{c,\sigma} (\omega_{\mathcal L^*}^y|_{B(x_0,CR)})(\xi) \,dy \lesssim \left ( \fint_{1.1 A_B} M_{c,\sigma} (\omega_{\mathcal L^*}^y|_{B(x_0,CR)})^p \,dy \right)^{1/p}
	\]	
	where $C>0$ depends on the constants of Lemma \ref{lemma:bound_green_harmonic_measure} and $M_{c,\sigma}(\omega_{\mathcal L^*}^y|_{B(x_0,CR)})(\xi) := \sup_{r>0} \omega_{\mathcal L^*}^y(B(\xi,r)\cap B(x_0, CR)) \sigma(B(\xi,r))^{-1}$.
	Thus, recalling that $x\in\gamma(\xi)\cap B_{R/2}(x_0)$, we get 
	\begin{align}
		\label{eq:localization_thm_important_bound}
		\fint_{B(x_0,R/2)\cap\partial\Omega} &|\wt N_{R/2}(\nabla u)(\xi)|^p d\sigma(\xi) \nonumber\\
		&\lesssim \fint_{B(x_0,R/2)\cap\partial\Omega} \fint_{1.1 A_B} M_{c,\sigma} \omega_{\mathcal L^*}^y|_{B(x_0,CR)}(\xi)^p dy d\sigma(\xi) \cdot 
		\left(\sigma(B(x_0,R))\fint_{1.1A_B} |\nabla u|\,dy \right)^p \\
		&\lesssim \fint_{1.1A_B} \fint_{B(x_0, CR)} \left ( \frac{d \omega_{\mathcal L^*}^y}{d\sigma} \right)^p d\sigma dy \left ( \sigma(B(x_0,R))\fint_{1.1A_B} |\nabla u|\,dy \right)^p \nonumber\\
		&\lesssim \left ( \fint_{1.1A_B} |\nabla u| \,dy \right)^p \nonumber
	\end{align}
	where we have used Fubini, the boundedness of the maximal operator $M_{c,\sigma}$ on $L^p(\sigma)$, the weak reverse H\"older property of the elliptic measure (Theorem \ref{thm:characterization_weak_A_infty} or \cite[Theorem 9.2 (b)]{MT} for the harmonic case), and that the elliptic measure $\omega_{\mathcal L^*}^y$ has total mass $1$ for all $y\in\Omega$.

	Summing up, we obtain
	\[
	\fint_{B(x_0,R/2)\cap\partial\Omega} |\wt N_{R/2}(\nabla u)|^p d\sigma \lesssim \left(\fint_{1.1A_B} |\nabla u| dy \right)^p.
	\]
	Thus, the proof is complete in the case where $\supp f \subset B(x_0, 2R)^c\cap\partial\Omega$ since $1.1 A_B \subset A(x_0, R, 2R)$.
	
	\vspace{-2.5mm}
	\subsubsection*{General case.}
	Let $B$ denote the ball $B(x_0, R)$ and $A_B = B(x_0, 2R)\backslash B(x_0,R)$.
	Let $\phi_B$ be a smooth bump function supported in $3B$ such that $\phi_B \equiv 1$ on $2B$ and $|\nabla \phi_B| \lesssim R^{-1}$. We can write $f$ as
	\[
	f = 
		(f - f_{3B}) \phi_B +  \left[f (1-\phi_B) + f_{3B}\phi_{B} \right] =: g+h,
	\]
	so that 
	\[
	u = u_g + u_h,
	\]
	where $f_{3B} = \fint_{3B\cap\pom} f \, d\sigma$, $u_g$ and $u_h$ are the solutions of the Dirichlet problem with boundary data $g$ and $h$ respectively.
	Then, if we let $\nabla_H g$ be the \haj~gradient given by the Leibniz' rule (Lemma \ref{lemma:leibniz}), we have
	\[
	\int_{3B\cap\partial\Omega} |\nabla_H g|^p \,d\sigma \lesssim \int_{3B\cap\partial\Omega} |\nabla_H f|^p \,d\sigma + \frac{1}{r(B)^p} \int_{3B\cap\partial\Omega} |f-f_{3B}|^p \,d\sigma \lesssim \int_{3B\cap\partial\Omega} |\nabla_H f|^p\, d\sigma
	\]
	where we have used the Poincar\'e inequality (Theorem \ref{thm:poincare_compact_supp}) for \haj~gradients with zero mean.
	
	Observe that
	\[
	\int_{B/2\cap\partial\Omega} \wt N_{R/2}(\nabla u)^p \,d\sigma \lesssim \int_{\partial\Omega} \wt N(\nabla u_g)^p \,d\sigma
	+ \int_{B/2\cap\partial\Omega} \wt N_{R/2}(\nabla u_h)^p \,d\sigma.
	\]
	For the function $u_g$, we have
	\[
	\int_{\partial\Omega} \wt N(\nabla u_g)^p \,d\sigma \lesssim \int_{\partial\Omega} |\nabla_H g|^p \,d\sigma \lesssim \int_{3B\cap\partial\Omega} |\nabla_H f|^p \,d\sigma
	\]
	as a consequence of the solvability of $(R^{\mathcal L}_p)$.
	Now, we may apply the localization theorem (the part we have proved already) to the function $u_h$ to obtain
	\[
	\fint_{B/2\cap\partial\Omega} \wt N_{R/2}(\nabla u_h)^p \,d\sigma \lesssim \left (\fint_{A_B} |\nabla u_h| \,dm  \right)^p
	\]
	{noting that $h$ is constant on $2B\cap\pom$.}
	We wish to have $u$ on the right hand side instead of $u_h$, therefore we bound
	\[
	\int_{A_B} |\nabla u_h| \,dy \leq \int_{A_B} |\nabla u| \,dy  + \int_{A_B} |\nabla u_g| \,dy,
	\]
	and since $\fint_{A_B} |\nabla u_g| \,dy \lesssim \fint_{3B\cap\partial\Omega} \wt N(\nabla u_g) \,d\sigma$ by 
	{standard} arguments, we end up obtaining
	\[
	\fint_{B/2\cap\partial\Omega} \wt N_{R/2}(\nabla u)^p \,d\sigma 
	\lesssim  \fint_{3B\cap\partial\Omega} |\nabla_H f|^p \,d\sigma
	+
	\left (\fint_{A_B} |\nabla u| \,dy\right)^p.
	\]
\end{proof}

\subsection{\texorpdfstring{$(R^{\mathcal L}_p)$ implies $(R^{\mathcal L}_q)$ for all $q\in(1-\epsilon,p]$, $p>1$}{Rp implies Rq for all q between 1-e and p}} \label{section:extrapolation_down}
In this section we will prove Theorem$\,\,$\ref{thm:extrapolation}. 
\begin{proof}[Proof of Theorem \ref{thm:extrapolation}]
	We will prove that there exists some $\epsilon>0$ such that for all $r\in(1-\epsilon,1]$, solvability of $(R^{\mathcal L}_p)$ 
	implies solvability {to} $(R^{\mathcal L}_r)$. To do so, we will prove that for any {Lipschitz} $(\infty,\infty,r)$ atom $f$, the solution $u$ to the Dirichlet problem  with boundary data $f$ satisfies 
\begin{equation}
	\label{eq:bound_nontangential_atoms}
	\int_{\partial\Omega} |\wt N(\nabla u_f)|^r \,d\sigma \lesssim 1.
\end{equation}
	Since Theorem \ref{thm:atomic_decomposition} ensures that we can write any Lipschitz function in $\dot M^{1,r}(\pom)$ for $r\leq1$ as a sum of Lipschitz $(\infty,\infty,r)$ atoms, we obtain solvability of $(R^{\mathcal L}_r)$. Finally, Theorem \ref{thm:interpolation} allows us to interpolate the result for all the exponents in the range $(1-\epsilon,p)$ and this ends the proof.
\begin{remark}
	If $\pom$ is compact and $\Omega$ is unbounded, we have to prove \eqref{eq:bound_nontangential_atoms} for Lipschitz $(\diam\pom, \infty, r)$ atoms as the regularity $(R_r^{\mathcal L})$ estimate involves the $M^{1,r}(\pom)$ norm. Clearly, we have that every {Lipschitz} $(\diam\pom, \infty, r)$ atom is also an $(\infty,\infty, r)$ atom, and every $(\infty, \infty, r)$ atom with support strictly contained in $\pom$ is a $(\diam\pom, \infty, r)$ atom thanks to the Poincar\'e inequality in Theorem \ref{thm:poincare_compact_supp}. 
	Hence, the only extra work we have to do is establish \eqref{eq:bound_nontangential_atoms} for any $(\diam\pom, \infty, r)$ atom $f$ with $\supp f = \pom$.
	But, this bound is direct thanks to H\"older's inequality and the solvability of $(R_p^{\mathcal L})$ as
	\begin{align*}
	\Vert \wt N(\nabla u_f) \Vert_{L^r(\sigma)} \leq \sigma(\pom)^{1/r-1/p} \Vert \wt N(\nabla u_f) \Vert_{L^p(\sigma)} \lesssim \sigma(\pom)^{1/r-1/p} \Vert f \Vert_{M^{1,p}(\sigma)} \lesssim 1
	\end{align*}
	using that $f$ is a $(\diam\pom, \infty, r)$ atom.
\end{remark}

Let $f$ be a $(\infty,\infty,r)$ atom, that is, let $f$ be a Lipschitz function such that $\supp f \subset B=B(x_0, R)$, $x_0 \in \pom$, and $\left(\fint_{B\cap\partial\Omega} |\nabla_H f|^p \,d\sigma\right)^{1/p} \leq \sigma(B)^{-1/r}$ for some $\nabla_H f \in D(f)\cap L^p$ ({if $\pom$ is compact and $\Omega$ is unbounded, we assume that $B\cap \pom \subsetneq \pom$}). Let $u$ be the solution to the continuous Dirichlet problem in $\Omega$ for $\mathcal L$ with boundary data $f$.
Then, by solvability of $(R_p^{\mathcal L})$, we have 
\[
\Vert \wt N(\nabla u) \Vert_p \lesssim \Vert \nabla_H f \Vert_p \leq \sigma(B)^{1/p-1/r}.
\]

\begin{claim*}
	For $x \in (2^{k+1}B\backslash 2^k B)\cap\Omega=:A_k$, $k\geq1$, we have $|u(x)|\lesssim r(B) \sigma(B)^{-1/r}$.
\end{claim*}
\begin{proof}[Proof of the claim.]
By the maximum principle, it is enough to show this for $x\in (3B\backslash 2B)\cap\Omega=A(x_0,2R,3R)\cap\Omega$ {since $u$ vanishes on $\pom \backslash B$ and also at $\infty$ if $\Omega$ is unbounded}. Let $\phi$ be a smooth cutoff function for $A(x_0,1.5R,3.5R)$ with support in $B(x_0, 4R)$. 
	As in the proof of \eqref{eq:bound_utimesphi_localization} and the following paragraphs in the {localization theorem} (Theorem \ref{thm:localization}), we can bound
	\begin{align*}
		(u\, \phi)(x) &\lesssim \frac 1 {R} \sup_{z \in B(x_0,4R)\backslash A(x_0, 1.4R, 3.6R)} G(x,z) \int_{A(x_0, 1.4R, 3.6R)\cap\Omega} |\nabla u| \,dy \\
		&\leq R \fint_{A(x_0, 1.4R, 3.6R)\cap\Omega} |\nabla u{(y)}|\,dy 
		\leq R \left( \fint_{4B\cap\Omega}|\nabla u|^p dy\right)^{1/p} \\
		&\lesssim R \left(\fint_{8B\cap\Omega} |\wt N(\nabla u)|^p\, d\sigma\right)^{1/p}  
		\lesssim R \sigma(B)^{-1/p} \left(\int_{\pom} |\nabla_H f|^p d\sigma\right)^{1/p}\\
		&\lesssim R \sigma(B)^{-1/r}
	\end{align*}
	where we have used that $G(x,z) \lesssim |x-z|^{1-n}$ and the solvability of $(R_p^{\mathcal L})$.
\end{proof}
On one hand, near $B\cap\pom$,  {thanks to the solvability of $(R_p^{\mathcal L})$, we have the bound
\begin{align*}
	\left (\fint_{8B\cap\partial\Omega} |\wt N(\nabla u)|^r d\sigma \right)^{1/r} &\leq 
	\left ( \fint_{8B\cap\partial\Omega} |\wt N(\nabla u)|^p d\sigma\right)^{1/p} \\
	&\lesssim_{(R_p^{\mathcal L})} \left ( \fint_{B\cap\partial\Omega} |\nabla_H f|^p d\sigma \right )^{1/p}  \leq \sigma(B)^{-1/r}.
\end{align*}
}

On the other hand, far from $B$, we have that
\begin{equation}
	\label{eq:bound_u_outside_2B}
	|u(x)| \lesssim \frac{R^{n-1+\alpha}}{|x_0-x|^{n-1+\alpha}} R\sigma(B)^{-1/r} = \frac{R^{\alpha+n-n/r}}{|x_0-x|^{n-1+\alpha}}, \quad \forall x \in \Omega \cap (2B)^c
\end{equation}
by the previous claim and Lemma \ref{lemma:AGMT}.
Note that the constant $\alpha$ depends only on the dimension $n$, $n$-Ahlfors regularity constants of $\partial\Omega$, and the ellipticity constant of $\mathcal L$.
Using Caccioppoli and \eqref{eq:bound_u_outside_2B}, for $x$ such that $|x_0-x| \approx \delta_\Omega(x)$ and $\delta_\Omega(x) \geq 3R$, we get 
\begin{equation}
	\label{eq:bound_nablau_far_bdry}
	\left(\fint_{B(x,\delta_\Omega(x)/4)}|\nabla u|^2\,dm\right)^{1/2} \lesssim \frac{R^{\alpha+n-n/r}}{|x_0-x|^{n +\alpha}}.
\end{equation}
Then,  {in $A_k \cap\pom$ for $k\geq 3$}, we obtain 
\[
\int_{A_k\cap\partial\Omega} |\wt N(\nabla u)|^r d\sigma \lesssim \int_{A_{k}\cap\partial\Omega} \left(|\wt N_{2^{k-2} R} (\nabla u)|^r + \left(c2^{-k(n+\alpha)}R^{-n/r}\right)^r \right)d\sigma
\]
using the bound for $\nabla u$ far from the boundary in \eqref{eq:bound_nablau_far_bdry}.
Using the {localization theorem} (Theorem \ref{thm:localization}), that $r\leq1$, {H\"older and boundary Caccioppoli inequalities (since  $u$ vanishes {on $A_k\cap\pom$}}), and the bound for $|u|$ for $x\notin 2B$ from \eqref{eq:bound_u_outside_2B}, we obtain
\begin{align*}
	\left ( \fint_{A_{k}\cap\partial\Omega} |\wt N_{2^{k-2} R}(\nabla u)|^r d\sigma \right)^{1/r} &\lesssim \fint_{A_k'\cap\partial\Omega} |\nabla u|\, dx
	\lesssim \left (\fint_{A_k'\cap\partial\Omega} |\nabla u|^2\, dx\right)^{1/2}\\
	&\lesssim 
	(2^k R)^{-1}\left (\fint_{A_{k}''\cap\Omega} |u|^2 dx \right)^{1/2}
	\lesssim 2^{-k(n+\alpha)}\sigma(B)^{-1/r}{,}
\end{align*}
where $A_k'$ and $A_k''$ are small enlargements of $A_k$ and $A_k'$ respectively. 

Summing over $k$, we obtain
\begin{align*}
	\int_{\partial\Omega} |\wt N(\nabla u)|^r d\sigma &\leq \int_{8B\cap\partial\Omega} |\wt N(\nabla u)|^r d\sigma+ \sum_{k=3}^\infty \int_{A_k \cap \partial\Omega} |\wt N(\nabla u)|^r\, d\sigma\\
	&\lesssim 1 + \sum_{k=3}^\infty \int_{A_{k}\cap\partial\Omega} \left(|\wt N_{2^{k-2}R} (\nabla u)|^r + \left(c2^{-k(n+\alpha)} \sigma(B)^{-1/r}\right)^r \right)d\sigma \\
	&\lesssim 1 + \sum_{k=3}^\infty \left(c2^{-k(n+\alpha)} \sigma(B)^{-1/r}\right)^r \sigma(A_k\cap\partial\Omega) \\
	&\lesssim 1 + \sum_{k=3}^\infty \left(2^{-k(n+\alpha)}\sigma(B)^{-1/r}\right)^r 2^{kn}R^n \\
	&= 1 + \sum_{k=3}^\infty 2^{-k(n+\alpha)r} 2^{kn} \lesssim 1.
\end{align*}
as long as $-(n+\alpha)r+n<0$.
Thus, choosing $ \epsilon = \frac{n}{n+\alpha} $ where $\alpha$ comes from Lemma \ref{lemma:AGMT} and depends only on the dimension $n$, the Ahlfors regularity constants of $\partial\Omega$, and the ellipticity constant of $\mathcal L$, the theorem follows.
\end{proof}
\color{black}

\subsection{\texorpdfstring{$(R^{\mathcal L}_1)$ implies $(R^{\mathcal L}_{1+\tilde\epsilon})$}{R1 implies R(1+e)}}
\label{section:extrapolation_up} 
In this section we aim to prove a converse statement to Theorem \ref{thm:extrapolation} in the endpoint case $p=1$. In particular, we will prove that solvability {$(R_1^{\mathcal L})$} implies that {elliptic} measure $\omega_{\mathcal L^*}$ is in weak-$\mathcal A_\infty(\sigma)$. {Using this fact, we will also show that {$(R_{1+\tilde\epsilon}^{\mathcal L})$} is solvable for some small $\tilde\epsilon>0$}.

The following lemma gives an easy to check characterization of the weak-$\mathcal A_\infty$ property. It appeared in this form in \cite[Lemma 3.2]{HL} but traces its roots {to} \cite{BL}.
Given a point $x \in \Omega$, let $\hat x$ be a ``touching point" of $x$, that is $\hat x\in\partial\Omega$ such that $|x-\hat x| = \delta_\Omega(x)$. Then, we define the boundary ball
\begin{equation}
	\label{eq:def_boundary_ball}
\Delta_x := B(\hat x, 10\delta_\Omega(x))\cap\partial\Omega.
\end{equation}
\begin{lemma}
\label{lemma:weak_A_infty} 
Let $\partial \Omega$ be $n$-Ahlfors regular, and suppose that there are constants $c_0, \eta \in(0,1)$, such that for each $x \in \Omega$, with $\delta_\Omega(x)<\operatorname{diam}(\partial \Omega)$, and for every Borel set $F \subset \Delta_x$,
\[
\sigma(F) \geq(1-\eta) \sigma(\Delta_x) \implies \omega_{\mathcal L^*}^x(F) \geq c_0 .
\]
Then $\omega^y_{\mathcal L^*}\in$ weak-$\mathcal A_{\infty}(\Delta)$, where $\Delta=B \cap \partial \Omega$, for every ball $B=B(\xi, r)$, with $\xi \in \partial \Omega$ and $0<r<\operatorname{diam}(\partial \Omega)$, and for all $y \in \Omega \backslash 4 B$. Moreover, the parameters in the weak-$\mathcal A_{\infty}$ condition depend only on $n$, the {Ahlfors} regularity constants, $\eta, c_0$, and the ellipticity parameter $\lambda$ of the divergence form operator $\mathcal L^*$.
\end{lemma}

The next lemma is a H\"older inequality for the Orlicz space $L \log L$ involving the centered Hardy-Littlewood maximal operator.
\begin{lemma}[$L \log L$-H\"older inequality]
\label{lemma:holder_inequality_maximal}
Let $\partial\Omega \subset \mathbb R^{n+1}$ be $n$-Ahlfors regular. Let $B$ be a ball centered on $\partial\Omega$, $f\in L^1(\partial\Omega)$, and $E\subset B\cap\partial\Omega$ with $\sigma(E)>0$. Then, we have
\[
\int_E f \,d\sigma \lesssim \log\left(1+\frac{\sigma(B\cap\pom)}{\sigma(E)}\right)^{-1} \int_{B\cap\partial\Omega} M_c(f\chi_B)\, d\sigma
\]
with constants independent of $E, f$, and the ball $B$.
\end{lemma}
A proof for $\R^n$ with Lebesgue measure can be found in \cite[Section 2.3]{L}, but the same ideas extend to {Ahlfors} regular metric spaces.

The following lemma gives a reverse Holder inequality in terms of the modified non-tangential maximal operator, which will be essential for our purposes.
\begin{lemma}
	\label{lemma:reverse_holder_nontangential}
	Let $\Omega \subset \mathbb{R}^{n+1}$ be a corkscrew domain with $n$-Ahlfors regular boundary. Then, for any $p \in\left[1,1+\frac{1}{n}\right)$, any function $v: \Omega \rightarrow \mathbb{R}$, and any ball $B$ centered in $\partial \Omega$, we have
	\[
	\fint_{B \cap \Omega} |v(y)|^p \,d y \lesssim
	\bigg( \fint_{2 B \cap \partial \Omega} \wt N\color{black} v\, d \sigma \bigg)^p {,}
	\]
	where $\wt N$ is the modified \color{black} non-tangential maximal operator with aperture large enough depending on the Ahlfors regularity of $\partial \Omega$.
\end{lemma}
\begin{proof}
Consider a Whitney cube decomposition $\mathcal{W}$ of $\Omega$ {into dyadic cubes}. For any $Q \in \mathcal{W}$, denote by $B_Q$ a boundary ball with radius $r(B_Q) \approx l(Q)$ such that $d(Q, B_Q) \approx l(Q)$, and $Q \subset \gamma(\xi)$ for every $\xi \in B_Q$ (for this we need the aperture large enough). We obtain
	\[
	\begin{aligned}
		\int_{B \cap \Omega}|v|^p d y&=  \sum_{\substack{Q \in \mathcal{W} \\
				Q \cap B \neq \varnothing}} \int_{Q \cap B}|v|^p d y \leq \sum_{\substack{Q \in \mathcal{W} \\
				Q \cap B \neq \varnothing}} m_Q(|v|^p) l(Q)^{n+1} \\
		&\lesssim \sum_{\substack{Q \in \mathcal{W} \\
				Q \cap B \neq \varnothing}} l(Q)^{n+1} \inf _{\xi \in B_Q} \wt N v(\xi)^p \color{black}\leq \sum_{\substack{Q \in \mathcal{W} \\
				Q \cap B \neq \varnothing}} \frac{l(Q)^{n+1}}{\sigma(B_Q)^p}\left(\int_{B_Q} \wt N v(\xi) d \sigma(\xi)\right)^p \\
		&\approx \sum_{\substack{Q \in \mathcal{W} \\
				Q \cap B \neq \varnothing}} l(Q)^{n+1-p n}\left(\int_{B_Q} \wt N v(\xi) d \sigma(\xi)\right)^p
	\end{aligned}
	\]
	where we have used H\"older's inequality and that $\sigma(B_Q) \approx l(Q)^n$ thanks to the Ahlfors regularity of $\partial \Omega$. Now, we use that $\eta:=n+1-p n>0$, and we rewrite the previous sum as a sum over generations of Whitney cubes from some starting generation {$k_0 \approx-\log _2(r(B))$. Then, we have that}
	\[
	\begin{aligned}
		\sum_{\substack{Q \in \mathcal{W} \\
				Q \cap B \neq \varnothing}} l(Q)^{n+1-p n}\left(\int_{B_Q} \wt N v(\xi) d \sigma(\xi)\right)^p & {\leq}\sum_{k \geq k_0} \sum_{\substack{Q \in \mathcal{W} \\
				Q \cap B \neq \varnothing \\
				l(Q)=2^{-k}}} 2^{-k \eta}\left(\int_{B_Q} \wt N v(\xi) d \sigma(\xi)\right)^p \\
		& \lesssim \sum_{k \geq k_0} 2^{-k \eta}\left(\int_{2 B \cap \partial \Omega} \wt N v(\xi) d \sigma(\xi)\right)^p \\
		& \approx\left(\int_{2 B \cap \partial \Omega} \wt N v(\xi) d \sigma(\xi)\right)^p 2^{-k_0 \eta}
	\end{aligned}
	\]
	where we have used that the boundary balls $B_Q$ are all contained in $2 B$, they have finite superposition in every generation $k, p>1$, and the sum is geometric. Taking into account that $2^{-k_0 \eta} \approx \frac{|B \cap \Omega|}{\sigma(2 B \cap \partial \Omega)^p}$ thanks to the corkscrew property and the $n$-Ahlfors regularity of $\partial \Omega$ finishes the proof.
\end{proof}
The following lemma is key to prove {that the elliptic measure $\omega_{\mathcal L^*}$ satisfies the assumptions of Lemma \ref{lemma:weak_A_infty}}.

\begin{lemma}
	\label{lemma:R1_implies_bound_for_maximal}
	Let $\Omega \subset \mathbb R^{n+1}$ be a corkscrew domain with $n$-Ahlfors regular boundary such that $(R_1^{\mathcal L})$ is solvable. Then, there exist constants $C>0$ depending on $n$, the corkscrew constant, the {Ahlfors} regularity constant, the ellipticity constant of $\mathcal L$, and the $(R_1^\mathcal L)$ constant,  and $C'>0$ depending only on $n$ and the {Ahlfors}  regularity constant, such that for all $x\in\Omega$, we have 
	\[
	\int_{\Delta_x} M_{c, \sigma, \delta_{\Omega}(x) / C'}(\omega_{\mathcal L^*}^x)(\xi) \,d \sigma(\xi) \leq C,
	\]
	where $M_{c, \sigma, \delta_{\Omega}(x) / C'}$ {is the truncated Hardy-Littlewood maximal operator \eqref{eq:def_maximal_truncat}} and $\Delta_x$ is the boundary ball corresponding to $x$ (see display \eqref{eq:def_boundary_ball}).
\end{lemma}
\begin{proof} 
	We set $r_x:= \delta_\Omega(x)/4K$ for some $K$ depending on the Ahlfors regularity constants, which will be chosen momentarily. We consider a maximal collection of points $\{\xi_i\}_{i=1}^{m_0} \subset \pom \cap \Delta_x$ which are $r_x$-separated. Then, for every $\xi \in \Delta_x$ and $r \in (0,r_x]$, we can find $i \in \{1,\hdots,m_0\}$ for which
	\[
	B(\xi,r) \subset B(\xi_i, 2r_x).
	\]
	We fix $\xi \in \Delta_x$, $r\in (0,r_x]$, and $i\in\{1,\hdots,m_0\}$. We take $K>128$ big enough (depending only on the $n$-Ahlfors regularity constant) so that there is at least one point $\xi_0 \in (B(\xi_i, (K-1)r_x)\backslash B(\xi_i, 10 r_x))\cap \pom \neq \varnothing$ (see Remark \ref{rmk:uniformly_perfect}). Then we consider a Lipschitz function $\varphi_i : \mathbb R^{n+1} \to \mathbb R$ satisfying:
	\begin{itemize}
		\item $\varphi_i=1$ in $A(\xi_i, 9r_x, K r_x)$,
		\item $\varphi_i = 0$ in $\mathbb R^{n+1} \backslash A(\xi_i, 8r_x, (K+1)r_x)$,
		\item $0\leq \varphi_i \leq 1$ and $\|\varphi_i\|_{\operatorname{Lip}} \leq r_x^{-1}$.
	\end{itemize}
	
	From these properties and the fact that $g_i := r_x^{-1} \chi_{B(\xi_i, (K+1) r_x)}$ is a \haj~upper gradient for $\varphi_i$, it follows that 
	\[
	\Vert \varphi_i \Vert_{\dot M^{1,1}(\sigma)} \lesssim_K r_x^{n-1}.
	\]
	Let $u_i$ be the solution of the continuous Dirichlet problem for $\mathcal L$ in $\Omega$ with data $\varphi_i$ given by 
	\[
	u_i(y) = \int_{\pom} \varphi_i \, d\omega_\mathcal{L}^y.
	\]
	
	We define $B_0 := B(\xi_0, r_x/4)$. If we set $v_i := 1 - u_i$, it is direct to check that $\mathcal L v_i=0$ in $4B_0$, $v_i=0$ on $4B_0 \cap \pom$, and $0\leq v_i \leq 1$. Therefore, by boundary H\"older continuity of $v_i$ (see \cite[Lemma 2.10]{AGMT} for example), it holds that
	\[
	1-u_i(y)=v_i(y) \leq C\left(\delta_{\Omega}(y) /  r_x\right)^\alpha \sup _{2 B_0 \cap \Omega} v_i \leq 1 / 2
	\]
	for all $y \in B_0$ such that $\delta_{\Omega}(y) \leq(2 C)^{-\alpha}  r_x$. Therefore, we have that $u_i(y) \geq \frac{1}{2}$ for $y \in B_0$ such that $\delta_{\Omega}(y) \leq(2 C)^{-\alpha} r_x=: r_0$. We define $r_i:=\left|\xi_i-\xi_0\right| \approx r_x$ and set $B_i:=B\left(\xi_i, r_i\right)$. Then, we can cover $\partial B_i \cap \partial \Omega$ by a uniformly bounded number of balls $\widetilde{B}_k$ centered at $\partial B_i \cap \partial \Omega$ with radius $r_0$ and by the same argument as before we can show that $u_i(y) \geq \frac{1}{2}$ for any $y \in \widetilde{B}_k$ such that $\delta_{\Omega}(y) \leq r_0$. This implies that
	\[
	u_i(y) \geq 1 / 2, \quad \text { for any } y \in\left\{y \in \partial B_i \cap \Omega: \delta_{\Omega}(y) \leq r_0\right\} .
	\]
	Since every $y \in\left\{y \in \partial B_i \cap \Omega: \delta_{\Omega}(y)>r_0\right\}$ can be connected by a Harnack chain of balls centered at $\partial B_i \cap \Omega$ with radii $\approx r_0$ to a point $z \in\left\{y \in \partial B_i \cap \Omega\right.$ : $\left.\delta_{\Omega}(y) \leq r_0\right\}$, then, by Harnack's inequality, we obtain that there exists a uniform constant $c \in (0,1/2)$ such that $u_i(y) \geq c$ for every $y \in \partial B_i \cap \Omega$. Hence, since the Green function for $\mathcal L^*$ satisfies $G(x,y) \lesssim r_x^{1-n}$ for every $y \in \partial B_i \cap \Omega$, by the maximum principle,
	\begin{equation}
		\label{eq:upper_bound_Green_function_in_terms_annulus}
		G(x,y) \lesssim \delta_\Omega(x)^{1-n} u_i(y) \quad\mbox{for every } y \in B_i. 
	\end{equation}
	Thus, by the Poincar\'e inequality for functions vanishing on $n$-Ahlfors regular sets in Corollary \ref{coro:Poincare_functions_vanishing_pom} as $u_i$ vanishes on $B(\xi,4r) \cap \pom$, Lemma \ref{lemma:reverse_holder_nontangential} and H\"older's inequality, we have that
	\begin{align}
		\label{eq:bound_u_annulus_by_N}
		\fint_{B(\xi, 4r)} u_i(y)\, dy &\lesssim r \fint_{B(\xi,4r)} |\nabla u_i|\, dy \nonumber\\
		&\lesssim r \left(\fint_{B(\xi, 8r)} \wt N(|\nabla u_i|^{1/p}) \,d\sigma\right)^{p}
		\leq r M_{c, \sigma, 8r}\left( \wt N(|\nabla u_i|^{1/p})\right)(\xi)^p\\
		&\leq r M_{c, \sigma, 8r}\left( \wt N(|\nabla u_i|)^{1/p}\right)(\xi)^p \nonumber
	\end{align}
	for $p = 1+\frac 1 {2n}\in(1,2)$.
	
	Therefore, by \cite[Lemma 2.6]{AGMT}, Cauchy-Schwarz, Caccioppoli's
	inequality, Moser's estimate at the boundary, \eqref{eq:upper_bound_Green_function_in_terms_annulus}, and \eqref{eq:bound_u_annulus_by_N}, we obtain
	\begin{align*}
		\frac{\omega^x_{\mathcal L^*}(B(\xi, r))}{\sigma(B(\xi, r))} &\lesssim \fint_{B(\xi, 2 r) \cap \Omega}|\nabla_y G(x,y)| \,d y \lesssim r^{-1} \fint_{B(\xi, 4 r) \cap \Omega} G(x,y)\, dy \\
		&\lesssim \delta_{\Omega}(x)^{1-n} r^{-1} \fint_{B(\xi, 4 r) \cap \Omega} u_i\, dy \lesssim \delta_{\Omega}(x)^{1-n} M_{c,\sigma,8r}\left(\wt{{N}}(|\nabla u_i|)^{1/p}\right)(\xi)^{p}.
	\end{align*}
	Consequently, for any fixed $\xi \in \Delta_x$, if we take supremum over all $r \leq r_x$, we get
	\[
	M_{c, \sigma, r_x} \omega_{\mathcal L^*}^x(\xi) \lesssim \delta_{\Omega}(x)^{1-n} M_{c,\sigma, 8r_x}\left(\wt{N}\left(\nabla u_i\right)^{1/p}\right)(\xi)^{ p},
	\]
	and, by the strong $L^p$ boundedness of the Hardy-Littlewood maximal operator and the solvability of $(R_1^{\mathcal L})$ in $\Omega$, we have
	\begin{align*}
		\int_{\Delta_x} M_{c, \sigma, r_x} \omega_{\mathcal L^*}^x(\xi)\, d\sigma(\xi) &\lesssim 
		\delta_{\Omega}(x)^{1-n} \sum_{j=1}^{m_0} \int_{\Delta_x} {M}_{c,\sigma, 8r_x}\left(\widetilde{{N}}\left(|\nabla u_j|\right)^{1/p}\right)^{p}\, d\sigma \\
		&\lesssim 
		\delta_{\Omega}(x)^{1-n} \sum_{j=1}^{m_0} \int_{\Delta_x} \widetilde{{N}}\left(|\nabla u_j|\right) \, d\sigma
		\lesssim \delta_{\Omega}(x)^{1-n} \sum_{j=1}^{m_0}\left\|\varphi_j\right\|_{\dot M^{1, 1}(\sigma)} \\
		&\lesssim m_0 \delta_{\Omega}(x)^{1-n} r_x^{n-1} \lesssim 1
	\end{align*}		
	which finishes the proof. 
	\color{black}
\end{proof}

Finally, we are ready to prove Theorem \ref{thm:equivalence_R1_weakAinf}.
\begin{proof}[Proof of Theorem \ref{thm:equivalence_R1_weakAinf}]
We will show that {under the assumption that $(R_1^{\mathcal L})$ is solvable,} the {elliptic measure $\omega_{\mathcal L^*}$} satisfies the properties of Lemma \ref{lemma:weak_A_infty}.

Let $\eta>0$ be small enough to be chosen soon, and take $F\subset \Delta_x$ with $\sigma(F) \geq (1-\eta)\sigma(\Delta_x)$ and $\tilde F = \Delta_x \backslash F$.
By the so-called {Bourgain's lemma} (Lemma \ref{lemma:Bourgain}), there exists $C>0$ such that 
\begin{equation}
    \label{eq:Bourgain}
    \omega_{\mathcal L^*}^x(\Delta_x) \geq C.
\end{equation}
By Lemma \ref{lemma:R1_implies_bound_for_maximal}, we also have
\begin{align}
	\label{eq:proof_thm_1.5}
\int_{\Delta_x} M_{c}\left(\frac{d\omega_{\mathcal L^*}^x}{d\sigma}\chi_{\Delta_x}\right) d\sigma &\leq 
\int_{\Delta_x} \left(M_{c, \delta_\Omega(x)/C'} \left(\frac{d\omega_{\mathcal L^*}^x}{d\sigma}\chi_{\Delta_x}\right)  
+ \sup_{r > \delta_\Omega(x)/C'} \frac{\omega_{\mathcal L^*}^x(B(\xi,r)\cap\Delta_x)}{\sigma(B(\xi,r))} \right) \,d\sigma \nonumber \\
&\lesssim \int_{\Delta_x} M_{c, \delta_\Omega(x)/C'} \left(\frac{d\omega_{\mathcal L^*}^x}{d\sigma}\chi_{\Delta_x}\right) d\sigma + \int_{\Delta_x}\frac{\omega_{\mathcal L^*}^x(\Delta_x)}{\delta_\Omega(x)^n}\, d\sigma(\xi) \lesssim  1, 
\end{align}
where $M_{c, \delta_\Omega(x)/C'}$ is the  centered and truncated Hardy-Littlewood maximal operator.
Then, using \eqref{eq:Bourgain}, Lemma \ref{lemma:holder_inequality_maximal}, and \eqref{eq:proof_thm_1.5}, we obtain
\begin{align*}
    \omega_{\mathcal L^*}^x(F) &= \omega_{\mathcal L^*}^x(\Delta_x) - \omega_{\mathcal L}^x(\tilde F)
    \geq C - \int_{\tilde F} \frac{d\omega_{\mathcal L^*}^x}{d\sigma} d\sigma \\
    &\geq C -  C_1 \log(1+\eta^{-1})^{-1} \int_{\Delta_x} M_{c}\left(\frac{d\omega_{\mathcal L^*}^x}{d\sigma}\chi_{\Delta_x}\right) d\sigma \\
	&\geq C -  C_2 \log(1+\eta^{-1})^{-1} \geq \frac C 2
\end{align*}
if $\eta$ is small enough.
\end{proof}

\begin{proof}[Proof of Theorem \ref{thm:extrapolation_1}]
Under the assumption that $(R_1^{\mathcal L^*})$ is solvable, we have
	\begin{equation}
	\label{eq:localization_p=1}
	\fint_{B\left(x_0, R / 2\right) \cap \pom} \wt{N}_{R / 2}(\nabla u) \,d \sigma \lesssim \fint_{B\left(x_0, 3 R\right) \cap \pom}\left|\nabla_H f\right| \,d \sigma+\fint_{A\left(x_0, R, 2 R\right) \cap \Omega}|\nabla u| \,d m
	\end{equation}
for all $f \in M^{1, 1}(\partial \Omega) \cap C(\partial \Omega)$, $\nabla_H f \in D(f) \cap L^1(\partial \Omega)$, and $u$ solution of the continuous Dirichlet problem with boundary data $f$.
The proof of \eqref{eq:localization_p=1} is analogous to the one of the Localization Theorem \ref{thm:localization} but using Lemma \ref{lemma:R1_implies_bound_for_maximal} instead of the weak reverse $p$-H\"older inequality for $ \frac{d\omega_{\mathcal L^*}}{d\sigma}$ on \eqref{eq:localization_thm_important_bound}.

The rest of the approach is that of \cite[Theorem 5.3]{KP} (see also \cite{DK1}). Let $f$ be a continuous function in $M^{1,1}(\pom)$, $\nabla_H f\in D(f) \cap L^1(\pom)$, $u$ be the solution to the continuous Dirichlet problem for $\mathcal L$ with boundary data $f$, and $\wt N_\alpha$ be the modified non-tangential maximal operator with aperture $\alpha$ large enough. Let $\lambda>0$ and $B(\xi_0, r)$ be a ball centered on $\pom$ with $r<\diam(\pom)/4$ such that $M_{\sigma}(\nabla_H f)(\xi_2) \leq \lambda$ for some $\xi_2 \in \pom \cap B(\xi_0, r)$ and $\wt N_\alpha (\nabla u)(\xi_3) \leq \lambda$ for some $\xi_3 \in  \pom \cap B(\xi_0, 2r) \backslash B(\xi_0,r)$.
Then, the estimate
\begin{equation}
	\label{eq:good_lambda}
\fint_{B(\xi_0,r)\cap \pom} \wt N(\nabla u) \, d\sigma \leq C \alpha^{-\eta} \fint_{B(\xi_0,3r)\cap \pom} \wt N(\nabla u) \, d\sigma + C \lambda
\end{equation}
with constants $C$ and $\eta>0$ independent of $\lambda$ and $f$ implies solvability of $(R_{1 + \tilde \epsilon}^{\mathcal L})$ for some $\tilde \epsilon>0$ (see \cite[Theorem 2.13]{DK1}). Inequality \eqref{eq:good_lambda} follows from \eqref{eq:localization_p=1} and the argument in \cite{KP}.

\end{proof}

\subsection{Extrapolation of solvability of the modified Poisson regularity problem}\label{section:extrapolation_Poisson}
In this section, we will prove Theorem \ref{thm:extrapolation_Poisson}.
First, we state a localization theorem for this setting. 
\begin{theorem}[Localization theorem for solutions to the Poisson problem]
	\label{thm:localization_Poisson}
	Let $1<p\leq2$,
	and $\Omega \subset \R^{n+1}$ be a corkscrew domain with {$n$-Ahlfors regular} boundary $\pom$ such that the problem $(\wt{PR}^{\mathcal L}_p)$ is solvable.  
	Let $x_0 \in \pom$, $0<R<\diam\pom$, $B= B(x_0,R)$, $H\in L^\infty_c(\Omega)$, and $\boldsymbol{\Xi}\in L^\infty_c(\Omega;\,\R^{n+1})$  with both $H$ and $\boldsymbol \Xi$ supported in $\Omega \cap B(x_0, 2R)^c$.
	Then,
	\[
	\fint_{B(x_0,R/2)\cap\partial\Omega} \wt N_{R/2}(\nabla u) \,d\sigma \lesssim 
	\fint_{A(x_0,R,2R)\cap\Omega} |\nabla u|\, dm  
	\]
	{where  ${A(x_0, R, 2R)=B(x_0,2R)\backslash B(x_0,R)}$,  $\wt N_{R/2}$ is the modified non-tangential maximal operator truncated at height $R/2$, and $u$ is the solution of the Poisson problem \eqref{eq:definition_Poisson_pb} for the operator $\mathcal L$ with Poisson data $H$ and $\boldsymbol{\Xi}$.}
\end{theorem}
Note that solvability of $(\widetilde{{PR}}^{\mathcal L}_{q})$ implies that the elliptic measure $\omega_{\mathcal L^*}$ is in weak-$\mathcal A_\infty(\sigma)$ (by Theorem \ref{thm:mod_reg_implies_mod_Dirichlet} and Proposition \ref{prop:dirichlet_tent_implies_weak_A_infty}).
This theorem has the same proof as the first part of the {Localization Theorem} \ref{thm:localization} (see also the proof of Theorem \ref{thm:extrapolation_1}), without any modification needed, as $u$ solves $\mathcal Lu = 0$ in $B(x_0, 2R)\cap\Omega$ and $u\equiv0$ on $B(x_0,2R)\cap\pom$.

\vvv

Now, the proof of Theorem \ref{thm:extrapolation_Poisson} is very similar to the proof of {Theorem \ref{thm:extrapolation}} as we only need to show 
\[
\int_\pom |\wt N(\nabla u)|^r \, d\sigma \lesssim 1
\] 
for $r \in (1-\epsilon, 1]$, and $u$ solution to \eqref{eq:definition_Poisson_pb} with Poisson data $\boldsymbol \Xi$ and $H$ such that $\delta_\Omega H$ and $\boldsymbol\Xi$ are $T^r_2$ atoms with support in the same ball $B$. As it is the case in the proof of {Theorem \ref{thm:extrapolation}}, we decompose
\[
\int_{\pom} |\wt N(\nabla u)|^r\, d\sigma = \int_{8B \cap \pom} |\wt N(\nabla u)|^r\, d\sigma + \sum_{k \geq 4} \int_{A_k \cap \pom} |\wt N(\nabla u)|^r\, d\sigma
\] 
where $A_k$ are the annuli $2^k B\backslash 2^{k-1}B$. The first term in the right hand side is bounded using that $(\wt{PR}^{\mathcal L}_p)$ is solvable and H\"older's inequality, and the other terms are controlled using the Localization Theorem \ref{thm:localization_Poisson}.
{We leave checking the details to the interested reader.}

\color{black}

\end{document}